\documentclass{article}

\usepackage{enumerate}
\usepackage[utf8]{inputenc}
\usepackage[T1]{fontenc}
\usepackage{epigraph}
\usepackage[centertags]{amsmath}
\usepackage{amssymb}
\usepackage{amsthm}
\usepackage{mathrsfs}
\usepackage{ifthen}
\usepackage[dvipsnames]{xcolor}
\usepackage{hyperref}

\usepackage{todonotes}

\newcommand{\ov}{\overline}

\def\cA{{\mathcal A}}
\def\cAI{\widetilde{{\mathcal A}}}

\def\cC{{\mathcal C}}
\def\cD{{\mathcal D}}
\def\cE{{\mathcal E}}
\def\cF{{\mathcal F}}
\def\cG{{\mathcal G}}
\def\cH{{\mathcal H}}
\def\cI{{\mathcal I}}
\def\cJ{{\mathcal J}}
\def\cK{{\mathcal K}}

\def\cM{{\mathcal M}}

\def\cO{{\mathcal O}}
\def\cP{{\mathcal P}}

\def\cR{{\mathcal R}}
\def\cS{{\mathcal S}}
\def\cT{{\mathcal T}}

\def\cV{{\mathcal V}}
\def\cW{{\mathcal W}}
\def\cX{{\mathcal X}}

\def\sH{{\mathscr H}}

\def\ga{{\mathfrak a}}
\def\gb{{\mathfrak b}}
\def\gc{{\mathfrak c}}
\def\gd{{\mathfrak d}}
\def\goe{{\mathfrak e}} 
\def\gI{{\mathfrak I}}
\def\gJ{{\mathfrak J}}
\def\gM{{\mathfrak M}}
\def\gP{{\mathfrak P}}
\def\gp{{\mathfrak p}}

\def\gu{{\mathfrak u}}

\def\zero{{\mathbf 0}}
\def\a{{\mathbf a}}
\def\b{{\mathbf b}}
\def\c{{\mathbf c}}
\def\bd{{\mathbf d}}
\def\be{{\mathbf e}}
\def\bl{{\mathbf l}}
\def\u{{\mathbf u}}
\def\ut{\tilde{\mathbf u}}
\def\v{{\mathbf v}}
\def\w{{\mathbf w}}
\def\y{{\mathbf y}}
\def\z{{\mathbf z}}

\def\d  {{\mathrm{\,d\,}}}
 
\def\vol {{\mathrm{\,vol\,}}}
\def\det {{\mathrm{\,det\,}}}

\def\Disc{{\mathrm{\,Disc\,}}}

\newcommand{\C}{\mathbb{C}}
\newcommand{\F}{\mathbb{F}}
\newcommand{\N}{\mathbb{N}}
\newcommand{\Q}{\mathbb{Q}}

\newcommand{\R}{\mathbb{R}}
\newcommand{\Z}{\mathbb{Z}}

\newcommand{\e}{\mathrm{e}}

\newcommand{\x}{\mathbf{x}}

\newtheorem{thrm}{Theorem}[section]
\newtheorem{lmm}[thrm]{Lemma}
\newtheorem{prpstn}[thrm]{Proposition}

\newtheorem*{dfntn}{Definition}
\newtheorem*{rmk}{Remark}

\numberwithin{sblmm}{thrm} 
\numberwithin{equation}{section}

\newcommand{\Mod}[1]{\ (\mathrm{mod}\ #1)}
\renewcommand{\mod}[1]{\mathrm{mod}\ #1}

\newcommand{\ZZ}{\mathbb{Z}}

\newcommand{\RR}{\mathbb{R}}
\newcommand{\Rc}{\mathcal{R}}
\renewcommand{\aa}{\mathbf{a}}
\newcommand{\bb}{\mathbf{b}}
\newcommand{\cc}{\mathbf{c}}

\newcommand{\xx}{\mathbf{x}}
\newcommand{\yy}{\mathbf{y}}

\newcommand{\zz}{\mathbf{z}}

\newcommand{\vv}{\mathbf{v}}
\newcommand{\ww}{\mathbf{w}}

\newcommand{\1}{\mathbf{1}}
\newcommand{\pf}{\mathfrak{p}}

\newcommand{\Sf}{\mathfrak{S}}
\newcommand{\Sft}{\tilde{\Sf}}
\newcommand{\Oc}{\mathcal{O}}
\newcommand{\af}{\mathfrak{a}}
\newcommand{\bfr}{\mathfrak{b}}

\newcommand{\Cc}{\mathcal{C}}
\newcommand{\Bc}{\mathcal{B}}

\begin{document}
%
%
%
%
\title{On the largest prime factor of  quartic polynomial values: the cyclic and dihedral cases}
\author{Cécile Dartyge and James Maynard} 

\date{}

\maketitle
%
%
%
%
\abstract{Let $P(X)\in\Z [X]$ be an irreducible, monic, quartic polynomial with cyclic or dihedral Galois group. We prove that there exists a constant $c_P>0$ such that for a positive proportion of integers $n$, $P(n)$ has a prime factor $\ge n^{1+c_P}$.}
%
%
%
%
\tableofcontents
%
%
%
%
\section{Introduction}
%
%
%
%
Let $P(X)\in\Z[X]$ be an irreducible degree  polynomial with $d\ge 2$. Assuming that there is no local obstruction, it is widely believed \cite{SchinzelSierpinski} that $P$ should take on infinitely many prime values, but unfortunately this conjecture remains completely open for all non-linear polynomials $P$.

As an approximation to this problem, one can look for integers $n$ for which $P(n)$ has a large prime factor. For general polynomials $P$, the best known bound is due to Tenenbaum \cite{Tenenbaum}, who shows that there are infinitely many integers $n$ such that $P(n)$ has a prime factor of size at least $n \exp((\log{n})^\alpha)$ for any $\alpha<2-\log{4}$. When the degree of $P$ is 5 or more, this is the best known result, but for some low degree polynomials, one can produce bounds which are much stronger.

 Hooley \cite{Hooley} proved the first result of this kind, showing that the largest prime factor $P^+(n^2+1)$ of $n^2+1$ satisfies $P^+(n^2+1)>n^{11/10}$ infinitely often. The exponent $11/10$ has been improved by Deshouillers and Iwaniec \cite{DI82}, next by La Bretèche and Drappeau \cite{BD20} and  the current record due to Merikoski \cite{Merikoski} is that $P^+(n^2+1)>n^{1.279}$ infinitely often. Heath-Brown \cite{HB01} showed that $P^+(n^3+2)>n^{1+10^{-303}}$ infinitely often.
 Irving \cite{Ir15} proved fifteen years later that exponent $1+10^{-303}$ can be replaced by $1+10^{-52}$.
 It seems plausible that the underlying methods could be adapted to more general degree 2 or degree 3 polynomials.

For degree 4 polynomials, results can currently only be obtained when the Galois group $G$ of $P(X)$ takes a simple form. When $P(X)=X^4-X^2+1$, the twelfth cyclotomic polynomial, Dartyge \cite{D15} proved that there are infinitely many 
 $n$ such that  $P^+(n^4-n^2+1)>n^{1+10^{-26531}}$.
 La Bretèche \cite{B15} generalised this result to quartic irreducible even monic polynomials with Galois group isomorphic to the Klein group $V:=\Z /2\Z\times \Z/2\Z$.  For such polynomials $P$, he proved that there exists  $c_P>0$ such that  $P^+(P(n))>n^{1+c_P}$ for a positive proportion of integers $n$. It seems plausible that the methods of  \cite{B15} and \cite{D15} may be adapted for some more general quartic polynomials, but the condition that the Galois group is $V$ is crucial to the method.

In this work we obtain results for irreducible quartic polynomials with  Galois group isomorphic to the cyclic group   $C_4:=\Z/4 \Z$ or the  dihedral group  $D_4=\Z/2\Z\ltimes \Z/4\Z$. Our method  doesn't work for polynomials with Galois group $A_4$ or $S_4$ which are the most frequent Galois groups for quartic irreducible polynomials.
However, the fifth cyclotomic polynomial $\Phi_5 (X)=X^4+X^3+X^2+X+1$, $X^4-5X^2+5$, $X^4+13X+39$ are   examples of polynomials with cyclic Galois group and $X^4+2$, $X^4+3X+3$, $X^4-5X^2+3$ are polynomials with Galois group $D_4$. (cf. \cite{C} for other examples of quartic polynomials with dihedral or cyclic Galois group). 
%
%
%
%
\begin{thrm}\label{cP}Let $P(X)$ be a monic quartic irreducible polynomial with Galois group $C_4$ or $D_4$. Then there exists a constant $c_P>0$ such that for $x>x_0 (P)$, 
we have 
\begin{equation*}
|\{ x<n\le 2x : P^+ (P(n)) \ge x^{1+c_P}\} |\gg x.
\end{equation*}
\end{thrm}
%
%
%
%

The key new technical innovation behind our proof of Theorem \ref{cP} is to incorporate `Type II' or `bilinear' information into the method of detecting large prime factors; previous approaches had relied solely on `Type I' information. This Type II information allows us to handle polynomials with Galois groups $C_4$ or $D_4$ which were out of reach of the Type I approach. In principle one could hope to handle the remaining possibilities $A_4$ or $S_4$ to cover all Galois groups by a similar procedure, but we do not know how to handle the relevant Type II estimates in this case, and so our paper is limited to $C_4$ and $D_4$.
Following the approaches of Heath-Brown \cite{HB01}, Dartyge \cite{D15} and La Bret\`eche \cite{B15}, the key to obtaining estimates like Theorem \ref{cP} is showing that a certain multivariate polynomial $q$ associated to $P(X)$ has a convenient prime factorisation for a positive proportion of its values. 

For quartic $P(X)$, this associated polynomial $q=q(a_1,a_2,a_3)$ is a ternary sextic form. 
If $P$ has a Galois group $V$, then $q(a_1,a_2,a_3)=q_1(a_1,a_2,a_3)q_2(a_1,a_2,a_3)q_3(a_1,a_2,a_3)$ is a product of 3 ternary quadratic forms, and the methods of \cite{D15} and \cite{B15} could then produce many suitable prime factorisations by showing equidistribution of $q_1$ and $q_2$ in suitable arithmetic progressions\footnote{Similarly, in the work of Heath-Brown \cite{HB01} dealing with cubic $P(X)$, the associated form $q$ is a binary cubic, and it suffices to just obtain distribution estimates for $q$ in arithmetic progressions}. (This is why we refer to their methods as `Type I' methods.) When $P$ has a larger Galois group, then the form $q(a_1,a_2,a_3)$ is the product of a quartic and a quadratic (if $G=C_4$ or $D_4$) or is an irreducible sextic (if $G=A_4$ or $S_4$). Unfortunately one cannot obtain a suitable factorisation by just considering analogous equidistribution in arithmetic progressions in these cases, since one would need to work with moduli which are too large for equidistribution to occur.

We find that if $G=C_4$ or $D_4$, the ternary quartic factor of $q$ has the additional algebraic structure of being an `incomplete norm form'. Maynard \cite{May15b} produced various Type II estimates which were used to count prime values of incomplete norm forms. By adapting the ideas underlying these estimates to our situation we are able to show that $q$ has a convenient prime factorisation for a positive proportion of its values. This part corresponds to Theorem \ref{DistriNorm} announced in Section \ref{sec:DistriNorm}.

Combining this result with the previous machinery (suitably generalised to our situation) then yields Theorem \ref{cP}.
%
%
%
%
\subsection{Outline of the proof of Theorem \ref{cP} }\label{outlines}
%
%
%
%

The proof of Theorem \ref{cP} takes three key steps. Step 1 is an argument due to Heath-Brown \cite{HB01} (see also \cite{Er52}), which reduces the problem to showing the existence of many integers where $P(n)$ has an unusually large friable part (i.e. a part without large prime factor). 

Step 2 follows and generalises \cite{HB01,D15,B15} and shows that by using the $q$-analogue of Van der Corput's method, it suffices to show that a certain ternary form $q(a_1,a_2,a_3)$ associated to $P$ takes many values with a suitable prime factorisation. This step makes use of the fact that $P$ is a quartic polynomial. The key new ingredient in our work is Step 3, where we establish that $q(a_1,a_2,a_3)$ takes on many values with the suitable prime factorisation when $P$ has Galois group $C_4$ or $D_4$. For this final step we incorporate ideas of Maynard \cite{May15b} on prime values of incomplete norm forms.

\textbf{Step 1: Reduction to many integers with large friable part. }

Let $r_1\in\overline{\Q}$ be a root of $P(n)$, $K=\Q(r_1)$ and $N_P=N_{K/\Q}$ the associated norm. Then we see that $N_P(n-r_1)=P(n)$, and so we are interested in counting integers $n$ such that the ideal $(n-r_1)$ has a prime ideal factor of large norm. In particular,
\[
\sum_{\substack{n\in[x,2x]\\ P^+(P(n))\ge x^{1+\eta}}}1=\sum_{\substack{n\in[x,2x]\\ \exists \pf|(n-r_1):\,N_P(\pf)\ge x^{1+\eta}}}1\gg \frac{1}{\log{x}}\sum_{n\in[x,2x]}\sum_{\substack{\pf^e|(n-r_1)\\ N_P(\pf)\ge x^{1+\eta} }}\log N_P(\pf).
\]
By inclusion-exclusion and the fact that $\sum_{\pf^e|(n-r_1)}\log{\pf}=\log P(n)$, we have that the double sum on the right hand side is given by
\[
\sum_{n\in[x,2x]}\log P(n)-\sum_{n\in[x,2x]}\sum_{\substack{\pf^e|(n-r_1)\\ N_P(\pf)\le 2x}}\log N_P(\pf)-\sum_{n\in[x,2x]}\sum_{\substack{\pf^e|(n-r_1)\\ 2x<N_P(\pf)< x^{1+\eta} }}\log N_P(\pf).
\]
Since $P(n)\asymp n^4$, the first sum is $(4+o(1))x\log{x}$. Swapping the order of summation and applying the Prime Ideal Theorem shows that the second sum is $(1+o(1))x\log{x}$. Let $\cA$ be the set of integers $n$ with $\sum_{\pf|(n-r_1),N_P(\pf)\le 2x}\log{N_P(\pf)}\ge (1+\delta_0)\log{x}$. We split the third sum according to whether $n\in\cA$ or not. Therefore the above expression is
\begin{align*}
(3+o(1))x\log{x}-\!\sum_{\substack{n\in[x,2x]\\ n\in \cA}}\sum_{\substack{\pf^e|(n-r_1)\\ 2x<N_P(\pf)< x^{1+\eta} }}\!\!\!\!\!\!\!\!\log N_P(\pf)-\sum_{\substack{n\in[x,2x]\\ n\notin \cA}}\sum_{\substack{\pf^e|(n-r_1)\\ 2x<N_P(\pf)< x^{1+\eta} }}\!\!\!\!\!\!\!\!\log N_P(\pf).
\end{align*}
If $n\in \cA$ then since prime ideals with $N_P(\pf)\le 2x$ contribute at least $(1+\delta_0)\log{x}$ to $\sum_{\pf|(n-r_1)}\log{N_P(\pf)}=(4+o(1))\log{x}$, the contribution from prime ideals with $N_P(\pf)>2x$ must be $\le(3-\delta_0-o(1))\log{x}$. If $n\notin \cA$ then we note from size considerations there can be at most 3 prime ideals with $N_P(\pf)\ge 2x$ dividing $(n-r_1)$, and so the inner sum over $\pf$ is at most $3(1+\eta)\log{x}$. Substituting these bounds into the above, we find
\[
\sum_{n\in[x,2x]}\sum_{\substack{\pf^e|(n-r_1)\\ N_P(\pf)\ge x^{1+\eta} }}\log N_P(\pf)\ge \delta_0\#\cA\log{x}-(3\eta+o(1)) x \log{x}.
\]
In particular, if $\#\cA\gg x$ then choosing $\eta=\delta_0\#\cA/(4x)$ shows that the left hand side is $\gg x\log{x}$. Thus it suffices to show
\[
\#\Bigl\{n\in[x,2x]:\, \prod_{\substack{\pf^e|(n-r_1)\\ N_P(\pf)\le x}}N_P(\pf)\ge x^{1+\delta_0}\Bigr\}\gg x.
\]

\textbf{Step 2: Reduction to values of a polynomial with convenient factorisation. }

By concentrating on multiples of friable principle ideals $\gJ=(a_0+a_1r_1+a_2r_1^2+a_3r_1^3)$ of norm $\asymp x^{1+\delta_0}$, where $r_1$ is a root of $P$, we find it suffices to show there is some dense set $\cA\subset \ZZ^4\cap[1,x^{(1+\delta_0)/4}]$ such that
\[
\sum_{\substack{(a_0,a_1,a_2,a_3)\in\cA}}\sum_{\substack{n\in [x,2x]\\ (a_0+a_1r_1+a_2r_1^2+a_3r_1^3)|(n-r_1)}}1\gg x.
\]
The condition $(a_0+a_1r_1+a_2r_1^2+a_3r_1 ^3)|(n-r_1)$ is equivalent to a congruence condition $n\equiv k_{\aa}\Mod{N_P(a_0+a_1r_1+a_2r_1^2+a_3r_1^3)}$, and so by completion of sums and swapping the order of summation, it suffices to obtain a power-saving in the exponential sums (for small integers $h\ne 0$ and with the standard notation $\e (t)=\exp (2i\pi t)$)
\[
\sum_{a_0,a_1,a_2,a_3\in \cA}\e\Bigl(\frac{h k_{a_0,a_1,a_2,a_3}}{N_P(a_0+a_1r_1+a_2r_1 ^2+a_3r_1 ^3)}\Bigr).
\]
This is complicated by the fact that the variables $a_0,a_1,a_2,a_3$ appear in both the numerator and denominator. However, for quartic $P$ we find that there are polynomials $B_{14}(a_0,a_1,a_2,a_3), B_{13}(a_0,a_1,a_2,a_3)$ and $q(a_1,a_2,a_3)$ with no common factor such that
\[
\e\Bigl(\frac{h k_{a_0,a_1,a_2,a_3}}{N_P(a_0+a_1r_1+a_2r_2+a_3r_3)}\Bigr)\approx \e\Bigl(\frac{hB_{13}(a_0,a_1,a_2,a_3)\overline{B_{14}(a_0,a_1,a_2,a_3)}}{q(a_1,a_2,a_3)}\Bigr),
\]
and now the denominator is independent of $a_0$. We wish to obtain a power-saving estimate for the sum over $a_0$, but this is complicated by the fact that the modulus of the expression $q(a_1,a_2,a_3)\asymp x^{6(1+\delta_0)/4}$ is much larger than the length $x^{(1+\delta_0)/4}$ of summation of $a_0$. To estimate such short exponential sums, we can use the $q$-analogue of Van der Corput's method provided the modulus $q(a_1,a_2,a_3)$ consists only of small prime factors.

Thus our task has reduced to showing that for a positive proportion of integers $a_1,a_2,a_3\in [1,x^{(1+\delta_0)/4}]$ we can ensure that the polynomial $q(a_1,a_2,a_3)$ has a convenient prime factorisation. Specifically, we will require that 
\begin{equation}\label{factoring}
q (a_1,a_2,a_3)=d_0d_1\cdots d_r\end{equation}
where $d_0< x^{2-\varepsilon}$, $\max (d_1,\ldots ,d_r)\le x^{1-\varepsilon}$, $\min (d_0,\ldots ,d_r)\ge x^\varepsilon$ for some fixed $\varepsilon >0$.

\textbf{Step 3: Counting factorisations of incomplete norm forms}

So far we have followed a similar approach to the works \cite{D15,B15}. If the Galois group of $P$ is the Klein group, then it turns out that the polynomial $q(a_1,a_2,a_3)$ is the product of three quadratic polynomials. By considering the distribution in suitable residue classes one can then guarantee that each quadratic has a suitable factor, and so $q(a_1,a_2,a_3)$ then has a suitable prime factorisation.

When the Galois group of $P$ is $C_4$ or $D_4$, it turns out that $q(a_1,a_2,a_3)=q_1(a_1,a_2,a_3)q_2(a_1,a_2,a_3)$ is the product of a quartic polynomial and a quadratic polynomial. Unfortunately the fact that one factor is quartic means that one cannot guarantee a suitable prime factorisation by looking at variables in residue classes to reasonably small moduli. The difficulty here is that $q_1(a_1,a_2,a_3)\approx (\max_i a_i)^4$, so the size of the values considered are very  large compared to the size of the variables $a_i$. Indeed, it is not known that an arbitrary ternary quartic form $q_1$ takes infinitely many values compatible with the factorisation \eqref{factoring}.

Fortunately in our problem the form $q_1$ is not arbitrary, and in fact we can show that $q_1$ corresponds to an incomplete norm form of a number field.  More precisely, we prove that there exist a number field $K$ of degree 4 over $\Q$ depending only on $P$ and some elements   
 $\nu_1,\nu_2,\nu_3\in K$ such that $q_1(a_1,a_2,a_3)=N_{K/\Q} (\sum_{i=1}^3a_i\nu_i)$.

Maynard \cite{May15b} gave asymptotic formulae for the number of primes represented by incomplete norm forms; that is
primes  $p$ such that $p=N(a_1+a_2\omega +\cdots +a_{n-k}\omega^{n-k-1})$ where $a_1,\ldots ,a_{n-k}$  are integers, $\omega$ is a root of monic and irreducible polynomial $f\in\Z [X]$ of degree $n\ge 4k$ and $N$ is a norm of the corresponding number field. For $k=1$ and $n=4$ this result counts values quartic norms in $3$ variables with a particular type of prime factorisation. We adapt the methods of \cite{May15b} to our situation to count representations of the type \eqref{factoring}. Unfortunately we require various additional technical conditions (such as a localized version of Maynard's estimates where the variables lie in suitable arithmetic progressions), which means that large parts of \cite{May15b} have to be generalised to our specific situation.  
Once suitable technical estimates have been obtained, we find \eqref{factoring} is satisfied for a positive proportion of $a_1,a_2,a_3$, as required.
%
%
%
%
\section{Acknowledgements}
%
%
CD was supported by ANR grant ANR-20-CE91-0006.

JM is supported by a Royal Society Wolfson Merit Award, and this project has received funding from the European Research Council (ERC) under the European Union’s Horizon 2020 research and innovation programme (grant agreement No 851318). Part of this work was conducted while JM was visiting the Institute for Advanced Study in Princeton.
%
%
%
%
\section{Initial steps}\label{Section:InitialSteps}
%
%
%
%
Following the argument of Heath-Brown sketched as `step 1' in our outline, we have the following result.
%
%
%
%
\begin{lmm}\label{HBP+}
Let $P\in\Z [X]$ be an irreducible quartic and monic polynomial of degree $4$ with root $r_1$, and let
\begin{equation}\label{defcE}
\cE(\delta):=\{ X<n\le 2X : \prod_{\substack{\pf^e|(n-r_1)\\ N_P(\pf)\le x}}N_P(\pf)\ge X^{1+\delta}\}.
\end{equation}
If $\delta_0, \delta_1>0$ are such that for all $X$ large enough in terms of $\delta_0,\delta_1, P$ we have $|\cE (\delta_0)|>\delta _1X$, then we have for sufficiently large $X$
\begin{equation*}
|\{ n\in]X,2X] : P^+ (P(n)) \gg X^{1+\frac{\delta_0\delta_1}{3}}\}|\ge (\delta_1\delta_0^2+o(1))X.
\end{equation*}
\end{lmm}
%
%
\begin{proof}
This is essentially \cite[Lemma 2]{HB01}, (or \cite[Lemme 4.1]{B15}) after noting that $\sum_{p|P(n),p\le z}\log{p}\ge\sum_{\pf|(n-r_1),N_P(\pf)\le z}\log{N_P(\pf)}$.
\end{proof}
%
%
Thus it suffices to show that $|\cE(\delta_0)|\gg X$ for some small absolute constant $\delta_0>0$. To do this we will choose a set of ideals $\cJ$ (the explicit, technical choice is made in Section \ref{sec:Ideal}) such that 
\begin{align}
\prod_{\substack{\pf^e|\gJ\\ N_P(\pf)\le X}}N_P(\pf)&\ge X^{1+\delta_0} \qquad\qquad\forall \gJ\in\cJ.\label{eq:JProp1}
\end{align}
Let $\cJ_2:=\{\gJ\in\cJ:\,P^-(N_P(\gJ))\ge X^{\theta_0}\}$ for some small absolute constant $\theta_0>0$. Then we see that for any $n\in[X,2X]$ there are at most $2^{4\theta_0^{-1}}$ ideals $\gJ\in\cJ_2$ with $\gJ|(n-r_1)$, since $(n-r_1)$ can have at most $4\theta^{-1}$ prime ideal factors with norm bigger than $X^{\theta_0}$.
We then see that
\begin{align*}
|\cE(\delta_0)|&\ge |\{ X<n\le 2X : \exists\ \gJ\in\cJ_2\ {\rm such}\ {\rm that}\ \gJ| (n-r_1)\}|\\
&\ge \frac{\sum_{\gJ\in\cJ_2}|\cE_\gJ|}{\sup_{X\le n\le 2X}|\{\gJ\in\cJ_2:\,\gJ|(n-r)\}|}\\
&\gg \sum_{\substack{\gJ\in\cJ\\ P^-(N_P(\gJ))\ge X^{\theta_0}}}|\cE_\gJ|,
\end{align*}
where  $\cE_\gJ:=\{ X<n\le 2X : \gJ| (n-r_1)\}$. Every ideal $\gJ$ has at most $\alpha_0^{-1}$ representations as $\gJ=KL$ for $K$ a prime ideal with $N_P(K)\in[X^{4\alpha_0},X^{5\alpha_0}]$. Thus we see that 
\[
|\cE(\delta_0)|\gg \sum_{\substack{K \in\cK}}\sum_{\substack{P^-(N_P(L))\ge X^{\theta_0}\\ KL\in\cJ}}|\cE_{KL}|,
\]
where
\begin{equation}
\cK:=\Bigl\{K\text{ prime ideal,}\  N_P(K)\in [X^{4\alpha_0},X^{5\alpha_0}]\Bigr\}.
\label{eq:KDef}
\end{equation}
We apply a linear sieve of level $X^{3\theta_0}$ to bound the condition $P^-(N_P(L))\ge X^{\theta_0}$ from below, giving
\[
|\cE(\delta_0)|\gg\sum_{K\in\cK}\sum_{KL\in\cJ}\Biggl(\sum_{d|N_P(L)} \lambda_d^-\Biggr )|\cE_{KL}|
\]
where $\lambda_d^-$ are the usual Rosser-Iwaniec lower bound linear sieve weights (\cite{I80a} and \cite{I80b}) supported on $d<X^{3\theta_0}$ with $p|d\Rightarrow p\le X^{\theta_0}$. We see that if $X$ is large enough $\cE_\gJ$ has density $\rho_P(N_P(\gJ))/N_P(\gJ)$, where
\begin{equation}
\varrho_P (\gI):={\rm card}
\{0\le n< N_P(\gI): n\equiv r_1\Mod \gI\}.
\label{eq:RhoPDef}
\end{equation}
With this in mind, we define the error $R_\gJ$ in the approximation by
\begin{equation}\label{defRJ}
R_\gJ:=|\cE_\gJ|-X\frac{\varrho_P (N_P(\gJ))}{N_P(\gJ)}.
\end{equation}
Thus
\begin{equation*}
|\cE _1|\gg XS_0+S_1,
\end{equation*}
where
\begin{equation}\label{defS0S1}\begin{split}
S_0&:=\sum_{K\in\cK}\sum_{KL\in\cJ (K)}\Bigg (\sum_{d|N_P(L)}\lambda_d^-\Bigg )\frac{\varrho _P (KL)}{N_P(KL)},\\
 S_1&:=\sum_{K\in\cK}\sum_{KL\in\cJ (K)}\Bigg (\sum_{d|N_P(L)}\lambda_d^-\Bigg )R_{KL}.
 \end{split}
\end{equation}
To obtain Theorem \ref{cP} we see it suffices to prove the following two key propositions. 
%
%
\begin{prpstn}[Estimate for $S_0$]\label{prpstn:S0}
Let $\theta_0$ be sufficiently small, and $\cJ$ be the set of ideals described in Section \ref{sec:Ideal}. Then we have
\[
S_0\gg 1.
\]
\end{prpstn}
%
%
\begin{prpstn}[Estimate for $S_1$]\label{prpstn:S1}
Let $\theta_0$ be sufficiently small, and  $\cJ$ be the set of ideals described in Section \ref{sec:Ideal}. Then we have
\[
S_1=o(X).
\]
\end{prpstn}
%
%
Together these propositions rely heavily on our key technical result, Theorem \ref{DistriNorm}. Section \ref{sec:S1} is devoted to establishing Proposition \ref{prpstn:S1}, which uses the fact that $\cJ$ is a set of ideals with small prime factors to bound the relevant exponential sums. Section \ref{sec:S0} is devoted to establishing Proposition \ref{prpstn:S0} assuming Theorem \ref{DistriNorm}. The rest of the paper is then devoted to establishing Theorem \ref{DistriNorm}, which asserts that $\cJ$ is a  set of nonzero density. 
%
%
%
%
\section{Localised divisors of values of incomplete norm forms}\label{sec:DistriNorm} 
%
%
%
%
As described in the outline, the key to the proof of Theorem \ref{cP} is to show that for a positive proportion of $a_1,a_2,a_3$ (in a box like $[A,2A]^3$) an auxiliary polynomial $q(a_1,a_2,a_3)=q_1(a_1,a_2,a_3)q_2(a_1,a_2,a_3)$ takes values where $P^+(q_2(a_1,a_2,a_3))<A^{2-\epsilon}$ and $P^+(q_1(a_1,a_2,a_3))<A^{1-\epsilon}$. The term $q_2$ will be a quadratic form, and so $P^+(q_2(a_1,a_2,a_3))<A^{2-\epsilon}$ if $p|q_2(a_1,a_2,a_3)$ for some $p\in [A^{2\epsilon},A^{3\epsilon}]$, which occurs if $a_1,a_2,a_3$ lie in suitable residue classes $\Mod{p}$. Thus it suffices to show that there are the expected number of $(a_1,a_2,a_3)$ such that $P^+(q_1(a_1,a_2,a_3))<A^{1-\epsilon}$ and $(a_1,a_2,a_3)$ lies in a suitable residue class modulo $p$ on average over $p\in[A^{2\epsilon},A^{3\epsilon}]$. Since $q_1$ will be an incomplete norm form for a quartic extension, we see that we are therefore counting friable values of an incomplete norm form (with some additional congruence constraints). The aim of this section is to introduce the notation to state Theorem \ref{DistriNorm}, and then to explain how this technical statement relates to our specific problem by giving a suitable asymptotic for such friable values of auxiliary forms.

 Let $K$ be a quartic extension of $\Q$ with a $\Z$-basis $\{\nu_1,\nu_2,\nu_3,\nu_4\}$ for $\Oc_K$ such that $\nu_1=1$ and $K=\Q (\nu_2)$. Given a large value $X$, we wish to count the number of $(a_1,a_2,a_3)$ in a small box such that $N_{K/\Q}(a_1\nu_1+a_2\nu_2+a_3\nu_3)$ has only small prime factors, and such that an auxiliary quadratic form $f(a_1,a_2,a_3)$ is a multiple of some fairly small $p\in [X^{\tau},X^{\tau'}]$.
 
With this in mind, we consider the box $\cX$ given by
\begin{align}
\cX&:=\prod_{i=1}^3[X_i,X_i(1+\eta_1)[,
\label{eq:cXDef}
\end{align}
where $\eta_1\in\R$ and $X_1,X_2,X_3\in\Z$ are parameters satisfying
\begin{align}
\eta_1&:=(\log{X})^{-100},
\label{eq:etaDef}\\
X_1,X_2,X_3&\in[\eta_1X,X],
\label{eq:Xi1}\\
N_{K/\Q}(X_1\nu_1+X_2\nu_2+X_3\nu_3)&\ge \eta_1^{1/10}\max_i(X_i^4).
\label{eq:Xi2}
\end{align}
We are then interested in the sets
\begin{equation}\label{setN}
\begin{split}
\cA&:=\{(a_1,a_2,a_3)\in \Z^3\cap\cX\},\\
\cA (\u_0,m,p)&:=\{ (a_1,a_2,a_3)\in\cA : \a\equiv \u_0\Mod m,\ p|f(a_1,a_2,a_3)\},\\
\cA_d (\u_0, m,p)&:=\{ (a_1,a_2,a_3)\in  \cA (\u_0,m,p): d|N_{K/\Q} (a_1\nu_1+a_2\nu_2+a_3\nu_3)\}.
\end{split}
\end{equation}
Since we wish to count points when $N_{K/\Q} (a_1\nu_1+a_2\nu_2+a_3\nu_3)$ has small prime factors, we will count how often $d|N_{K/\Q} (a_1\nu_1+a_2\nu_2+a_3\nu_3)$ for an integer $d$ of the form $d=q_1\cdots q_\ell$ where each $q_i$ is a prime localised to lie in an interval $[X^{\theta_{j}},X^{\theta_{j}'}]$ for some fixed constants $\theta_i,\theta_i'$. We will require $\theta_{j},\theta_{j}'$ satisfy the following conditions.
\begin{itemize}
\item (Non-trivial intervals counting primes which are not too large)
\begin{equation}
\delta<\theta_i<\theta_i'<1-\delta\quad\forall\, 1\le i\le \ell.
\label{eq:th1}
\end{equation}
\item ($q_{1j}$ are distinct primes)
\begin{equation}
[\theta_i,\theta_i']\cap[\theta_j,\theta_j']=\emptyset \quad\forall \,1\le i<j\le \ell .
\label{eq:th2}
\end{equation}
\item ($\prod_{j=1}^\ell q_{1j}$ is not too large to divide $N$)
\begin{equation}
\sum_{i=1}^\ell\theta_i'<4-\delta.
\label{eq:th3}
\end{equation}
\item (Impossible for $q_{1j}^2$ to divide $N(a_1\nu_1+a_2\nu_2+a_3\nu_3)$)
\begin{equation}
\theta_j+\sum_{i=1}^\ell\theta_i>4+\delta\quad \forall\,1\le j\le \ell.
\label{eq:th4}
\end{equation}
\item (The product of the first $q_{1i}$ is of controlled size) There exists $\ell '\in [1,\ell-1]$ such that
\begin{equation}
1+\delta<\sum_{i=1}^{\ell '}\theta_i<\sum_{i=1}^{\ell '}\theta_i'<2-\delta.
\label{eq:th5}
\end{equation}
\end{itemize}
The conditions \eqref{eq:th1}-\eqref{eq:th4} are minor constraints to avoid some technical issues and to ensure that we expect that $d|N_{K/\Q}(a_1\nu_1+a_2\nu_2+a_3\nu_3)$ can actually occur; these constraints could be significantly weakened at the cost of some effort. The condition \eqref{eq:th5} is a technical condition which is vital for our method.

To avoid some further technical issues we will focus on the case when the quadratic form $f$ is irreducible but not geometrically irreducible, and so the condition $f(a_1,a_2,a_3)$ becomes a product of two linear factors over $\F_p$ after restricting $p$ to an arithmetic progression. Again, this setup could be relaxed at the cost of additional technical effort, but is the situation that arises when dealing with Theorem \ref{cP}.  It would be also interesting to have a more general result for incomplete norm forms and ternary forms $f$.

Finally we are in a position to state our counting result. 
%
%
%
%
\begin{thrm}[localised factors of values of incomplete norm forms]\label{DistriNorm}
Let $f(X_1,X_2,X_3)\in\Z[X_1,X_2,X_3]$ be a homogeneous quadratic polynomial which splits into two distinct linear factors 
$$f(X_1,X_2,X_3)=L_1(X_1,X_2,X_3)L_2(X_1,X_2,X_3)$$  over a suitable extension of $\Q$.
Let $D_f\in\N$ such that if $p\equiv 1\Mod{D_f}$ then the $\F_p$-reduction of the two linear forms $L_1 (X_1,X_2,X_3)$, $L_2 (X_1,X_2,X_3)$ are in $\F_p [X_1,X_2,X_3]$. 

 Let $K$ be a quartic extension of $\Q$ with $\{\nu_1,\nu_2,\nu_3,\nu_4\}$ being a $\Z$-basis for $\Oc_K$ such that $\nu_1=1$ and $K=\Q (\nu_2)$. Let $X_1,X_2,X_3$ satisfy \eqref{eq:Xi1} and \eqref{eq:Xi2}. Let $\ell,\ell'\in\N$ such that $1\le\ell '<\ell$ and 
 $\theta_1,\theta_1',\dots,\theta_\ell,\theta_\ell '$ be reals satisfying \eqref{eq:th1}-\eqref{eq:th5}. Let $0<\tau<\tau'$ satisfy 
\begin{equation}
\label{eq:tau'}
\tau'<\min\Bigl(\frac{4-2\theta'_{1}-\ldots-2\theta'_{\ell '}}{100},\frac{\theta_1+\cdots+\theta_{\ell'}-1}{100}\Bigr).
\end{equation}
Let $\cA_d(\u_,m,p)$ be as given by \eqref{setN}.

Then for any choice of $\u_0\Mod{m}$ and $A>0$, we have
\begin{align*}
\sum_{\substack{p\in [X^{\tau},X^{\tau'}]\\ p\equiv 1\Mod{D_f} }}&\sum_{\substack{ q_{1},\ldots,q_{\ell}\text{ prime}\\ q_{j}\in [X^{\theta_{j}},X^{\theta_{j}'}]\,\forall 1\le j \le \ell}}|\cA_{q_{1}\cdots q_{\ell}}(\u_0,m,p)|\\
&=\eta_1^3X_1X_2X_3\frac{2\log(\frac{\tau'}{\tau})}{m^3\varphi (D_f)}
\prod_{i=1}^{\ell}\log \Bigl (\frac {\theta_i'}{\theta_i}\Bigr)+O\Bigl(\frac{ X_1X_2X_3}{(\log X)^A}\Bigr).
\end{align*}
The implied constant depends on $f,K,A,\delta$ and the $\theta_i,\theta_i'$ only.
\end{thrm}
%
%
%
%
At first sight Theorem \ref{DistriNorm} looks like a Type I estimate since we are counting $a_1,a_2,a_3$ such that $N_{K/\Q}(a_1\nu_1+a_2\nu_2+a_3\nu_3)$ is a multiple of $q_1\ldots q_\ell$. 
However, since there are typically no values of $a_1,a_2,a_3$ such that this occurs (it is only a thin set of $q_j$'s when there is a solution), we instead are required to view this as a Type II estimate counting $N_{K/\Q}(a_1\nu_1+a_2\nu_2+a_3\nu_3)=m_1m_2$ where $m_1=q_1\cdots q_{\ell '}$ is a product of $\ell'$ primes of constrained size and $m_2=q_{\ell '+1}\cdots q_\ell r$ is the product of $\ell -\ell '$ primes and some other integer $r$.
%
%
%
%
\subsection{Application to Theorem \ref{cP}}
%
%
%
%
If $P$ is an irreducible monic quartic polynomial, then (generalising previous works) there is an auxilliary sextic form $q(a_1,a_2,a_3)$ such that provided $q$ takes suitably friable values a positive proportion of the time, then we can use exponential sum methods to establish Theorem \ref{cP}. If $P$ has Galois group $C_4$ or $D_4$, then it turns out that the roots $r_1,r_2,r_3,r_4$ of $P$ can be ordered such that $r_1r_2+r_3r_4\in \Q$ (c.f. Lemma \ref{lmm:C4D4}), and that $q$ factorises as $q_1q_2$ for a quartic form $q_1$ and a quadratic form $q_2$ (c.f. Lemma \ref{lmm:q}) which split completely in the splitting field of $P$.

Moreover, we find that for the quartic extension $K:=\Q(r_1+r_3)$ of $\Q$, the form $q_1$ satisfies
\begin{equation*}
    q_1 (a_1,a_2,a_3)=\pm N_{K/\Q}(a_1+a_2 (r_1+r_3)+a_3 (r_1^2+r_1r_3+r_3^2)),
\end{equation*}
and so takes the shape of an incomplete norm form (c.f. Proposition \ref{Nq1}).

The quadratic $q_2$ takes the form
\begin{equation}\label{formf}\begin{split}
q_2(a_1,a_2,a_3)&=[a_1+(r_1+r_2)a_2+(r_1^2+r_1r_2+r_2^2)a_3]\\
&\times [a_1+(r_3+r_4)a_2+(r_3^2+r_3r_4+r_4^2)a_3].
\end{split}
\end{equation}
Since the two polynomials $P_1(X):=(X-(r_1+r_2))(X-(r_3+r_4))$ and $P_2(X):=(X-(r_1^2+r_1r_2+r_2^2))(X-(r_3^2+r_3r_4+r_4^2))$ are in $\Z [X]$, $r_1+r_2$ and $r_1^2+r_1r_2+r_2^2$ are of degree at most $2$ over $\Q$. Let $\Delta_1$ and $\Delta _2$ be the discriminant of these two polynomials
and 
\begin{equation}
    D_{q_2}:= 
    \begin{cases}
    [8,\Delta_1 ,\Delta_2]
& \text{if}\ \Delta_1\Delta_2\not =0,\\
[8,\Delta_1+\Delta_2] & \text{otherwise}.
\end{cases}
\label{eq:DqDef}
\end{equation}
Since $P$ is irreducible of degree $4$, we don't have $\Delta_1 =\Delta_2 =0$.\footnote{If $\Delta_1=\Delta_2=0$ then the roots of $r_1+r_2$ and $r_1r_2$ are in $\Q$. This contradicts the fact that 
$[\Q(r_1):\Q]=4$.} If $p\equiv 1\Mod{D_{q_2}}$ and $\Delta_1\Delta_2\not =0$, then $(\Delta_1/p)=(\Delta _2/p)=1$  where  $(n/p)$ is the Legendre symbol. Thus the polynomials $P_1$ and $P_2$ modulo $p$ factor into products of two degree one polynomials. 
The linear factors of $q_2$ in \eqref{formf} have their coefficients in $\F_p$.  We also verify that it is still the case when $p\equiv 1\Mod{D_{q_2}}$  and $\Delta_1\Delta_2=0$.

Then $N_{K/\Q}(\sum_{i=1}^4 a_i\nu_i)$ is a quartic form in the integer variables $a_1,a_2,a_3,a_4$,  and we have for all $a_1,a_2,a_3,a_4\in\Z$
$$N_{K/\Q}\Big (\sum_{i=1}^4 a_i\nu_i\Big )=\prod_{i=1}^4 \Big (\sum_{j=1}^4 a_j\sigma _i (\nu_j)\Big ),$$
where $\sigma_1,\sigma_2,\sigma_3 ,\sigma _4$ are the different embeddings of $K/\Q$.

Given an irreducible quartic polynomial $P\in\Z [X]$ with Galois group $C_4$ or $D_4$ it is the case (see Lemma \ref{lmm:C4D4}) that the distinct roots $r_1,r_2,r_3,r_4$ of $P$ can be ordered 
such that $r_1r_2+r_3r_4\in\Q$. We are interested in the auxiliary polynomial $q_2$ (see \eqref{defq2}), given by

To ensure that $q_1(a_1,a_2,a_3)=N_{K/\Q}(a_1\nu_1+a_2\nu_2+a_3\nu_3)$ is composed only of suitably small prime factors, we will look for $a_1,a_2,a_3$ such that
\[
q_{11}q_{12}q_{13}q_{14}\ldots q_{1\ell}|N_{K/\Q}(a_1\nu_1+a_2\nu_2+a_3\nu_3)
\]
for some suitable primes $q_{11},q_{12},q_{13},q_{14},\ldots ,q_{1\ell}<X^{1-\delta}$ with $\prod_{j=1}^\ell q_{1j}>X^{3+\delta}$. 
In the application to Theorem \ref{cP}, we will only need the case $\ell =6$, but the proof in this particular case is exactly the same as in the general case.
%
%
%
%
\section{Algebraic properties of auxilliary polynomials}
%
%
%
%
\subsection{Ideals}
%
%
 Let $r_1$ be a root of $P$.  We define for any ideal
 $\gI$  of  $\Z [r_1]$ the function
\begin{equation*}
\varrho_P (\gI)={\rm card}
\{0\le n< N_P(\gI): n\equiv r_1\Mod \gI\},
\end{equation*}
where $N_P=N_{\Q (r_1)/\Q}$ is the norm on $\Q (r_1)$.
If $\gI$ is principal, $\gI=(\alpha)$, we will write simply $\varrho_P (\alpha)$ in place of $\varrho_P ((\alpha))$. 
%
%
%
%
\begin{lmm}\label{ideaux}
Let $\gI$ be an  ideal of $\Oc_{\Q(r_1)}$ such that  $(N_P(\gI),\Disc (P))=1$. If the equation  $n\equiv r_1\Mod \gI$ has a  solution with $n\in\Z$ 
then  $\gI$ is a product of prime ideals whose norm is a prime number.  Furthermore $\gI$ can't be divisible by two different prime ideals with the same norm.  Conversely,  if $\gI$ satisfies these  different conditions then this  congruence admits some solutions and $\varrho_P (\gI)=1$. 
Finally if  $\gI$ is an   ideal such that   $\varrho_P (\gI)=1$ then for  $m\in\Z$, $\gI|m\Leftrightarrow N_P(\gI)|m$. 
\end{lmm}
%
%
\begin{proof}
This is \cite[Lemma 3.1]{B15}.
The particular case $P=\Phi_{12}$ is handled in \cite[Lemma 3.1]{D15}. 
\end{proof}
%
%
%
%
\subsection{The roots of $P$ modulo $m$}
%
%
%
%
In this part only we suppose that  $P(X)=X^n+c_{n-1}X^{n-1}+\dots+c_0\in\Z[X]$ is monic, irreducible of degree $n$.
In our problem, the degree of $P$ is $4$ but the argument of this part is valid for all irreducible and monic polynomials and might be used in other contexts. Throughout the rest of the paper we fix a root $r_1$ of $P$.

For $\alpha\in\Z [r_1]$,  we write  $\alpha=a_0+a_1 r_1+a_2r_1^2+a_3r_1^3+\cdots + a_{n-1}r_1^{n-1}$.
Let $m_\alpha : \Q (r_1)\rightarrow \Q (r_1)$  be the multiplication-by-$\alpha$ map: $m_\alpha (x)=\alpha x$.
Let $M_\alpha$ be the  matrix of $m_\alpha$ with respect to the basis  $\{1,r_1,r_1^2,r_1^3, \ldots , r_1^{n-1}\}$ and  $N_P(\alpha )=N_{\Q(r_1)/\Q}(\alpha )$ its determinant.
For $P(X)=X^4+2$ the corresponding matrix is
\begin{equation*}
\left (\begin{matrix}
a_0 & -2a_3 & -2a_2 & -2a_1\\ a_1 & a_0 & -2a_3 & -2a_2\\ a_2 & a_1 & a_0 & -2a_3\\ a_3& a_2 & a_1 & a_0\\
\end{matrix}
\right ). 
\end{equation*} 
More generally since $r_1^n= -c_0-c_1r_1-\cdots - c_{n-1}r_1^{n-1}$, we have
\begin{equation}\label{malpha}
M_\alpha=\begin{pmatrix}
a_0 & -c_0a_{n-1} & * & \cdots & *\\
a_1 & a_0-c_1a_{n-1} & * & \cdots & *\\
\vdots & \vdots & \vdots &  \vdots & \vdots \\
a_{n-1} & a_{n-2}-c_{n-1}a_{n-1} & * & \cdots & *\\
\end{pmatrix}.
\end{equation}

In this section we prove results analogous to \cite[Lemma 4.1]{D15} or \cite[Lemma 3.2]{B15}.
As in these two papers, we let $B_{ij}=B_{ij}(\alpha)$ be the cofactor formed by taking the determinant of the $(n-1)\times(n-1)$ matrix formed by removing line $i$ and column $j$ from $M_\alpha$ and multiply it by $(-1)^{i+j}$. If $\alpha=a_0+a_1r_1+\dots+a_{n-1}r_1^{n-1}$ then $B_{ij}$ is a polynomial in the $a_i$. By an abuse of notation we will sometimes use $B_{ij}$ to refer to this polynomial, and sometimes the value attained at a particular point $(a_0,a_1,\dots,a_{n-1})$. The intended usage should be clear from the context.
%
%
%
%
\begin{lmm}\label{r}
Let $\alpha =a_0+a_1r_1+\cdots +a_{n-1}r_1^{n-1}$, with $a_0,\ldots ,a_{n-1}\in\Z$ be such that  $(N_P(\alpha), B_{1n})=1$. Then there exists an integer  $k_\alpha$, with $0\le k_\alpha<N_P(\alpha )$
such that we have
\begin{equation*}
n-r_1\equiv 0\Mod{(\alpha)}\Leftrightarrow n\equiv k_\alpha \Mod{N_P(\alpha )}.
\end{equation*}
This integer $k_\alpha$ satisfies the  congruence
\begin{equation*}
k_\alpha\equiv B_{2n}\overline{B_{1n}}\Mod{N_P(\alpha )}.
\end{equation*}
Furthermore, if  $\gJ$ is an   ideal of  $\Z [r_1]$ containing a  principal ideal  $(\alpha )$ with $\alpha$ as above then there exists a unique   $k_\gJ$
with $0\le k_\gJ<N_P(\gJ)$  and 
\begin{equation*}
n-r_1\in \gJ\Leftrightarrow n\equiv k_\gJ\Mod{N_P(\gJ )}.
\end{equation*}
\end{lmm}
%
%
\begin{proof}
The starting point is the following trivial observation: $\alpha r_1^j\in (\alpha)$ for all $j=0,1,2,3,\ldots ,  n-1$. Let $(m_{i,j})_{1\le i,j\le n}$ be the  coefficients of $M_\alpha$. We obtain the  equations
\begin{equation*}
m_{1,j}+m_{2,j}r_1+\cdots +m_{n,j}r_1^{n-1}=0 \Mod {(\alpha)},\,\forall\, 1\le j\le n.
\end{equation*}
This system can be represented as
\begin{equation}\label{relmat}
\left (
\begin{matrix}
m_{2,1} & m_{3,1} & \ldots & m_{n,1}\\ m_{2,2} & m_{3,2} & \ldots & m_{n,2}\\
\vdots &\vdots &\ldots  &\vdots\\
m_{2,n}& m_{3,n}& \ldots & m_{n,n}\\
\end{matrix}\right )
\left (
\begin{matrix}
r_1\\ r_1^2\\ \vdots \\ r_1^{n-1}
\end{matrix}
\right )
=\left ( \begin{matrix} -m_{1,1}\\ -m_{1,2}\\ \vdots\\ -m_{1,n}
\end{matrix}\right )\Mod{(\alpha )}
\end{equation}
If we remove the i-th line in this system and apply Cramer's rule, we find
\begin{equation}\label{reli}
r_1\det \left (
\begin{matrix}
m_{2,1} & m_{3,1} & \ldots & m_{n,1}\\ \vdots &\vdots &\ldots  &\vdots \\ m_{2,i-1} & m_{3,i-1} & \ldots & m_{n,i-1}\\m_{2,i+1} & m_{3,i+1} & \ldots & m_{n,i+1}\\
\vdots &\vdots &\ldots  &\vdots\\
m_{2,n-1}& m_{3,n-1}& \ldots & m_{n,n-1}\\
m_{2,n}& m_{3,n}& \ldots & m_{n,n}\\
\end{matrix}\right )=\det \left (
\begin{matrix}
-m_{1,1} & m_{3,1} & \ldots & m_{n,1}\\ \vdots &\vdots &\ldots  &\vdots\\ -m_{1,i-1} & m_{3,i-1} & \ldots & m_{n,i-1}\\ -m_{1,i+1} & m_{3,i+1} & \ldots & m_{n,i+1}\\
\vdots &\vdots &\ldots  &\vdots\\
-m_{1,n-1}& m_{3,n-1}& \ldots & m_{n,n-1}\\
-m_{1,n}& m_{3,n}& \ldots & m_{n,n}\\
\end{matrix}\right )\Mod{(\alpha)}.
\end{equation}
The transpose of the matrix on the left is the submatrix of $M_\alpha$ obtained by removing the first line and the $i^{th}$ column. The matrix on the right is the submatrix of $M_\alpha$ obtained by removing the second line and the $i^{th}$ column and by multiplying all elements of the first column by $-1$. 

We recall that the   $B_{ij}$, $1\le i,j\le n$, are   the cofactors of $M_\alpha$, so that 
 \begin{equation}\label{cofactor}
 M^{-1}_\alpha =\frac{1}{N_P(\alpha)}\left (\begin{matrix}
 B_{11}& B_{21} & \ldots & B_{n1}\\
 B_{12} & B_{22} & \ldots & B_{n2}\\
 \vdots & \vdots & \ldots & \vdots\\
 B_{1n} & B_{2n} & \ldots & B_{nn}\end{matrix}\right ).
 \end{equation}
 With this notation, \eqref{reli} becomes
\begin{equation*}
(-1)^{i+1}B_{1i}r_1\equiv  -(-1)^{i+2} B_{2i}\Mod{(\alpha )}.
\end{equation*}
In particular, this gives
\begin{equation}\label{r11}
B_{1i}r_1\equiv B_{2i}\Mod{(\alpha )}.
\end{equation}
By Lemma \ref{ideaux}, if an integer is congruent to $0\Mod{(\alpha)}$ then it is divisible by $N_P(\alpha)$. Therefore considering $i=n$ now gives the claim of the first part of Lemma \ref{r}. 

For the second part when $J| (\alpha)$, thus it suffices to take $k_J\in [0,N_P(J)]$ such that $k_J\equiv{k_\alpha}\Mod{N_P(J)}$. The claim now follows from \eqref{r11}.
  \end{proof}
%
%
%
%
 \medskip
 
 We end this subsection by observing some connection between the cofactors $B_{1i}$ and $B_{2j}$
 with $1\le i,j\le n$.
Since $(m_\alpha )^{-1}=m_{\alpha ^{-1}}$, we have 
\begin{equation*}
\alpha ^{-1} = \frac{1}{N_P(\alpha)}(B_{11}+B_{12}r_1+\cdots +B_{1n}r_1^{n-1}),
\end{equation*}
and the columns of  $M_\alpha^{-1}$ satisfy the same   relations \eqref{malpha} as the one in  $M_\alpha$. By the  relations \eqref{malpha} for $M_{\alpha^{-1}}$, we see that
\begin{equation}\label{B2iB1j}
\left (\begin{matrix} B_{21}\\ B_{22}\\ \vdots\\ B_{2(n-1)}\\ B_{2n}
\end{matrix}\right )=
\left (\begin{matrix} -c_0B_{1n}\\ B_{11}-c_1B_{1n}\\ \vdots \\ B_{1(n-2)}-c_{n-2}B_{1n}\\ B_{1(n-1)}-c_{n-1}B_{1n}\end{matrix}\right ).
\end{equation}
In particular the last line implies that   
\begin{equation}
B_{2n}=B_{1(n-1)}-c_{n-1}B_{1n}.
\label{eq:CofactorRelation}
\end{equation} 
For $n=4$, and $c_3=0$, we recover the  formula $B_{14}r_1\equiv B_{24}=B_{13}\Mod{(\alpha )}$,
 proved in  \cite{B15} and in \cite{D15}.
%
%
%
 %
\subsection{Elimination of $a_0$}
%
%

The aim of this subsection is to approximate the fraction $k_\gJ/N_P(\alpha )$ by a fraction whose denominator depends 
only on $a_1,a_2,a_3$. Now and for the rest of this paper we restrict our attention to $P$ having degree $4$. 
In this subsection we prove the analogue of  \cite[Lemma 3.3]{B15}, or \cite[Lemma 6.2]{D15}. A natural way to proceed is to work with some   resultants of the different forms defined previously.

%
%
\begin{lmm}\label{lmm:q3indep}
There is a homogeneous polynomial $q_3=q_3(a_1,a_2,a_3)$ in $a_1,a_2,a_3$ such that
 \begin{equation}\label{BBN}
B_{24}B_{13}-B_{14}B_{23}=q_3 N_P(\alpha ).
\end{equation}
\end{lmm}
%
%
\begin{proof}
Applying \eqref{r11} with $i=3,4$, $n=4$ we find 
\begin{equation*}
    B_{13}B_{24}\equiv B_{14}B_{23}\Mod{N_P(\alpha)}.
\end{equation*}
Since this holds for all $a_0,a_1,a_2,a_3$, we deduce that there exists a form $q_3=q_3(a_0,a_1,a_2,a_{3})$ such that \footnote{In \cite{B15} and \cite{D15} this form corresponds to the form $q_4$.}
 \begin{equation}
B_{24}B_{13}-B_{14}B_{23}=q_3 N_P(\alpha ).
\end{equation}
Therefore we just need to show that $q_3$ actually doesn't depend on $a_0$. $N_P(\alpha )$ has degree $4$ in $a_0$ while the polynomials $B_{ij}$, $i\not =j$ are of degree $2$ in $a_0$, and so by equating the coefficients of $a_0^4$ we see that $q_3$ must not depend on $a_0$.
\end{proof}
%
%
\begin{rmk}
One can explicitly compute $q_3$ in terms of the coefficients $c_i$ of $P$; it is given by
\begin{equation}\label{deff}
    q_3(a_1,a_2,a_3)=a_2^2-a_1a_3-c_3a_2a_3+c_2a_3^2.
\end{equation}
When $c_3=0$ this coincides with the form $-q_4$
given in \cite[equation (2.7)]{B15}.
\end{rmk}
%
%

\medskip

%
%

\begin{rmk} 
Lemma \ref{lmm:q3indep} makes important use of the fact that $P$ is a quartic polynomial. For polynomials $P$ of  degree $d>4$ the form $q_3$ would have degree $d-4$ in $a_0$, and so would no longer independent of $a_0$.
\end{rmk}

 %
%

\medskip

 %
%

Following the notation of  \cite{B15} and \cite{D15}, we write ${\rm Resultant} (P_1,P_2;x)$ for the resultant of the polynomials $P_1$, $P_2$ 
with respect to the  variable $x$.
We will be interested by the two following  resultants
\begin{equation}\label{resultant}\begin{split}
R&:=R(a_1,a_2,a_3)={\rm Resultant} (B_{14},N_P(\alpha );a_0)\\
R_0&:=R_0(a_1,a_2,a_3)={\rm Resultant} (B_{13},B_{14};a_0)\\
\end{split}
\end{equation}
 %
%
\begin{lmm}\label{RR0}
With the previous notation we have 
\begin{equation*}
q_3^2 R=R_0^2.
\end{equation*}
\end{lmm}
 %
%
\begin{proof}
The proof of Lemma \ref{RR0} is the same as that of \cite[Lemma 2.1]{B15}.
Since  $B_{14}$ is of degree $2$ in $a_0$, we have
\begin{equation*}
q_3^2R={\rm Resultant} (B_{14},q_3 N_P(\alpha );a_0)={\rm Resultant} (B_{14},B_{24}B_{13}-B_{14}B_{23} ;a_0).
\end{equation*}
But $B_{24}=B_{13}-c_3B_{14}$ and $B_{13}$ is also of degree $2$ in $a_0$. We deduce that  
\begin{equation*}
q_3^2R={\rm Resultant} (B_{14},B_{13}^2;a_0)=R_0^2.
\end{equation*}
This ends the proof of   Lemma \ref{RR0}.
\end{proof}
 %
%
We see that the polynomial $q_3$ divides $R_0$, and so we can write 
\begin{equation}
R_0=qq_3
\label{eq:qDef}
\end{equation}
for some homogeneous polynomial $q=q(a_1,a_2,a_3)$. Moreover, since $R_0$ is the resultant of $B_{13}$ and $B_{14}$, there are two  polynomials $U$ and $V\in\Z [a_0,a_1,a_2,a_3]$ such that
\begin{equation}\label{UV}
UB_{13}+VB_{14}=qq_3.
\end{equation}
We are now ready to state the main result of this section. It is analogous to \cite[Lemma 6.2]{D15} or \cite[Lemma 3.3]{B15}.
 %
%
\begin{lmm}\label{kj}Suppose $a_0,a_1,a_2,a_3$ are such that $(B_{14}(a_0,a_1,a_2,a_3),q(a_1,a_2,a_3))=1$. Then
$(N_P (\alpha ), B_{14} (a_0,a_1,a_2,a_3))=1$
and for  $h\in\Z$ we have
\begin{equation*}
\e\Big ( \frac{-hk_{\alpha} }{N_P(\alpha )}\Big )=\e\Big ( \frac{-hU(a_0,a_1,a_2,a_3)\overline{B_{14}(a_0,a_1,a_2,a_3)} }{q(a_1,a_2,a_3)}+hR(a_0,a_1,a_2,a_3)\Big ),
\end{equation*}
where $U=U(a_0,a_1,a_2,a_3)$ is defined by \eqref{UV} and $R$ is given by
\begin{equation*}
R(a_0,a_1,a_2,a_3)=\frac{U}{qB_{14}}-\frac{B_{24}}{N_P(\alpha )B_{14}}.
\end{equation*}
\end{lmm}
 %
%
\begin{proof}
To simplify notation, for the proof let $q,q_3,U,B_{14},B_{14},B_{23},B_{24},N_P(\alpha)$ denote the values of the polynomials evaluated at $a_0,a_1,a_2,a_3$. 

Since $q$ divides the resultant $R$ defined in \eqref{resultant}, if $q$ is coprime with $B_{14}$, we have $(N_P (\alpha), B_{14})=1$. By Lemma \ref{r},
\begin{equation*}
\e\Big (\frac{k_{\alpha}}{N_P(\alpha)}\Big )=\e\Big (\frac{B_{24}\overline{B_{14}}}{N_P(\alpha )}\Big ).
\end{equation*}
We use the Bézout relation  
\begin{equation}\label{Bezout}
\frac{\bar u}{v}+\frac{\bar v}{u}\equiv\frac{1}{uv}\Mod 1\quad {\rm for}\ (u,v)=1,
\end{equation}
and the fact that  $(N_P(\alpha ), B_{14})=1$.  This yields the formula
\begin{equation}\label{kj1}
\e\Big (\frac{k_\alpha}{N_P(\alpha)}\Big )=\e\Big (-\frac{B_{24}\overline{N_P(\alpha)}}{B_{14}}+\frac{B_{24}}{B_{14}N_P(\alpha)}\Big ).
\end{equation}
Combining \eqref{eq:CofactorRelation}, \eqref{BBN} and \eqref{UV}, we obtain
\begin{equation*} \begin{split}
UN_P(\alpha )q_3&=U[B_{13}(B_{13}-c_3B_{14})-B_{14}B_{23}]\\
& =U(B_{13}^2-B_{14}(B_{23}+c_3B_{13}))\\
&=B_{13}(q_3q-VB_{14})-UB_{14}(B_{23}+c_3B_{13})).\\
\end{split}
\end{equation*}
This rearranges to give
\begin{equation*}
(UN_P(\alpha) -qB_{13})q_3=B_{14}(-VB_{13}-U(B_{23}+c_3B_{13})).
\end{equation*}
Since  $q_3$ and $B_{14}$ are coprime, we  deduce that
\begin{equation}\label{Congruence}
UN_P(\alpha) -qB_{13}\equiv 0\Mod{B_{14}}.
\end{equation}
Since $B_{24}\equiv B_{13}\Mod {B_{14}}$, we obtain
\begin{equation*}
B_{24}\ov{N_P(\alpha)}\equiv B_{13}\ov{N_P(\alpha )} \Mod {B_{14}}\equiv U\bar q\Mod {B_{14}}.
\end{equation*}
We insert this  in \eqref{kj1} and apply  \eqref{Bezout}  one more time. This gives the desired result.
\end{proof}
 %
%

%
%
\subsection{Explicit computations of $B_{13},B_{14},U,V$.}
%
%

We have used  SAGE to explicitly compute the polynomials $q$, $B_{13}$, $B_{14}$, $U$ and $V$.  The cofactors $B_{13}$ and $B_{14}$ are of degree $2$ in $a_0$
\begin{align}
    B_{13}=& -a_2a_0^2 + \Bigl(a_1^2 + c_3a_1a_2 + (-c_3^2 + c_2)a_2^2 + (-2c_2)a_1a_3 \nonumber\\
    &\qquad + (c_3^3 - c_2c_3 + c_1)a_2a_3 + (-c_2c_3^2 + c_2^2 + c_1c_3 - c_0)a_3^2\Bigr)a_0 \nonumber\\
    &+ (-c_3)a_1^3 + c_3^2a_1^2a_2 + (-c_2c_3)a_1a_2^2 + (c_1c_3 - c_0)a_2^3 \nonumber\\ 
    &+ (-c_3^3 + 2c_2c_3)a_1^2a_3 + (c_2c_3^2 - 3c_1c_3 + 2c_0)a_1a_2a_3\nonumber\\
    &+ (-c_1c_3^2 + 2c_0c_3)a_2^2a_3 + (-c_2^2c_3 + 2c_1c_3^2 - 2c_0c_3)a_1a_3^2\nonumber \\ 
    &+ (c_1c_2c_3 - c_0c_3^2 - c_0c_2)a_2a_3^2 + (-c_1^2c_3 + c_0c_2c_3 + c_0c_1)a_3^3, \label{eq:B13}\\
    B_{14}=& -a_3a_0^2 + \Bigl(2a_1a_2 -c_3a_2^2 -c_3a_1a_3 + c_3^2a_2a_3 + (-c_2c_3 + 2c_1)a_3^2\Bigr)a_0\nonumber\\
    &- a_1^3 + c_3a_1^2a_2 -c_2a_1a_2^2 + c_1a_2^3 + (-c_3^2 + 2c_2)a_1^2a_3 + (c_2c_3 - 3c_1)a_1a_2a_3\nonumber \\
    & + (-c_1c_3 + c_0)a_2^2a_3 + (-c_2^2 + 2c_1c_3 - c_0)a_1a_3^2\nonumber\\ 
    & +(c_1c_2 - c_0c_3)a_2a_3^2 + (-c_1^2 + c_0c_2)a_3^3. \label{eq:B14}
    \end{align}
    The quantities $U$ and $V$ are of degree $1$ in $a_0$.
    In some step we will need the  explicit formula for the coefficient in $a_0$ in $U$ and in $V$
    \begin{align}
      U=&a_0\Bigl(-a_1^2a_3^2 + 2a_1a_2^2a_3  -2c_3a_1a_2a_3^2 + 2c_2a_1a_3^3 -c_3a_2^3a_3 +\nonumber\\
     & \qquad(2c_3^2 - c_2)a_2^2a_3^2 + (-c_3^3 + c_1)a_2a_3^3 + (c_2c_3^2 - c_2^2 - c_1c_3 + c_0)a_3^4\Bigr)\nonumber\\
    &+3a_1^3a_2a_3 -2c_3a_1^3a_3^2 -4a_1^2a_2^3 + 4c_3a_1^2a_2^2a_3 \nonumber\\
    &+ (c_3^2 - 6c_2)a_1^2a_2a_3^2 + (-c_3^3 + 3c_2c_3 + 2c_1)a_1^2a_3^3 \nonumber\\ 
    &+ 4c_3a_1a_2^4 + (-9c_3^2 + 3c_2)a_1a_2^3a_3 + (6c_3^3 + c_2c_3 - 3c_1)a_1a_2^2a_3^2 \nonumber\\
    & + (-c_3^4 - 5c_2c_3^2 + 3c_2^2 + 2c_1c_3 + c_0)a_1a_2a_3^3 + (c_2c_3^3 + c_1c_3^2 - 4c_1c_2 - c_0c_3)a_1a_3^4  -c_3^2a_2^5\nonumber\\ 
    & + (3c_3^3 - c_2c_3 - c_1)a_2^4a_3 + (-3c_3^4 + 5c_1c_3 - 2c_0)a_2^3a_3^2 \nonumber\\
    &+ (c_3^5 + 3c_2c_3^3 - 2c_2^2c_3 - 7c_1c_3^2 + c_1c_2 + 4c_0c_3)a_2^2a_3^3\nonumber \\ 
    &+ (-2c_2c_3^4 + c_2^2c_3^2 + 3c_1c_3^3 + 2c_1c_2c_3 - 2c_0c_3^2 - c_1^2 - 2c_0c_2)a_2a_3^4 \nonumber\\ 
    &+ (c_2^2c_3^3 - c_2^3c_3 - 3c_1c_2c_3^2 + 2c_1c_2^2 + c_1^2c_3 + 2c_0c_2c_3 - c_0c_1)a_3^5,\label{U}
    \end{align}
    \begin{equation}\label{V}
    \begin{split}
    V=&a_0\Bigl(a_1^2a_2a_3 -2a_1a_2^3 + 2c_3a_1a_2^2a_3 -2c_2a_1a_2a_3^2 + c_3a_2^4 \\
    &\qquad+ (-2c_3^2 + c_2)a_2^3a_3 + (c_3^3 - c_1)a_2^2a_3^2 + (-c_2c_3^2 + c_2^2 + c_1c_3 - c_0)a_2a_3^3\Bigr)\\
    &-a_1^4a_3 + a_1^3a_2^2 -2c_3a_1^3a_2a_3 + 4c_2a_1^3a_3^2 + 2c_3a_1^2a_2^3 + (-c_3^2 - 4c_2)a_1^2a_2^2a_3 \\
    &+ (-c_3^3 + 5c_2c_3)a_1^2a_2a_3^2 + (2c_2c_3^2 - 6c_2^2 - 2c_1c_3 + 2c_0)a_1^2a_3^3 + (-3c_3^2 + c_2)a_1a_2^4 \\ &+ (6c_3^3 - c_2c_3 - c_1)a_1a_2^3a_3 + (-3c_3^4 - 7c_2c_3^2 + 5c_2^2 + 6c_1c_3 - 5c_0)a_1a_2^2a_3^2 \\ 
    &+ (7c_2c_3^3 - 4c_2^2c_3 - 5c_1c_3^2 + 5c_0c_3)a_1a_2a_3^3 + (-4c_2^2c_3^2 + 4c_2^3 + 4c_1c_2c_3 - 4c_0c_2)a_1a_3^4 \\
    &+ (c_3^3 - c_2c_3 + c_1)a_2^5 + (-3c3^4 + 4c_2c_3^2 - c_2^2 - 3c_1c_3 + 2c_0)a_2^4a_3  \\ 
    &+(3c_3^5 - 3c_2c_3^3 + c_1c_3^2 + c_1c_2 - 2c_0c_3)a_2^3a_3^2 \\
    &+ (-c_3^6 - 2c_2c_3^4 + 5c_2^2c_3^2 + 3c_1c_3^3 - 2c_2^3 - 4c_1c_2c_3 - 2c_0c_3^2 + 4c_0c_2)a_2^2a_3^3 \\
    &+(2c_2c_3^5 - 3c_2^2c_3^3 - 2c_1c_3^4 + c_2^3c_3 + c_1c_2c_3^2 + 2c_0c_3^3 + c_1^2c_3 - 2c_0c_2c_3 - c_0c_1)a_2a_3^4 \\
    &+ (-c_2^2c_3^4 + 2c_2^3c_3^2 + 2c_1c_2c_3^3 - c_2^4 - 2c_1c_2^2c_3 - c_1^2c_3^2 - 2c_0c_2c_3^2 + 2c_0c_2^2 + 2c_0c_1c_3 - c_0^2)a_3^5.
    \end{split}
    \end{equation}
We don't write the expression for $q$ because it would take more than one page and we won't need to know its precise shape
during the proof.
Let $U=a_0U_1+U_0$, $V=a_0V_1+V_0$. Then $U_1$ satisfies:
\begin{equation}\label{defU1}\begin{split}
    U_1&=-a_1^2a_3^2 + 2a_1a_2^2a_3  -2c_3a_1a_2a_3^2 + 2c_2a_1a_3^3 -c_3a_2^3a_3 +\\ 
    & \qquad (2c_3^2 - c_2)a_2^2a_3^2 + (-c_3^3 + c_1)a_2a_3^3 + (c_2c_3^2 - c_2^2 - c_1c_3 + c_0)a_3^4\\
    &=a_3 \Bigl( -a_1^2a_3 + 2a_1a_2^2  -2c_3a_1a_2a_3 + 2c_2a_1a_3^2 -c_3a_2^3 +\\
     &\qquad  (2c_3^2 - c_2)a_2^2a_3 + (-c_3^3 + c_1)a_2a_3^2 + (c_2c_3^2 - c_2^2 - c_1c_3 + c_0)a_3^3\Bigr).
    \end{split}
\end{equation}
We observe that 
\begin{equation}\label{U1V1}
a_2U_1+a_3V_1=0.
\end{equation}
 %
%
\subsection{Factorisation of $q$}\label{factorisationq}
%
%
\begin{lmm}\label{dlBM}
Let $P\in\Z [X]$ be an irreducible monic quartic polynomial and $r_1,r_2,r_3,r_4$ its roots.
Let $R$ and $R_0$ be the two resultants introduced in \eqref{resultant}. Let $a(r):=a_0+a_1r+a_2r^2+a_3r^3$.
Then there exists $t_P\in\Q^*$ such that
\begin{equation*}
    R(a_1,a_2,a_3)=t_P\prod_{1\le i<j\le 4}(a(r_i)-a(r_j))^2.
\end{equation*}
Furthermore, the resultant $R_0$  is  divisible by 
\begin{equation*}
\prod_{1\le i<j\le 4}(a(r_i)-a(r_j)).
\end{equation*} 
\end{lmm}
%
%
\begin{proof}
This is \cite[Lemme 7.1]{B15} in the special case of quartic polynomials.
\end{proof}
%
%
\begin{lmm}\label{tP}
The coefficient $t_P$ in Lemma \ref{dlBM}
is given by
\begin{equation*}
    t_P=\prod_{1\le i<j\le 4}\frac{1}{(r_i-r_j)^2}.
\end{equation*}
\end{lmm}
 %
%
\begin{proof}
The proof follows the argument of La Bretèche and Mestre, but for completeness we repeat the main steps.

We note that $N_P(\alpha )$ is the determinant of the linear map $g_a : \Q [X]/P(X)\rightarrow \Q [X]/P(X)$ given by $g_a (H(X))=a(X)H(X)$ where $a(X)=a_0+a_1X+a_2X^2+a_3X^3$. Let $L_1 (X),\ldots , L_4 (X)$ be the Lagrange interpolation polynomials for the roots $r_1,\ldots ,r_4$ of $P$. Thus $L_i(x)=\prod_{j\ne i}(x-r_j)/(r_i-r_j)$ and in particular $L_i(r_j)=1$ if $i=j$, $0$ if $i\ne j$. Then  for all $i=1,2,3,4$,
 \begin{equation*}
     g_a (L_i (X))=a(X)L_i(X)=\sum_{j=1}^4
     a(r_j)L_j(X)L_i(X)=a(r_i)L_i(X),
 \end{equation*}
in $\Q [X]/(P)$, since $P(X)| L_i(X)L_j(X)$ if $i\not =j$ and $P(X)| (L_i^2(X)-L_i(X))$.
Thus the matrix of $g_a$ with respect to the basis $\{L_1(X),L_2(X), L_3 (X), L_4(X)\}$ is diagonal with coefficients $a(r_1), a(r_2), a(r_3), a(r_4)$ on the diagonal.

Let $T$ be the matrix of the polynomials 
$L_1 (X),L_2(X),L_3(X),L_4 (X)$ with respect to the standard basis $\{1,X,X^2,X^3\}$.
Then the matrix of $N_P (\alpha )g_a^{-1}$ with respect to the standard basis is $N_P (\alpha )M_\alpha ^{-1}$ with $M_{\alpha}^{-1}$ given by \eqref{cofactor}. Thus have
\begin{equation}\label{Bga}
    \begin{pmatrix}
    B_{11} & B_{21} & B_{31} & B_{41}\\
    B_{12} & B_{22} & B_{32} & B_{42}\\
    B_{13} & B_{23} & B_{33} & B_{34}\\
    B_{14} & B_{24} & B_{34} & B_{44}
    \end{pmatrix}
    = T \begin{pmatrix} \prod_{j\not =1} a(r_j) & 0 & 0 & 0\\ 0 &\prod_{j\not =2} a(r_j)
    & 0 & 0\\
    0 & 0 & \prod_{j\not =3} a(r_j) & 0\\
    0 & 0 & 0 & \prod_{j\not =4} a(r_j)
    \end{pmatrix}T^{-1}.
\end{equation}
The form $N_P(\alpha )=\prod_{i=1}^4 a(r_i)$ is quartic and monic in $a_0$. If we write $B_{14}=B_{14}(a_0)$ as 
an element of $\Z [a_1,a_2,a_3][a_0]$,
the resultant $R$ satisfies
\begin{equation*}
    R=\prod_{i=1}^4  B_{14}(d_i),
\end{equation*}
where $d_i = -a_1r_i-a_2r_i^2-a_3r_i^3$ for $i=1,2,3,4$
are the roots of $y\mapsto N_P (y+a_1r_1+a_2r_1^2+a_3r_1^3)$.
Let $P_i (X):= d_i + a_1X+a_2X^2+ a_3 X^3$. Formula \eqref{Bga} with $d_1$ in place of $a_0$,   gives
\begin{equation*}
 \begin{pmatrix}
    B_{11}(d_1) & B_{21}(d_1) & B_{31} (d_1) & B_{41} (d_1)\\
    B_{12}(d_1) & B_{22}(d_1) & B_{32}(d_1) & B_{42}(d_1)\\
    B_{13}(d_1) & B_{23}(d_1) & B_{33}(d_1) & B_{34}(d_1)\\
    B_{14}(d_1) & B_{24}(d_1) & B_{34}(d_1) & B_{44}(d_1)
    \end{pmatrix}   
    =T\begin{pmatrix}
    \prod_{\ell\not =1}P_1 (r_\ell) & 0 & 0 & 0\\
    0 & 0 & 0 & 0\\ 0 & 0 & 0 & 0\\
    0 & 0 & 0 & 0
    \end{pmatrix}T^{-1}.
\end{equation*}
We have similar formulas for the polynomials $P_2$, $P_3$, $P_4$. The first column
of the matrix  of the left corresponds to the coordinates in the standard basis of
the image of the constant polynomial $1$ by the map $N_P(\alpha )g_\alpha ^{-1}$. The  decomposition of the polynomial $1$ in the Lagrange basis is $1=L_1(X)+L_2(X)+L_3(X)+L_4(X)$.
The first column of the left matrix is then
\begin{equation*}
    \begin{pmatrix} B_{11}(d_1)\\
    B_{12} (d_1)\\ B_{13} (d_1)\\ B_{14}(d_1)
    \end{pmatrix} =T\begin{pmatrix}
    \prod_{\ell\not =1}P_1 (r_\ell) & 0 & 0 & 0\\
    0 & 0 & 0 & 0\\ 0 & 0 & 0 & 0\\
    0 & 0 & 0 & 0
    \end{pmatrix}
    \begin{pmatrix} 1 \\ 1\\ 1\\ 1
    \end{pmatrix}=T\begin{pmatrix}\prod_{j=2}^4 P_1 (r_j)\\ 0\\ 0\\ 0\end{pmatrix}.
\end{equation*}
In particular we deduce that $B_{14} (d_1)$
is the coefficient of $X^3$ in the polynomial
$\prod_{j=2}^4 P_1 (r_j) L_1 (X)$.
Since $L_1 (X)=\prod_{j=2}^4 (X-r_j)/(r_1-r_j)$, we get
\begin{equation*}
    B_{14} (d_1)=\frac{\prod_{j=2}^4 P_1 (r_j)}
    {\prod_{j=2}^4 (r_1-r_j)}=-
    \prod_{i=2}^4\frac{a(r_j)-a(r_1)}{r_j-r_1}.
\end{equation*}
In the same way we prove  for $i=2,3,4$ :
\begin{equation*}
    B_{14} (d_i)=\frac{\prod_{j\not =i} P_i (r_j)}
    {\prod_{j\not =i}^4 (r_i-r_j)}=-
    \prod_{j\not =i}\frac{a(r_j)-a(r_i)}{r_j-r_i}.
\end{equation*}
This completes the proof of Lemma \ref{tP}.
\end{proof}
 %
%
\begin{rmk} Lemma \ref{tP} is stated for quartic polynomials but is in fact also valid for irreducible polynomials
of degree $n\ge 2$. For these polynomials,
if the resultant is between $N_P(\alpha)$ and the cofactor $B_{1n}$, then
$t_P^{-1}=(-1)^n\prod_{1\le i<j\le n} (r_i-r_j)^2.$
For the resultant between $N_P(\alpha)$ and
$B_{1\ell}$ for some $1\le \ell\le n-1$,
we may also have for $t_P$ an explicit  but more complicate formula, involving
 the coefficients of $X^\ell$
in the Lagrange interpolation polynomials 
associated with the roots $r_1,\ldots ,r_n$ of $P(X)$.
\end{rmk}
%
%
\begin{lmm}\label{lmm:qFactorisation}
The polynomial $q(a_1,a_2,a_3)\in\Q[a_1,a_2,a_3]$ satisfies
\begin{equation*}
    q=\pm\prod_{1\le i<j\le 4}\frac{a(r_i)-a(r_j)}{r_i-r_j},
\end{equation*}
where $a(r):=a_0+a_1r+a_2r^2+a_3r^3$.
\end{lmm}
%
%
\begin{proof}
This follows immediately from putting together Lemmas \ref{dlBM},  \ref{RR0} and \ref{tP}.
\end{proof}
%
%

\subsection{The factor $q_1$ as a an incomplete norm form}

%
%
A key point in the work of  \cite{D15} and \cite{B15} is that the form $q$ may be factored as a product of 3 quadratic forms whenever $P$ has a suitably small Galois group. In this  section, we prove that if $G=C_4$ or $D_4$ then  $q$ is a product of two forms $q=q_1q_2$, where $q_1$ has degree $4$, $q_2$ has degree $2$ and $q_1$ is related to a  norm form of a certain number field.
%
%
\begin{lmm}\label{lmm:C4D4}
Let $P(X)$ be a monic quartic with Galois group $C_4$ or $D_4$. Then there is an ordering of the roots $r_1,r_2,r_3,r_4$ of $P$ such that
\[
r_1r_2+r_3r_4\in\Q.
\]
\end{lmm}
%
%
\begin{proof}
We recall  the notation $P(X)=X^4+c_3X^3+c_2X^2+c_1X+c_0$. The cubic resolvent of $P$ is 
\begin{align*}
R_3(X)&=(X-(r_1r_2+r_3r_4))(X-(r_1r_3+r_2r_4))(X-(r_1r_4+r_2r_3))\\
&=X^3-c_2X^2+(c_3c_1-4c_0)X-(c_3^2c_0+c_1^2-4c_2c_0).
\end{align*}
We therefore see that the claim of the lemma is equivalent to $R_3(X)$ having a root in $\Q$ when $P(X)$ has Galois group $G=C_4$ or $D_4$. This fact (often stated in the form that the splitting field of $R_3(X)$ is a degree 2 extension) is a standard fact about cubic resolvants; see for example the web page of K. Conrad 
\cite{C}  or the book of Jensen, Ledet and Yui \cite{JLY02} for some nice expositions on the Galois group of quartic polynomials.
\end{proof}
%
%
\begin{lmm}\label{lmm:q}
Let $P(X)$ have Galois group $C_4$ or $D_4$. Then the form $q\in \Q[a_1,a_2,a_3]$ has the factorisation
\[
q=q_1q_2
\]
where $q_1\in\Q[a_1,a_2,a_3]$ has degree 4 and $q_2\in\Q[a_1,a_2,a_3]$ has degree 2. These are explicitly given by 
\begin{equation}\label{defq1}
q_1=\frac{(a(r_1)-a(r_3))(a(r_1)-a(r_4))(a(r_2)-a(r_3))(a(r_2)-a(r_4))}{(r_1-r_3)(r_1-r_4)(r_2-r_3)(r_2-r_4)},
\end{equation}
and
\begin{equation}\label{defq2}
q_2=\frac{(a(r_1)-a(r_2))(a(r_3)-a(r_4))}{(r_1-r_2)(r_3-r_4)},
\end{equation}
where $a(X)=a_0+a_1X+a_2X^2+a_3X^3$ and $r_1,r_2,r_3,r_4$ are the roots of $P(X)$, ordered such that $r_1r_2+r_3r_4\in\Q$.
\end{lmm}
%
%
\begin{proof}
We recall from Lemma \ref{lmm:qFactorisation} that the explicit formulae \eqref{defq1} and \eqref{defq2} give a factorisation $q=q_1q_2$ over $\overline{\Q}$. Thus we wish to show that in fact $q_1,q_2\in\Q[a_1,a_2,a_3]$, so that this is also a factorisation over $\Q$. A direct computation gives for all $1\le i<j\le 4$:
\begin{equation}\label{rij}
\frac{a(r_i)-a(r_j)}{r_i-r_j}=a_1+a_2(r_i+r_j)+a_3(r_i^2+r_ir_j+r_j^2).
\end{equation}
If $G=C_4$ then $G=\langle \sigma\rangle$ where $\sigma$ is the permutation $\sigma =(r_1r_3r_2r_4)$.
(This choice of $\sigma$ is motivated by the fact that we must have $\sigma (r_1r_2+r_3r_4)=
r_1r_2+r_3r_4$.)
Since $\sigma (q_1)=q_1$ and $\sigma (q_2)=q_2$, we have that $q_1$ and $q_2$ are fixed by all of $G=\{ Id, \sigma ,\sigma^2,\sigma ^3\}$, and so $q_1,q_2\in\Q[a_1,a_2,a_3]$, giving the result in this case.

If $G=D_4$, then $G=\langle \sigma,\tau\rangle$ with $\tau =(r_3r_4)$ and $\sigma =(r_1r_3r_2r_4)$. Since $\tau (q_1)=q_1$ and $\tau (q_2)=q_2$ we also observe that $q_1,q_2\in \Q[a_1,a_2,a_3]$ in this case. This completes the proof.
\end{proof}
%
%
The main result of this section is the following proposition. 
%
%
\begin{prpstn}\label{Nq1}
Let $P(X)\in\Z[X]$ be irreducible, monic, quartic with Galois group $C_4$ or $D_4$.
Let $r_1,r_2,r_3,r_4$ be the roots of $P$ ordered such that $r_1r_2+r_3r_4\in \Q$ and 
let $K:=\Q(r_1+r_3)$. Then the form $q_1$ defined in \eqref{defq1} satisfies
\begin{equation*}
    q_1 (a_1,a_2,a_3)=\pm N_{K/\Q}(a_1+a_2 (r_1+r_3)+a_3 (r_1^2+r_1r_3+r_3^2)).
\end{equation*}
\end{prpstn}
%
%
\begin{proof}
We consider the cases when $G=C_4$ and $G=D_4$ separately.\\
\textbf{Case 1: $G=C_4$}. Let $G=\langle \sigma\rangle$ with $\sigma=(r_1r_3r_2r_4)$ and $r_1r_2+r_3r_4\in\Q$. We see that 

\begin{equation*}
q_1=\prod_{i=0}^3\sigma^i\Big ( \frac{a(r_1)-a(r_3)}{r_1-r_3}\Big )
=N_{\Q (r_1)/\Q}(a_1+a_2 (r_1+r_3)+a_3 (r_1^2+r_1r_3+r_3^2)).
\end{equation*}
To finish the proof, it remains to prove that $\Q(r_1+r_3)=\Q (r_1)$ is the splitting field of $P$. Since it is obviously contained in the splitting field, we just need to verify the field is not fixed by $\sigma^2$. But $c_3=-(r_1+r_2+r_3+r_4)=-(r_1+r_3)-\sigma^2(r_1+r_3)$ so if $\Q(r_1+r_3)$ is fixed by $\sigma^2$ then $\Q(r_1+r_3)=\Q$. But in this case $r_1+r_3=\sigma (r_1+r_3)=r_3+r_2$, so the roots would not be distinct, which contradicts our assumption. Thus $\Q(r_1+r_3)=\Q(r_1)$ as desired.

\medskip

\textbf{Case 2: $G=D_4$}. Let $G=\langle \sigma, \tau\rangle$ with $\sigma$ as above and $\tau=(r_3r_4)$. We work with the permutation $\sigma\tau=(r_1r_3)(r_2r_4)$. Let $L$ be the splitting field of  $P(X)$ and $K_0=\{ x\in L : \sigma\tau (x)=x\}$.
Then  $L/K_0$ is a Galois extension of degree  $2$ and $[K_0:\Q]=4$.
We observe that $r_1+r_3, \frac{a(r_1)-a(r_3)}{r_1-r_3}\in K_0$.

Now,  by looking the orbit of $\{ 1,3\}$ under the subgroup of $S_4$ generated by $\{ (1324), (34)\}$, 
we see that 
\begin{equation*}
N_{L/\Q}\Big ( \frac{a(r_1)-a(r_3)}{r_1-r_3}\Big )=q_1^2
\end{equation*}
and
\begin{equation*}
N_{L/K_0}\Big ( \frac{a(r_1)-a(r_3)}{r_1-r_3}\Big )=\Big ( \frac{a(r_1)-a(r_3)}{r_1-r_3}\Big )^2,
\end{equation*}
since $\frac{a(r_1)-a(r_3)}{r_1-r_3}\in K_0$.
By the transitive property of the   norms,
\begin{equation*}
N_{L/\Q}\Big ( \frac{a(r_1)-a(r_3)}{r_1-r_3}\Big )=N_{K_0/\Q}\Big (N_{L/K_0} \Big (\frac{a(r_1)-a(r_3)}{r_1-r_3}\Big )\Big )=N_{K_0/\Q}\Big (
\frac{a(r_1)-a(r_3)}{r_1-r_3}\Big )^2.
\end{equation*}
We deduce  that
$q_1=\pm N_{K_0/\Q}\Big (\frac{a(r_1)-a(r_3)}{r_1-r_3}\Big )$.

As in the case (i), to finish the proof it remains to check that $\Q(r_1+r_3)=K_0$. We have already seen that $\Q (r_1+r_3)\subset K_0$, and so it suffices to show  $[\Q(r_1+r_3):\Q]=4$. This follows from an identical argument to that of case 1 because 
 the intermediate extension between $K_0$ and $\Q$ is the subfield of $K_0$ fixed  by $\sigma ^2 = (r_1 r_2)(r_3r_4)$.
 \end{proof}
%
%

\medskip

%
%
We will   apply Theorem \ref{DistriNorm} with $K=\Q (r_1+r_3)$ and 
$\nu_1=1$, $\nu_2=r_1+r_3$, $\nu_3=r_1^2+r_3^2+r_1r_3$.
In the next lemma, we verify that these $3$ vectors $\nu_1$, $\nu_2$, $\nu_3$ are linearly independent over $\Q$ (even though the situation would be simpler if there was a linear dependence).
%
%
\begin{lmm}\label{independance}
With the previous notation, $1,r_1+r_3, r_1^2+r_3^2+r_1r_3$ are linearly independent over $\Q$.
 \end{lmm}
 %
%
 \begin{proof}
 In the proof of Proposition \ref{Nq1}, we have seen that   $r_1+r_3\not\in\Q$, and so certainly $1$ and $r_1+r_3$ are linearly independent.  Suppose that there exists  $u,v\in\Q$ such that 
$r_1^2+r_3^2+r_1r_3=u+v(r_1+r_3)$. If we apply $\sigma^2=(r_1r_2)(r_3r_3)$ to this expression, we find $r_2^2+r_4^2+r_2r_4=u+v(r_2+r_4)$.
Summing this two equations gives 
\begin{equation*}
\sum_{i=1}^4r_i^2+r_1r_3+r_2r_4=2u+v(r_1+r_2+r_3+r_4).
\end{equation*}
This contradicts  the fact that   $r_1r_3+r_2r_4\not\in\Q$ (since $\sum_i r_i,\sum_i r_i^2\in\Q$).
\end{proof}
%
%

\subsection{On the solutions of some congruence equations with $B_{14}$ and $q$}

%
%

In this section we compute the number of solutions of various equations involving the factors $q_1,q_2$ and the cofactors $B_{13}, B_{14}$. These preliminary lemmas will be applied in several places in the proof of Theorem \ref{cP}. 

Some parts of this section are similar to \cite[Lemma 3.9]{B15}  or \cite[Section 13]{D15}, but both of these previous approaches relied on the condition $G=(\Z /2\Z)^2$ which we do not assume, and so we require a slightly different approach.

Let $\delta_P$ be the discriminant of the splitting field of $P$.

%
%
\begin{lmm}\label{pgcdq1q2}
Suppose that $(p,a_3\delta_P\Disc{P})=1$ and $a_2\in \Z$. 
Let $Q_p (a_2,a_3)$ denote the number of integers $a_1$  with $0\le a_1<p$
such that
\begin{equation}\label{pq1q2}
q_1 (a_1,a_2,a_3)\equiv q_2 (a_1,a_2,a_3)\equiv 0\Mod p.
\end{equation}
Then 
$$\label{qq12} Q_p(a_2,a_3)=\begin{cases}
1, & 
\text{if} \ P((a_2-c_3a_3)\ov{a_3})\equiv 0\Mod p;\\

0, & \text{otherwise}.\\
\end{cases}
$$
\end{lmm}
%
%
\begin{proof}
Let $L$ be the splitting field of $P$ and $\Oc_L$ its ring of integers. Since $(p,\delta_P)=1$, $p$ is not ramified in $\Oc_L$ and so its decomposition into prime ideals is 
$p\cO_L=\prod_{i=1}^s \gP_i$ with $N_L (\gP_i)=p^t$ for some integers $s,t$ with $st=[L:\Q ]$. Formulas \eqref{defq1}, \eqref{defq2}, \eqref{rij} give us the factorisation of the polynomials $q_1$ and $q_2$ over $\Oc_L$. The condition $q_1(a_1,a_2,a_3)\equiv q_2(a_1,a_2,a_3)\equiv 0\Mod{p}$ is equivalent to one of the factors of $q_1$ and one of the factors of $q_2$ vanishing $\Mod{\gP_m}$ for each $1\le m\le s$. 

First we suppose that \eqref{pq1q2} has a solution. Thus for all $1\le m\le s$, there exists $(i,j)\in\{ (1,3), (1,4), (2,3), (2,4)\}$ and $(k,\ell )\in\{ (1,2), (3,4)\}$
such that
\begin{equation*}\left\{
\begin{matrix}
a_1+a_2(r_i+r_j)+a_3(r_i^2+r_ir_j+r_j^2)&\equiv 0\Mod{\gP_m},\\ 
a_1+a_2(r_k+r_\ell )+a_3(r_k^2+r_kr_\ell+r_\ell^2)&\equiv 0\Mod{\gP_m}.\\
\end{matrix}
\right.
\end{equation*}
Eliminating $a_1$,  we find
$$a_2 (r_i+r_j-r_k-r_\ell)\equiv a_3 (r_k^2+r_kr_\ell +r_\ell ^2-r_i^2-r_ir_j-r_j^2)\Mod {{\gP_m}}.$$
For notational simplicity we concentrate on the case $i=k=1$, $j=3$, $\ell =2$; the other cases are entirely analogous (noting that $\{i,j\}\cap\{k,\ell\}\ne \emptyset$).
We obtain
$$(r_3-r_2)a_2\equiv  a_3(r_2-r_3)(r_1+r_2+r_3)\Mod{\gP_m}.$$
Since $p\nmid\Disc (P)$ and $(r_3-r_2)|\Disc(P)$, we see that $r_3-r_2\not\equiv 0\Mod{\gP_m}$, and so (recalling $c_3=-r_1-r_2-r_3-r_4\in\Z$) we have $a_2\equiv  a_3 (c_3+r_4)\Mod{\gP_m}$. This implies that $r_4\equiv (a_2-a_3c_3)\ov{a_3}\Mod {\gP_m}$
and so
\[
P((a_2-a_3c_3)\ov{a_3})\equiv 0\Mod{\gP_m}.
\]
Since this argument is valid for all $m$, we find that $P((a_2-a_3c_3)\ov{a_3})\equiv 0\Mod{p}$. Thus if $P((a_2-c_3a_3)\ov{a_3})\not\equiv 0\Mod p$ then $Q_p(a_2,a_3)=0$.

Now we suppose that $P((a_2-c_3a_3)\ov{a_3})\equiv 0\Mod{p}$. Then there exists $j\in\{1,2,3,4\}$ such that $ r_j\equiv (a_2-a_3c_3)\ov{a_3}\Mod p$.
We may suppose that $j=4$; the other cases are analogous. We see that this implies that $a_2\equiv a_3(-r_1-r_2-r_3)\Mod{p}$ and that $r_4\in \Z+p\Oc_L$. Moreover,
we check that 
\begin{align*}
a_2 (r_1+r_3)+a_3 (r_1^2+r_1r_3+r_3^2)&=a_3(-c_2-c_3r_4-r_4^2)\Mod{p},\\
 a_2(r_1+r_2)+a_3(r_1^2+r_1r_2+r_2^2) &=a_2 (r_1+r_3)+a_3(r_1^2+r_1r_3+r_3^2)\Mod{p}.
 \end{align*}
Thus the system \eqref{pq1q2} admits the solution $a_1=-(a_2(r_1+r_3)+a_3(r_1^2+r_1r_3+r_3^2))\Mod{p}$, noting this is in $\Z+p\Oc_L$. Thus $Q_p(a_2,a_3)\ge 1$.

Moreover, there are no other solutions modulo $p$, because the previous computations showed that for any $\{ i,j,k,\ell\}=\{ 1,2,3,4\}$, if we have
\[
    \left\{ \begin{matrix}
    a_1 +a_2(r_i+r_j)+a_3(r_i^2+r_ir_j+r_j^2)=0\Mod{\gP_m},\\
    a_1 +a_2(r_i+r_k)+a_3(r_i^2+r_ir_k+r_k^2)=0\Mod{\gP_m},
    \end{matrix}\right.
    \]
   then we must have $(a_2-c_3a_3)\ov{a_3}=r_\ell\Mod{\gP_m}$. But the roots $r_1,r_2,r_3,r_4$ are distinct modulo $p$ when $(p,\Disc P)=1$, and so we must have $\ell=4$. Thus the only solution is $a_1\equiv -a_2(r_i+r_j)-a_3(r_i^2+r_ir_j+r_j^2)\Mod{p}$ (noting that these are the same for all choices of $\{i,j,k\}=\{1,2,3\}$). Thus $Q_p(a_2,a_3)=1$ when $P((a_2-c_3a_3)\overline{a_3})\equiv 0\Mod{p}$.
\end{proof}
%
%
Recall that $B_{14},B_{13}$ are cubic forms in $a_0,a_1,a_2,a_3$ given explicitly by \eqref{eq:B13} and \eqref{eq:B14}. For later estimates, we need to understand the number of solutions in $a_0$ of the equations $B_{14}\equiv 0\Mod p$ or $B_{13}\equiv 0\Mod p$. Since $B_{14}$ has degree $2$ in $a_0$, we can get an explicit formula for its roots in $\ov{\F_p}$ with the discriminant.
%
%
\begin{lmm}\label{lmm:Delta14}
Let $\Delta_{14}\in\Z[a_1,a_2,a_3]$ be the discriminant of $B_{14}$
viewed as a polynomial in $a_0$. Then
\begin{equation}\label{Delta14}
        \Delta_{14}=-q_3h,
\end{equation}
where $h$ is given by
\begin{align*}
h(a_1,a_2,a_3)&=-4a_1^2+4c_3a_1a_2+(-3c_3^2+4c_2)a_1a_3-c_3a_2^2+(c_3^3-4c_1)a_2a_3\\
&\qquad\qquad
        +(-c_2c_3^2+4c_1c_3-4c_0)a_3^2\\
        &=(r_1+r_2-r_3-r_4)^2 q_3(a_1,a_2,a_3)-q_2 (a_1,a_2,a_3).
\end{align*}
\end{lmm}
%
%
We remind the reader that $q_3$ is the form defined in \eqref{deff} and $q_2$ is the form given by \eqref{formf}.
%
%
\begin{proof}
This follows from explicit computation using the formula for the discriminant of a quadratic.
\end{proof}
We recall that we have ordered the roots of $P$, $r_1,r_2,r_3,r_4$ so that $r_1r_2+r_3r_4\in\Q$.
%
%
\begin{lmm}\label{t1t2}
Let $t_1:=r_1r_2+r_3r_4$ and $t_2:=(r_1+r_2)(r_3+r_4)$. Then $t_1,t_2\in\Z$.
\end{lmm}
%
%
\begin{proof}
First we note that $t_2$ is fixed by the permutations $(r_1r_3r_2r_4)$ and $(r_3r_4)$, so $t_2\in\Q$. Let $R(X)$ be a cubic resolvent associated to $P$, given by (see \cite{C})
\begin{equation*}\begin{split}
R(X):=   & (X-(r_1+r_2)(r_3+r_4))(X-(r_1+r_3)(r_2+r_4))(X-(r_1+r_4)(r_2+r_3))\\&=
   X^3-2c_2X^2+(c_2^2+c_3c_1-4c_0)X+(c_3^2c_0+c_1^2-c_3c_2c_1).
    \end{split}
\end{equation*}
Then we see that $R(X)\in \Z[X]$ and it is a well-known fact that when $P$ has Galois group $C_4$ or $D_4$, $R(X)$ has a unique root over $\Q$, which must be $t_2$. Since $R(X)$ is monic we see that $t_2\in\Z$. Since $t_1+t_2=c_2\in\Z$ we see that $t_1\in\Z$.
\end{proof}
%
%
\begin{rmk}
(i) If $t_2=0$, that is $(r_1+r_2)(r_3+r_4)=0$, then we have in fact $r_1+r_2=r_3+r_4=0$ since $\sigma (r_1+r_2)=r_3+r_4$.
This implies that $c_3=0=c_1$. This situation is analogous to \cite[Lemma 3.9]{B15} (or also \cite[Lemmas 13.2 and 13.3]{D15} for the polynomial $X^4-X^2+1$.)

(ii) We have $t_1\not =0$, since otherwise we would have $r_1-r_2=\pm (r_3-r_4)$. If we compose with the embedding $\tau = (r_3 r_4)$,
we find $r_1-r_2=r_3-r_4=0$ which is not possible.
\end{rmk}
%
%
\begin{lmm}\label{B14p}
Let $a_1,a_2,a_3\in\Z$ be such that $(q(a_1,a_2,a_3),q_3(a_1,a_2,a_3))=1$ and $q(a_1,a_2,a_3)$ is squarefree. Let $t_2=(r_1+r_2)(r_3+r_4)\in\Z$.

Let $p$ be a prime with $p|q(a_1,a_2,a_3)$ and $p\nmid a_2a_3\delta_P\Disc P$.
\begin{enumerate}[(i).]
\item If $p|q_1(a_1,a_2,a_3)$ or $p\nmid c_3^2-4 t_2$, then 
\[
|\{ 0\le a_0 < p : B_{14}(a_0,a_1,a_2,a_3)\equiv 0\Mod p\}|=2.
\]
\item  If  $p|q_2(a_1,a_2,a_3)$ and $p|c_3^2-4t_2$  then
\[
|\{ 0\le a_0 < p : B_{14}(a_0,a_1,a_2,a_3)\equiv 0\Mod p\}|=1.
\]
\item We have
\[
|\{ 0\le a_0 <p : B_{13}(a_0,a_1,a_2,a_3)\equiv B_{14}(a_0,a_1,a_2,a_3)\equiv 0\Mod{p})\} |=1.
\]
\end{enumerate}
\end{lmm}
%
%
\begin{proof}
We recall from \eqref{resultant} and \eqref{eq:qDef} that $q|R_0$, the resultant of $B_{13}$ and
$B_{14}$ viewed as polynomials in $a_0$. Therefore since $p|q(a_1,a_2,a_3)$, we have that $p|R_0(a_1,a_2,a_3)$, and so the two quadratic polynomials in $a_0$, $B_{13}$ and $B_{14}$ have a common root 
in some finite extension of $\F_p$.

If this common root is not in $\F_p$ then its conjugate is also a common root of $B_{13}$ and $B_{14}$, and so we would have $R_0(a_1,a_2,a_3)=q(a_1,a_2,a_3)q_3(a_1,a_2,a_3)\equiv 0\Mod{p^2}$. But this is impossible since we assume that $q(a_1,a_2,a_3)$ is squarefree and coprime to $q_3(a_1,a_2,a_3)$ with $p|q(a_1,a_2,a_3)$. Therefore the common root must lie in $\F_p$. This proves assertion (iii).

Since the common root of $B_{13}$ and $B_{14}$ is in $\F_p$ and $B_{14}$ is quadratic, both the roots of $B_{14}$ (seen as a polynomial in $a_0$) are in $\F_p$. Thus the number of $0\le a_0<p$ with $B_{14}\equiv 0\Mod p$ is $1$  when $p|\Delta_{14}$ and $2$ otherwise.

If $p|q_2$, by Lemma \ref{lmm:Delta14}, $p|\Delta_{14}$ if and only if 
$p|(r_1+r_2-r_3-r_4)^2$. This gives the assertion (i) and (ii) in the case $p|q_2$ because $(r_1+r_2-r_3-r_4)^2=c_3^2-4t_2$.

We now consider the case $p|q_1$.
Let $L$ be the splitting field of $P$, $\Oc_L$ its integer ring and $p\Oc_L=\prod_{m=1}^s\gP_m$, the decomposition of $p$ in $\Oc_L$. 
Then for all $m$ there exists $(i,j)\in\{ (1,3), (1,4), (2,3), (2,4)\}$ such that 
$a_1\equiv -a_2(r_i+r_j)-a_3(r_i^2+r_ir_j+r_j^2)\Mod{\gP_m}$.
We may suppose that $i=1$ and $j=3$, the other cases being similar.
Substituting $-a_2(r_1+r_3)-a_3(r_1^2+r_1r_3+r_3^2)$ for $a_1$ in the expression for $h$ in Lemma \ref{lmm:Delta14}, gives
\begin{equation}\label{hr1r3}\begin{split}
    h(a_1,a_2,a_3)
    \equiv -(a_3&(r_1+r_2+r_3)+a_2)(a_3(r_1+r_3+r_4)+a_2)\\
    &\times     (r_1-r_2+r_3-r_4)^2\Mod {\gP_m}.\end{split}
\end{equation}
We have that $a_3(r_1+r_2+r_3)+a_2\not\equiv 0\Mod{\gP_m}$. If this were not the case we would have $a_3(-c_3-r_4)+a_2\equiv 0\Mod {\gP_m}$, and then $P((a_2-a_3c_3)\ov{a_3})=0\Mod p$.
By Lemma \ref{pgcdq1q2} we would have $p| (q_1,q_2)$ which is not possible when $q$ is squarefree. Similarly $a_3(r_1+r_3+r_4)+a_2\not \equiv 0\Mod{\gP_m}$. 

Thus $\Delta_{14}\equiv 0\Mod{\gP_m}$ if and only if $r_1-r_2+r_3-r_4\equiv 0\Mod{\gP_m}$ for all $m$. But this is equivalent to $r_1-r_2+r_3-r_4\equiv 0\Mod{p}$, and so $\gamma(r_1-r_2+r_3-r_3)\equiv 0\Mod{p}$ for all embeddings $\gamma$. Applying this with $\gamma=\iota,\tau$ we see that $p|\Delta_{14}$ if and only if $r_1\equiv r_2\Mod{p}$, which is impossible since $p\nmid \Disc(P)$. Thus when $p|q_1$ we have $p\nmid \Delta_{14}$, and so $B_{14}$ has two roots $\Mod{p}$.
\end{proof}
%
%
\begin{lmm}\label{B13B14N}
Let $a_0,a_1,a_2,a_3,p\in\Z$ be such that $(q(a_1,a_2,a_3),q_3(a_1,a_2,a_3))=1$, $q(a_1,a_2,a_3)$ is squarefree and $p| (q(a_1,a_2,a_3),B_{14}(a_0,a_1,a_2,a_3))$. Then we have
\[
p|N_P(\alpha)\Leftrightarrow p|B_{13}(a_0,a_1,a_2,a_3).
\]
where $\alpha=a_0+a_1r_1+a_2r_1^2+a_3r_1^3$.
\end{lmm}
%
%
\begin{proof}
This is a variant of \cite[Lemma 13.3]{D15} (or \cite[Section 6.1]{B15}). By \eqref{B2iB1j} and \eqref{BBN}, we have 
$$(B_{13}-c_3B_{14})B_{13}-B_{14}(B_{12}-c_2B_{14})=q_3 N_P(\alpha ).$$
The Lemma follows from this formula since $(p,q_3(a_1,a_2,a_3))=1$.
\end{proof}
%
%
%
%
\section{The set of ideals $\cJ$}\label{sec:Ideal}
%
%
In this section we define a set $\cJ$ of principle ideals which correspond to the forms $q_1$ and $q_2$ having a convenient prime factorisation. This will have a slightly technical definition to ensure that it is compatible with later arguments. 

It is known (see \cite[Lemma 4.2]{May15b}) that there is a fundamental domain $\cD_P$ of the units action group such that if $\alpha =a_0+a_1r_1+a_2r_1^2+a_3r_1^3\in\cD_P$, then $\max (|a_0|, |a_1|, |a_2|, |a_3|)\ll N_P(\alpha )^{1/4}$ and so $|\sigma  (\alpha )|\ll N_P(\alpha )^{1/4}$ for all embeddings $\sigma$. We recall that the forms $q_1(a_1,a_2,a_3)$ and $q_2(a_1,a_2,a_3)$ are defined by \eqref{defq1} and \eqref{defq2}, the polynomials $P_1(X):=(X-(r_1+r_2))(X-(r_3+r_4))$ and $P_2(X):=(X-(r_1^2+r_1r_2+r_2^2))(X-(r_3^2+r_3r_4+r_4^2))$ with discriminants $\Delta_1$ and $\Delta_2$ respectively, $D_{q_2}$ from \eqref{eq:DqDef}, and $\delta_P$
is the discriminant of the splitting field of $P$.
With this notation we introduce a constant $q_0$ depending only on the polynomial $P$ 
\begin{equation}\label{defq0}
q_0=512(1+c_3^2+|c_2|+|t_1|+|t_2|)\delta_P\Disc{P},
\end{equation}
where $t_1$ and $t_2$ are the integers defined in Lemma \ref{t1t2}.
The set $\cJ$ will depend on various auxiliary absolute constants $$\alpha_0,\theta_{11},\dots,\theta_{16},\theta_{21},\tau_{11},\dots,\tau_{16},\tau_{21}\in(0,1).$$ These constants will be required to satisfy various inequalities, specifically
\begin{align}
[\theta_{ij},\theta_{ij}+\tau_{ij}]\cap[\theta_{i'j'},\theta_{i'j'}+\tau_{i'j'}]&=\emptyset\quad \text{for }(i,j)\ne (i',j'),\label{eq:Con10}\\
0<\theta_{1j}<\theta_{1j}+\tau_{1j}&<7/32  \ {\rm for}\ {\rm all}\ 1\le j\le 6,\label{eq:Con1}\\
\alpha_0&<\frac{1}{2^{15}},\label{eq:Con8}\\
\sum_{j=1}^6(\theta_{1j}+\tau_{1j})&<1+\alpha_0/2,\label{eq:Con3}\\
\theta_{11},\theta_{12},\theta_{13},\theta_{14},\theta_{15},\theta_{16},\theta_{21}&>1+\alpha_0 -\sum_{j=1}^6\theta_{1j}\label{eq:Con5},
\end{align}
\vspace{-0.6cm}
\begin{align}
\frac{1+\alpha_0}{4}<\theta_{11}+\theta_{12}+\theta_{13}
&<\frac{2+\alpha_0}{4}-\tau_{11}-\tau_{12}-\tau_{13},\label{eq:Con4}\\
    \theta_{21}+\tau_{21}&<\frac{2+\alpha_0}{200}-\frac{\sum_{i=1}^3(\theta_{1i}+\tau_{1i})}{50},\label{eq:Con13}\\
    \theta_{21}+\tau_{21}&<
    \Big (\frac{4(\theta_{11}+\theta_{12}+\theta_{13})}{1+\alpha_0}-1\Big )\frac{2+\alpha_0}{800}.\label{eq:Con14}
    \end{align}
There is reasonable flexibility in how we might choose these constants (and the above constraints could likely be weakened significantly), but for concreteness, we can chose the following explicit values of these variables:\\
\begin{tabular}{llll}
$\alpha_0= 0.00001$, & $\theta_{11}=0.1398$, & $\theta_{12}=0.1401$, & $\theta_{13}=0.1402$, \\ $\theta_{14}=0.21$,&
$\theta_{15}=0.19$, & $\theta_{16}=0.1799$, & $\theta_{21}=0.001$,\\ 
\multicolumn{2}{l}{$\tau_{ij}=0.0000001$ for all $(i,j)\in I_\cC.$} & &\\
\end{tabular}

Now we are ready to define the set $\cJ$. The set $\cJ$ is the set of all principal ideals
$(\alpha)$ of $\cO_{\Q(r_1)}$ with generator $\alpha=a_0+a_1r_1+a_2r_1^2+a_3r_1^3$ where $(a_0,a_1,a_2,a_3)\in\Z^4\cap \cD_P$, satisfying the conditions (C1), (C2), (C3), (C4) and (C5) below.
%
%
\begin{enumerate}
\item[(C1)] 
$q(a_1,a_2,a_3)$ is squarefree.
\item[(C2)] \textit{Size conditions:} We have
\[
\begin{split}
q(a_1,a_2,a_3)&\ge X^{3/2},\\
|B_{14}(a_0,a_1,a_2,a_3)|&\ge X^{3/4},\\
N_P(\alpha)&\in[X^{1+\alpha _0/2},X^{1+\alpha_0}].
\end{split}
\]
\item[(C3)] \textit{Factorisation conditions on $\alpha$:} There exists ideals $K,L$ such that $(\alpha)=KL$ with $K$ a prime ideal satisfying
\begin{equation}\label{K}
X^{4\alpha_0}<N_P(K)\le X^{5\alpha _0}.
\end{equation}
\item[(C4)] \textit{Factorisations conditions of auxiliary polynomials:}
The  values of the forms $q_1(a_1,a_2,a_3)$ and $q_2(a_1,a_2,a_3)$ evaluated at $a_1,a_2,a_3$ can be factored as: 
\begin{equation}\label{qij}
\begin{split}
q_1(a_1,a_2,a_3)&=\prod_{j=1}^7 q_{1j},\\ 
q_2(a_1,a_2,a_3)&=q_{21}q_{22} \ {\rm with} \ q_{21}\equiv 1\Mod {D_{q_2}},
\end{split}
\end{equation}
where $q_{21}$, $q_{11},q_{12},q_{13},q_{14},q_{15},q_{16}$ are prime numbers satisfying
\[
q_{ij}\in [X^{\theta_{ij}},X^{\theta_{ij}+\tau_{ij}}]
\]
for all $(i,j)\in\{ (1,1),(1,2),(1,3),(1,4),(1,5),(2,1)\}$, and where $q_{22},q_{16}$ are integers (not necessarily prime) with 
\[
P^-(q_{22}),P^-(q_{17})>q_0
\]
where $q_0$ is given by \eqref{defq0}.
\item[(C5)]\textit{Coprimality conditions: }
\begin{enumerate}
\item $(a_2,a_3)=30$ and 
$a_2,a_3\equiv 30\Mod{900} $, $a_1\equiv 1\Mod{30}$.
\item 
$(N_P (\alpha ),q_0)=1$.
\item $(q(a_1,a_2,a_3),q_3(a_1,a_2,a_3))=1$.
\item  $(q(a_1,a_2,a_3),B_{14}(a_0,a_1,a_2,a_3))=1$. 
\item $(q(a_1,a_2,a_3),a_2a_3)=1$.
\end{enumerate}
\end{enumerate}
%
%
With this definition of $\cJ$, we can verify the property \eqref{eq:JProp1} if $\delta_0$ is chosen small enough.
%
%
\begin{lmm}
We have that for all $\gJ\in\cJ$
\[
\prod_{\substack{\pf^e\|\gJ\\ N_P(\pf)\le X}}N_P(\pf)\ge X^{1+\alpha_0/2}.
\]
\end{lmm}
%
%
\begin{proof}
This is a consequence of (C2) which forces $N_P(\alpha)\ge X^{1+\alpha_0/2}$ and (C3), which forces all ideal factors of $(\alpha)$ to have norm at most $\max(X^{5\alpha_0},X^{1-3\alpha_0})<X$. (We note that \eqref{eq:Con8} implies that $19\alpha_0<1$).
\end{proof}
%
%
The next Lemma says that the congruence $n\equiv{r_1}\Mod{\gJ}$ can be solved when $\gJ\in\cJ$. We recall that $\varrho_P$ is defined in \eqref{eq:RhoPDef}.
%
%
\begin{lmm}
For all $\gJ\in\cJ$  we have $\varrho_P (\gJ)=1$.
\end{lmm}
%
%
\begin{proof} Let $\gJ\in\cJ$. There exists
 $\alpha =a_0+a_1r_1+a_2r_1^2+a_3r_1^3$ with $(a_0,a_1,a_2,a_3)\in\Z^4\cap \cD_P$ satisfying (C1),(C2),(C3),(C4),(C5) and such that $\gJ=(\alpha )$.
By Lemma \ref{kj} and (C5)(d), $(N_P(\gJ),B_{14} (a_0,a_1,a_2,a_3))=1$. The condition (C5)(b) and Lemmas \ref{r} and \ref{ideaux} imply then that $\varrho_P (\gJ)=1$.
\end{proof}
%
%
\begin{rmk}
As mentioned in Section \ref{Section:InitialSteps}, we will work with the set $\cJ_2$ which is the set of $\gJ\in\cJ$ such that $P^-(N_P (\gJ))>X^{\theta_0}$. This condition implies (C5)(b).
\end{rmk}
%
%
We see from condition $(C2)$ that if $\af\in\cJ$ then $\af=(a_0+a_1\nu_1+a_2\nu_2+a_3\nu_3)$ for some $\a\in\Z^4$ which lies in the region
\begin{equation}\label{defRP}\begin{split}
\cR:=\Bigl\{ \a\in&\R ^4\cap \cD_P :\ 7X^{1+\alpha_0/2}<\widetilde{N}(a_0,a_1,a_2,a_3)\le X^{1+\alpha_0}, \\
&\ |B_{14}(a_0,a_1,a_2,a_3)|\ge X^{3/4},\ |q(a_1,a_2,a_3)|\ge X^{3/2}\Bigr\}
.\end{split}\end{equation}
Here we have written $\widetilde{N}_P$ as the extension of $N_P(\alpha)$ to $\mathbb{R}^4$;
\begin{equation}\label{eq:RealNDef}
\widetilde{N}(a_1,a_2,a_3,a_4):=\prod_{i=1}^4 \Big (\sum_{j=1}^4 a_j\sigma _i (\nu_j)\Big ).
\end{equation}
By our choice of $\cD_P$ we see that if $\a\in\cR$ then $|a_i|\ll X^{(1+\alpha_0)/4}$ for all $i\in\{1,2,3,4\}$.
%
%
For notational convenience we set $I_\cC$ to be the set
\begin{equation}\label{defIC}
I_\cC:=\{ (1,1),(1,2),(1,3),(1,4),(1,5),(1,6),(2,1)\},
\end{equation}
so that condition $(C4)$ forces $q_{ij}\in [X^{\theta_{ij}},X^{\theta_{ij}+\tau_{ij}}]$ for all $(i,j)\in I_\cC$, for example.

\section{Proof of Proposition \ref{prpstn:S1}: The term $S_1$}\label{sec:S1}

In this section we establish Proposition \ref{prpstn:S1} by bounding the sum $S_1$ defined by \eqref{defS0S1}. The overall approach is similar to previous works. First we reduce to controlling exponential sums, then remove the $a_0$-dependence in the denominator of the phase which means that we can apply the $q$-analogue of Van der Corput's method whenever the denominator of the phase takes a suitably friable form.

%
%
\begin{lmm}[Reduction to exponential sums] \label{TEl}
Let $S_1$ be as given by \eqref{defS0S1}, and $\eta_0,\alpha_0,\theta_0>0$ be such that 
\begin{equation*}
    \alpha_0 <\eta_0 <1-\frac{9}{4}\alpha_0,\qquad 12\theta_0+19\alpha _0 <1.
\end{equation*}
Then for $X\ge 2$, $H=X^{\eta_0}$ we have
\begin{equation}\label{S1start}
    S_1\ll (\log H)\sum_{K\in\cK}\sum_{\substack{{A}\\
    {N_P(A)|\cP (X^{\theta_0})}\\
    {N_P (A)\le X^{3\theta _0}}}}\sum_{h\le H^2}
    \frac{|E_1 (X,h ;KA)|+|E_2 (X,h ; KA)|}{h+h^2/H}+o(X),
\end{equation}
where for $\ell \in\{1,2\}$
\begin{equation*}
    E_\ell (X,h;KA):=\sum_{\substack{{(\alpha)\in\cJ}\\
   { KA |(\alpha )}}}\e \Big ( \frac{ h\ell X}{N_P(\alpha )}-\frac{hU\ov{B_{14}}}{q}\Big ).
    \end{equation*}
\end{lmm}
%
%
\begin{proof}
This is \cite[Lemma 5.1]{B15}.
\end{proof}
%
%
To show that $S_1$ is small, our task is therefore reduced to showing cancellation in the exponential sums $E_\ell$. Lemma \ref{kj} allows us to put the exponential phase into a form where we can then apply the $q$-analogue of Van der Corput's method. The bounds from this method are summarised in the following lemma.
%
%
\begin{lmm}[$q$-Van der Corput for short exponential sums]\label{HB} Let $k,D\ge 1$, $\varepsilon >0$. Let $f,g,v\in\Z [X]$ of degree $\le D$ and $r=r_0\cdots r_k$ be squarefree
such that $P^-(r)>2^k D$. Suppose that for every $p|r$ there is no polynomial $w\in\Z [X]$ of degree $\le k+1$ such that 
$f(X)\equiv w(X)g(X)\Mod p$. Moreover, suppose that $v(X)$ is not the zero polynomial $\Mod p$ for any $p|r$. Then  for $A,B,h\ge 1$ we have
\begin{equation*}\begin{split}
\sum_{\substack{{ A<n\le A+B}\\ {(v(n)g(n),r)=1}}}\e \left (\frac{hf(n)\ov{g(n)}}{r}\right )
&\ll_{k,D,\varepsilon}
r^\varepsilon B \Big [\Big (\frac{\Delta}{r_0}\Big  )^{1/2^{k+1}}+\Big (\frac{r_0}{\Delta B^2}\Big )^{1/2^{k+1}}\\
&+\sum_{j=1}^k\Big (\frac{r_{k+1-j}}{B}\Big )^{1/2^j}\Big ],
\end{split}
\end{equation*}
where $\Delta:=(r_0,h)$.
\end{lmm}
%
%
\begin{proof}
This is \cite[Lemme 3.10]{B15} (a small variation of \cite[Theorem 2]{HB01}).
\end{proof}
%
%
To apply this lemma, the denominator $q(a_1,a_2,a_3)$ in our exponential phase must have a good factorisation. We will apply Theorem \ref{DistriNorm} to show that for a positive proportion of $(a_1,a_2,a_3)$ the denominator $q=q(a_1,a_2,a_3)$ has such a factorisation.
We want the $\e(hU\ov{B_{14}}/q)$ factor to oscillate suitably to give this cancellation via Lemma \ref{HB}. The following lemma will ensure that this factor is not degenerate.
%
%
\begin{lmm}\label{qU0U1}
Let $U=a_0U_1+U_0$, $V=a_0V_1+V_0$ as in \eqref{defU1} and in \eqref{U1V1}.
If $a_0,a_1,a_2,a_3\in\Z$ are such that $(a_0+a_1\nu_1+a_2\nu_2+a_3\nu_3)\in\cJ$, then 
\[
(U_0(a_1,a_2,a_3),U_1(a_1,a_2,a_3),q(a_1,a_2,a_3))=1.
\]
\end{lmm}
%
%
\begin{proof} 
Imagine for a contradiction that $p| q(a_1,a_2,a_3),U_0(a_1,a_2,a_3),U_1(a_1,a_2,a_3)$. Then $U(a_0',a_1,a_2,a_3)=0\Mod{p}$ for all $a_0'$, and so the equation $U B_{13}+VB_{14}=q q_3$ \eqref{UV} simplifies to give
\[
V(a_0',a_1,a_2,a_3)B_{14}(a_0',a_1,a_2,a_3)\equiv 0\Mod{p}
\]
for all $a_0'$. Condition (C5)(d) then implies that $B_{14}(a_0',a_1,a_2,a_3)$ does not identically vanish $\Mod{p}$, so $V_1(a_1,a_2,a_3)=V_0(a_1,a_2,a_3)=0\Mod{p}$. 

By conditions (C1) and (C5)(c), $a_1,a_2,a_3$ satisfy the hypotheses of Lemma \ref{B14p}. But this implies that there is a choice of $a_0'$ such that $B_{14}(a_0',a_1,a_2,a_3)=B_{13}(a_0',a_1,a_2,a_3)=0\Mod{p}$. Evaluating \eqref{UV} at $a_0',a_1,a_2,a_3$ then implies that
\[
q(a_1,a_2,a_3)q_3(a_1,a_2,a_3)\equiv 0\Mod{p^2}.
\]
This is impossible since $(q(a_1,a_2,a_3),q_3(a_1,a_2,a_3))=1$ and $q(a_1,a_2,a_3)$ is squarefree by conditions (C5)(c) and (C1). This gives the result.
\end{proof}
%
%
Finally, we need a short lemma to show that we can restrict attention to $q(a_1,a_2,a_3)$ being not too small.
%
%
\begin{lmm}[Bounding terms with $q_2 (a_1,a_2,a_3)$ small]\label{removingsmallq22}
Let $\tau_{20}>0$ and  for $\ell =1,2$, $E_\ell '(X,h;KA)$ be the contribution in $E_\ell (X,h;KA)$ given by the $(\alpha)\in\cJ$ such that $|q_2(a_1,a_2,a_3)|\le X^{(1+\alpha_0)/2-\tau_{20}}$.
Then 
\[
E'_\ell (X,h;KA)\ll \frac{X^{1+\alpha_0-\tau_{20}/2}}{N_P (KA)}.
\]
\end{lmm}
\begin{proof}
Since $N_P(AK)$ is square-free by construction, by Lemma \ref{r}, there exists an integer $j$ such that $r_1\equiv j\Mod{KA}$. The condition $KA|(\alpha )$ is therefore equivalent to 
\[
    a_0\equiv {-a_1 j-a_2 j^2-a_3 j^3}\Mod{N_P(AK)}.
\]
Thus, for any given $a_1,a_2,a_3$ there are $O(X^{(1+\alpha_0)/4}/N_P(KA))$ terms $a_0$ in $E_{\ell}'(X,h;KA)$.

We recall that $q_2 (a_1,a_2,a_3)=\prod_{i=0}^1 L_{i}(a_1,a_2,a_3)$ with for $i=0,1$:
\[ 
L_{i}(a_1,a_2,a_3)=a_1+(r_{1+2i}+r_{2+2i})a_2+(r_{1+2i}^2+r_{1+2i}r_{2+2i}+r_{2+2i}^2)a_3.
\]
If $|q_2 (a_1,a_2,a_3)|\le X^{(1+\alpha_0)/2 -\tau_{20}}$
then 
\begin{equation}
\label{smallLi0Li1}
\min_{i=0,1}|L_{2i} (a_1,a_2,a_3)| \ll X^{(1+\alpha_0)/4-\tau_{20}/2}.
\end{equation}
For any given $a_2,a_3$, the number of $a_1$ satisfying \eqref{smallLi0Li1} is $O(X^{(1+\alpha_0)/4-\tau_{20}/2}).$ Since there are $O(X^{(1+\alpha_0)/2})$ choices of $a_2,a_3$, the total number of terms in $E'(X,h;KA)$ is $O(X^{1+\alpha_0-\tau_{20}/2})$.
\end{proof}

%
%
We are now able to bound $S_1$ suitably.
%
%
\begin{proof}[Proof of Proposition \ref{prpstn:S1}]
First we wish to apply Lemma \ref{TEl}. By \eqref{eq:Con8}, we have $\alpha_0<1/20$, so the conditions of the lemma hold if $\eta_0$ is slightly larger than $\alpha_0$ and $\theta_0$ is sufficiently small. This gives 
\[
    S_1\ll (\log H)\sum_{K\in\cK}\sum_{\substack{{A}\\
    {N_P(A)|\cP (X^{\theta_0})}\\
    {N_P (A)\le X^{3\theta _0}}}}\sum_{h\le H^2}
    \frac{|E_1 (X,h ;KA)|+|E_2 (X,h ; KA)|}{h+h^2/H}+o(X),
    \]
where
\[
    E_\ell (X,h;KA):=\sum_{\substack{{(\alpha)\in\cJ}\\
   { KA |(\alpha )}}}\e \Big ( \frac{ h\ell X}{N_P(\alpha )}-\frac{hU\ov{B_{14}}}{q}\Big ).
   \]
   We write $E_\ell=E_\ell'+E_{\ell}$ where $E_\ell'$ is the contribution from terms in $E_\ell$ with $|q_2(a_1,a_2,a_3)|\le Y$, and $E_\ell''$ is the contribution from terms with $|q_2(a_1,a_2,a_3)|>Y$. By Lemma \ref{removingsmallq22}, the contribution to $S_1$ from $E_\ell'$ is $O(X^{1-\epsilon+o(1)})$ provided
\begin{equation}\label{eq:YSize}
Y<X^{(1+\alpha_0)/2-4\eta_0-\epsilon}.
\end{equation}  
Therefore we concentrate on the contribution from $E_\ell''$. As in the proof of Lemma \ref{removingsmallq22}, there exists an integer $j$ such that the condition $KA|(\alpha )$ is therefore equivalent to 
\begin{equation}\label{a0KA}
    a_0\equiv {-a_1 j-a_2 j^2-a_3 j^3}\Mod{N_P(AK)}.
    \end{equation}
Let $\tilde a_0=\tilde a_0 (a_1,a_2,a_3 ;KA)$ be a solution of the congruence \eqref{a0KA}. We may write 
$a_0 =\tilde a_0 +mN_P (KA)$ with $m\in\cR' (a_1,a_2,a_3)$
where
\begin{equation*}
    \cR' (a_1,a_2,a_3):=\{ m : (\tilde a_0 +mN_P(KA), a_1,a_2,a_3)\in \cR\}.
    \end{equation*}
(We recall that $\cR$ is the domain defined in \eqref{defRP}.)
This set $\cR' (a_1,a_2,a_3)$ can be written as a finite union  of intervals $I'(a_1,a_2,a_3)$. 

Any $a_0$ of the above form ensures that conditions $(C2)$ and $(C3)$ are satisfied. Conditions (C1), (C4) and (C5) parts (a),(c),(e) don't depend on $a_0$. Thus we find
\[
    E_\ell ''(X,h;KA)\ll \sum_{\substack{a_1,a_2,a_3\ll X^{(1+\alpha_0)/4} \\ q_2(a_1a_2,a_3)>Y \\ \eqref{qij}}}\Bigl|\sum_{\substack{m\in I'(a_1,a_2,a_3) \\ (N_P(\alpha),q_0)=(q,B_{14})=1}}\e \Big ( \frac{ h\ell X}{N_P(\alpha )}-\frac{hU\ov{B_{14}}}{q}\Big )\Bigr|.
   \]
Here by $\sum_{\eqref{qij} }$ we mean that the summation is constrained by the factorisation condition \eqref{qij} .

We now need to control the gcd between $N_P (KA)$ and $q$. We define $t=(N_P(KA),q)$ and $t'=q/t$. Since $q$ is squarefree, $(t,t')=1$. We apply Bezout formula \eqref{Bezout}
to separate the congruence in $t$ and in $t'$ 
and use partial summation to remove the factor
$e(h\ell X/N_P(\alpha))$.
This gives  for $\ell=1,2$, (as in \cite[p. 239]{B15})
\begin{equation}\label{El}
    E_\ell ''(X,h;KA)\ll X^{2\eta_0+\alpha_0/4}\sum_{\substack{ (a_1,a_2,a_3)\in\cC  \\  q_2(a_1,a_2,a_3)>Y \\ \eqref{qij}}}
    \max_{B\ll \frac{X^{\frac{1+\alpha_0}{4}}}{N_P(KA)}}
    \Big |\sum_{\substack{{m\le B}\\ {(g(m),t')=1}}}\e \Big ( \frac{h\bar t f(m)\ov{g(m)}}{t'}\Big )\Big |,
\end{equation}
where $\cC$ is the projection of $\cR$ onto the final 3 coordinates and
\[
f(m):=U(\tilde a_0+mN_P(KA)),\quad  g(m):=B_{14} (\tilde a_0+mN_P (KA)).
\]
We recall from \eqref{qij} that for all $a_1,a_2,a_3$ under consideration $q(a_1,a_2,a_3)$ factors as $\prod_{i=1}^7 q_{1i}\prod_{j=1}^2q_{2j}$ for some integers $q_{ij}$ of constrained sizes. We now wish to apply Lemma \ref{HB}, which requires that for all $a_1,a_2,a_3$ under consideration and all $p|q(a_1,a_2,a_3)$, there is no polynomial $w(X)\in\Z [X]$ of degree less than $10$ such that $f(X)\equiv {w(X)g(X)}\Mod p$. 

Let $p|q(a_1,a_2,a_3)$. By (C5)(e), $a_3$ is coprime with $p$, and so by \eqref{eq:B14}, $B_{14}\Mod{p}$ is a polynomial of degree exactly two in $a_0$ since its lead coefficient is $-a_3$. By Lemma \ref{qU0U1} , $(p,U_0(a_1,a_2,a_3),U_1(a_1,a_2,a_3))=1$, and so $U(a_0,a_1,a_2,a_3)\Mod p$ is not identically zero and has degree at most $1$ in $a_0$. This implies that  for all $p|q$, there is no polynomial $w\in\Z [X]$ such that 
$U(X,a_1,a_2,a_3)\equiv w(X)B_{14}(X,a_1,a_2,a_3)\Mod p$ and we can apply Lemma $\ref{HB}$ with $k=8$.
We take $r_0=q_{22}$, $r_1=q_{21}$, $r_2=q_{17}$, $r_3=q_{16}$, \dots, $r_8=q_{11}$. 
By \eqref{qij} and \eqref{eq:Con5}, we observe that $q_{17}<q_{1j}$ for all $1\le j\le 6$.
Let
\begin{equation}
\label{bigtheta11}
\theta_{max}+\tau_{max}=\sup_{(i,j)\in I_\cC}(\theta_{ij}+\tau_{ij}),
\end{equation}
where we recall from \eqref{defIC} that $$I_\cC:=\{ (1,1),(1,2),(1,3),(1,4),(1,5),(1,6), (2,1)\}.$$ 
Then the sum over $m$ in \eqref{El} is bounded by 
\[ \begin{split}
\Big |\sum_{\substack{ m\le B \\ (g(m),t')=1 }}&\e \Big ( \frac{h\bar t f(m)\ov{g(m)}}{t'}\Big )\Big|\\
&\ll q^\varepsilon B\Big ( \Big (\frac{(h,q)(q_{22},t)}{q_{22}}\Big )^{1/2^9}+\Big ( \frac{q_{22}}{B^2}\Big )^{1/2^9}+\sup_{(i,j)\in I_\cC}\Big (\frac{q_{ij}}{B}\Big)^{1/2^8}\Big).
\end{split}
\]
We insert this bound into $E_\ell '' (X,h ; KA)$, and then subsitute this into $S_1$. Writing $q_{22}=X^{\theta_{22}}$, this gives
\[\begin{split}
S_1 &\ll X^{2\eta_0+1+\frac{5\alpha_0}{4}+\varepsilon}
\big ( X^{-\theta_{22}2^{-9}}+X^{(\theta_{22}+\tau_{22}-\frac{1+\alpha_0}{2}+6\theta_0+10\alpha_0)2^{-9}}\\
&+X^{(\theta_{max}+\tau_{max}-\frac{1+\alpha_0}{4}+3\theta_0+5\alpha_0)2^{-8}}\big )+X^{1-\epsilon+o(1)}.\end{split}
\]
Thus we see that $S_1=o(X)$ provided
\begin{align*}
 2\eta_0+\frac{5\alpha_0}{4}&<\frac{\theta_{22}}{2^9}\\
\frac{\theta_{22} }{2^9}&<\frac{1}{2^9}\Big ( \frac{1+\alpha_0}{2}-6\theta_0-10\alpha_0\Big )-2\eta_0-\frac{5\alpha_0}{4}\\
  \frac{\theta_{max}+\tau_{max}}{2^8}&<\frac{1}{2^8}\Big ( \frac{1+\alpha_0}{4}-3\theta_0-5\alpha_0\Big )-2\eta_0-\frac{5\alpha_0}{4}.
\end{align*}
We recall that $q_{22}=q_2(a_1,a_2,a_3)/q_{21}$, that $q_2(a_1,a_2,a_3)\in[Y,X^{(1+\alpha_0)/2}]$ and $q_{21}\in [X^{\theta_{21}},X^{\theta_{21}+\tau_{21}}]$. Thus on choosing $Y=X^{(1+\alpha_0)/4-4\eta_0-\epsilon}$ so \eqref{eq:YSize} is satisfied, we see that the bound $S_1=o(X)$ holds provided
\begin{align*}
 2\eta_0+\frac{5\alpha_0}{4}&<\Bigl(\frac{1+\alpha_0}{4}-\theta_{21}-\tau_{21}-4\eta_0\Bigr) \frac{1}{2^9}\\
\frac{1}{2^9}\Bigl(
\frac{1+\alpha_0}{4}-\theta_{21}
\Bigr)
&<\frac{1}{2^9}\Big ( \frac{1+\alpha_0}{2}-6\theta_0-10\alpha_0\Big )-2\eta_0-\frac{5\alpha_0}{4}\\
  \frac{\theta_{max}+\tau_{max}}{2^8}&<\frac{1}{2^8}\Big ( \frac{1+\alpha_0}{4}-3\theta_0-5\alpha_0\Big )-2\eta_0-\frac{5\alpha_0}{4}.
\end{align*}
These follow from \eqref{eq:Con1}, \eqref{eq:Con8} and \eqref{eq:Con13} on taking $\theta_0$ sufficiently small and $\eta_0$ sufficiently close to $\alpha_0$.
\end{proof}
%
%
%
%
\section{Proof of Proposition \ref{prpstn:S0}: The sum $S_0$}\label{sec:S0}
%
%
In this section we estimate the sum $S_0$ from \eqref{defS0S1} and establish Proposition \ref{prpstn:S0} all under the assumption of Theorem \ref{DistriNorm}.

%
%
\subsection{The variable $a_0$ in  $S_0$}
%
%
 With the notation $\alpha =a_0+a_1r_1+a_2r_1^2+a_3r_1^3$, we consider
the subset $\cR\in\R^4$ is defined by \eqref{defRP}.

For $S_0$ we proceed in the same way as in \cite{B15}, \cite{D15}, \cite{HB01} but with slight differences in some steps where a bound in $O(X^\varepsilon)$  is not always sufficient.  

%
%

\begin{lmm}[Removing the variable $a_0$]\label{lmm:S01}
Let $12\theta_0+ 22\alpha _0 <1$. We have
\begin{equation*}
    S_0= \Bigl(\frac{4\e^\gamma}{3}\log(5/4)\log 2+o(1)\Bigr)\prod_{p<X^\theta_0}
    \Big ( 1-\frac{g(p)}{p}\Big )S_{01}+O(X^{-\alpha_0/5}),
\end{equation*}
where 
\begin{align}
S_{01}&:=\sum_{(a_1,a_2,a_3)\in\cC\cap\cG}
    I(a_1,a_2,a_3)h(q(a_1,a_2,a_3)),\\
g(p)&:=|\{ \gP : N_P (\gP )=p\}|,\\
    \cC&:=\{(a_1,a_2,a_3)\in\R^3:\exists a_0\in \R\text{ s.t. }(a_0,a_1,a_2,a_3)\in\cR\},\label{eq:CRDef}\\
    \cG&:=\{(a_1,a_2,a_3)\in \ZZ^3:\, \exists a_0\in \Z\text{ s.t. }(\alpha)\in\cJ\},\label{eq:GDef}\\
 h(q)&:=\mu ^2 (q)\prod_{p|q}\frac{ (1-2/p)}{1-g(p)/p}\1_{P^-(q)>q_0},
 \label{eq:HDef}\\
 I(a_1,a_2,a_3)&:=\int_{a_0\in\cD (a_1,a_2,a_3)}\frac{\d a_0}{\widetilde{N}_P(a_0,a_1,a_2,a_3)},\\
 \cD(a_1,a_2,a_3) &:=\{ a_0\in\R :(a_0,a_1,a_2,a_3)\in\cR\}.\label{eq:DDef}
 \end{align}
 Here $\widetilde{N}_P(a_0,a_1,a_2,a_3)$ is the quartic form coinciding with $N_P(a_0+a_1r_1+a_2r_1^2+a_3r_1^3)$ on integers.
\end{lmm}

%
%
\begin{proof}
We want to isolate the variable $a_0$. We note that the condition $(\alpha)\in\cJ$ implies that $(q(a_1,a_2,a_3),B_{14}(a_0,a_1,a_2,a_3))=1$ and that $(a_0,a_1,a_2,a_3)\in \cR$ but otherwise there are no further dependencies between $a_0$ and $a_1,a_2,a_3$. We use Möbius inversion to detect the condition $(q,B_{14})=1$ when evaluated at $a_0,a_1,a_2,a_3$. This give rise to a  squarefree $r|(q,B_{14})$ which we decompose as $r=r_1'r_2'$
with $r_1'|N_P(KA)$ and $(r_2',N_P(KA))=1$. This yields
\begin{equation}\begin{split}
    S_0&=\sum_{K\in\cK}\sum_{A}\lambda_{N_P(A)}^-
    \sum_{(a_1,a_2,a_3)\in \cC\cap\cG}\sum_{\substack{{r'_1|N_P(KA)}\\ {r'_1|q(a_1,a_2,a_3)}}}\mu (r'_1)\\
    &\times\sum_{\substack{{r_2'|q(a_1,a_2,a_3)}\\{(r_2',N_P(KA))=1}}} \mu (r_2')
    \sum_{\tilde a_0\in S(r_1',r_2')}
    \sum_{\substack{{a_0\in\cD (a_1,a_2,a_3)}\\ {a_0\equiv \tilde a_0\Mod{r_2'N_P(KA)}}}}\frac{1}{N_P(\alpha )},
    \end{split}
    \label{eq:S0eq}
\end{equation}
where $\cC$, $\cG$ are as in \eqref{eq:CRDef} and \eqref{eq:GDef}
\begin{equation}
    S(r_1',r_2'):=\{ 0\le a_0\le  r_2'N_P(KA): r_1'r_2'| B_{14}(a_0,a_1,a_2,a_3),\, KA | (\alpha)\}.\label{eq:Sr1r2}
\end{equation}
(We have suppressed the dependence of $S(r_1',r_2')$ on $a_1,a_2,a_3$ for notational convenience.) The inner sum over $a_0$ is now over points in an interval with a congruence constraint, and so by partial summation (and recalling from \eqref{defRP} that $N_P(\alpha)\gg X^{1+\alpha_0/2}$ for all $\aa\in \cR$), we obtain
\begin{equation}
\sum_{\substack{{a_0\in\cD (a_1,a_2,a_3)}\\ {a_0\equiv \tilde a_0\Mod{r_2'N_P(KA)}}}}\frac{1}{N_P(\alpha )}=\frac{I (a_1,a_2,a_3)}{r_2'N_P(AK)}+O\Bigl(\frac{1}{X^{1+\alpha_0/2}}\Bigr).
\label{eq:a0sum}
\end{equation}
The $O(X^{-(1+\alpha_0/2)})$ error term in \eqref{eq:a0sum} contributes to $S_0$ a total
\[
\ll \frac{1}{X^{1+\alpha_0/2-o(1)}}\sum_{N_P(K)\ll X^{5\alpha_0}}\sum_{N_P(A)\le X^{3\theta_0}}\sum_{(a_1,a_2,a_3)\in\cC}1\ll X^{-1/4 +3\theta_0+21\alpha_0/4+o(1)}.
\]
(Recall that if $\aa\in\cR$ then $\|\aa\|_\infty \ll X^{(1+\alpha_0)/4}$ by our choice of fundamental domain). This is $O(X^{-\alpha_0/4+o(1)})$ if $12\theta_0+22\alpha_0<1$, as in the assumptions of the lemma.

Thus we are left to consider the contribution from the main term of \eqref{eq:a0sum}, namely
\begin{equation}
    \sum_{(a_1,a_2,a_3)\in \cC}\sum_{K\in\cK}\sum_{A}\lambda_{N_P(A)}^-\sum_{\substack{ r'_1r_2'|q(a_1,a_2,a_3)\\ r'_1|N_P(KA)\\ (r_2',N_P(KA))=1}} \mu(r_1')\mu (r_2')\frac{|S(r_1',r_2')|I (a_1,a_2,a_3)}{r_2'N_P(AK)}.
\label{eq:S0Main}
\end{equation}
 By the Chinese Remainder Theorem, we have 
\begin{equation}
|S(r_1',r_2')|=\prod_{p| r_2'N_P(KA)}|S(r_1',r_2',p)|,
\label{eq:CRT}
\end{equation}
where
\begin{equation*}
    |S(r_1',r_2',p)|:=
    \begin{cases} |\{ 0\le a_0<p : p| (B_{14}(a_0,a_1,a_2,a_3), N_P(\alpha ))\}|, & \text{if}\ p|r_1',\\
    |\{ 0\le a_0 <p : p| B_{14}(a_0,a_1,a_2,a_3) \} |, & \text{if}\ p|r_2',\\
    |\{ 0\le a_0 <p : p|N_P (\alpha )\}|, & \text{if}\ p|N_P(KA)/r_1'.\end{cases}
\end{equation*}
We compute $|S(r_1',r_2',p)|$ using Lemmas \ref{B14p} and
\ref{B13B14N}. Under the condition $P^-(q)>q_0$ we find
\begin{equation*}
    |S(r_1',r_2',p)|=\begin{cases}
    2 & \text{if}\ p|r_2', \\
     1 & \text{if}\ p|N_P(KA).  
    \end{cases}
\end{equation*}
Using this bound in \eqref{eq:CRT} gives
\begin{equation*}
    |S(r_1',r_2')|=2^{\omega (r_2')}.
\end{equation*}
Inserting this in the previous expression \eqref{eq:S0Main} for the main term of $S_0$, we see that the sum over $r_1'$ is $1$ if $(q(a_1,a_2,a_3),N_P(KA))=1$, and $0$ otherwise. Thus the expression \eqref{eq:S0Main} simplifies to
\[
\sum_{(a_1,a_2,a_3)\in \cC}I (a_1,a_2,a_3)\Bigl(\sum_{r_2'|q(a_1,a_2,a_3)} \frac{\mu (r_2')2^{\omega(r_2')}}{r_2'}\Bigr)h_1(q(a_1,a_2,a_3)),
\]
where
\[
h_1(q):=\Biggl(\sum_{\substack{K\in\cK\\ (N_P(K),q)=1}}\frac{1}{N_P(K)}\Biggr)\Biggl(\sum_{(N_P(A),q)=1}\frac{\lambda_{N_P(A)}^- }{N_P(A)}\Biggr).
\]
Recalling that $\cK$ is the set of prime ideals with norm between $X^{4\alpha_0}$ and $X^{5\alpha_0}$, we see that for $q\ll X^{O(1)}$
\begin{align*}
\sum_{\substack{K\in \cK\\ (N_P(K),q)=1}}\frac{1}{N_P(K)}&=\log(5/4)+o(1),\\
\sum_{\substack{(N_P(A),q)=1}}\frac{\lambda^-_{N_P(A)}}{N_P(A)}&=\sum_{\substack{d\le X^{3\theta_0}\\ (d,q)=1}}\frac{\lambda^-_{d}g(d)}{d}\\
&=\Bigl(\frac{2 e^\gamma\log{2}}{3}+o(1)\Bigr)\prod_{p<X^{\theta_0}}\Bigl(1-\frac{g(p)}{p}\Bigr)\prod_{\substack{p|q\\ p\le X^{\theta_0}}}\Bigl(1-\frac{g(p)}{p}\Bigr)^{-1}.
\end{align*}
Here we used the fact that the linear sieve lower bound function evaluated at 3 is $2 e^\gamma\log{2}/3$. Putting these expressions together now gives the result.
\end{proof}
%
%
\subsection{Splitting into small boxes}
%
%
We see from condition (C2) that if $\af\in\cJ$ then $\af=(a_0+a_1r_1+a_2r_1^2+a_3r_1^3)$ for some $\a\in\Z^4$ which lies in the region $\cR$ given by \eqref{defRP}. We recall that $\eta_1 =(\log x)^{-100}$. We cover the region $\cR$ by hyper-rectangles of type
\begin{equation}\label{HyperH}
 \cH =]A_0,A_0+\eta_1 A_0]\times]A_1,A_1(1+\eta_1)]\times]A_2,A_2(1+\eta _1)]\times]A_3,A_3(1+\eta _1)].
 \end{equation}
 The number of such hyper-rectangles is $O(\eta_1^{-4})(\log X)^4=O(\eta_1^{-5})$.
 
 Furthermore the contribution to $S_{01}$ from hyper-rectangles such that $\min (|A_i|)\le X^{1/4-7\alpha_0/8}$ 
 is $O(X^{1-\alpha_0/8+\varepsilon})$ which is sufficiently small.
 
We will say that $\cH$ is a `good' hyper-rectangle if $\cH\subset\cR$ and
\begin{equation}\label{goodcube}
\begin{split}
\min (|A_0|,|A_1|,|A_2|,|A_3|)&\ge X^{1/4-7\alpha_0/8}, \\
\min (|A_0|,|A_1|,|A_2|,|A_3|)&\ge\eta_1\max (|A_0|,|A_1|, |A_2|, |A_3|),\\
q_1 (A_1,A_2,A_3)&\ge \eta_1^{1/10} \max (|A_1|,|A_2|,|A_3|)^4.
\end{split}
\end{equation}
If $\cH$ is not `good' then we say $\cH$ is `bad'. We note that the second and third assertions in this definition corresponds to the 
conditions \eqref{eq:Xi1} and \eqref{eq:Xi2}. 

We denote by $\sH_\cR$ the set of all good hyper-rectangles. To each hyper-rectangle $\cH$ we associate its projection to $\R^3$ by ignoring $a_0$:
\begin{equation}\label{HyperH'}
 \cH' =]A_1,A_1(1+\eta_1)]\times]A_2,A_2(1+\eta_1)]\times]A_3,A_3(1+\eta_1)].
 \end{equation}
%
%
\begin{lmm}[Splitting into small boxes]\label{lmm:S02}
Let $S_{01}$ be as in Lemma \ref{lmm:S01}. We have that 
\[
S_{01}\gg \sum_{\cH\in\sH_\cR}\frac{A_0 \eta_1}{\widetilde{N}_P(A_0,A_1,A_2,A_3)}S_{02}(\cH),
\]
where
\[
S_{02}(\cH):=\sum_{\substack{{q_{ij}\in [X^{\theta_{ij}},X^{\theta_{ij}+\tau_{ij}}]}\\
(i,j)\in I_\cC\\ {q_{21}\equiv 1\Mod{D_{q_2}}}}}\sum_{\substack{{(a_1,a_2,a_3)\in\cH'}\\ {\prod_{j=1}^6 q_{1j}|q_1(a_1,a_2,a_3)}\\ {q_{21}|q_2(a_1,a_2,a_3)}\\
{ (q(a_1,a_2,a_3),q_3(a_1,a_2,a_3))=1}\\ {(q,a_2a_3)=1}\\
{(a_2,a_3)=30,\ a_1\equiv 1\Mod{30}}\\
a_2,a_3\equiv 30\Mod{900}
}}h(q(a_1,a_2,a_3)).
\]
\end{lmm}
%
%
We recall from \eqref{eq:RealNDef} that $\widetilde{N}_P(a_0,a_1,a_2,a_3)$ is the quartic form coinciding with $N_P(a_0+a_1r_1+a_2r_1^2+a_3r_1^3)$ on integers.
%
%
\begin{proof}
By splitting the sum over $a_1,a_2,a_3$ and the integral over $a_0$ into the hyperrectangles $\cH$, and then restricting only to good hyperrectangles for a lower bound, we find
\begin{equation*}
S_{01}\ge \sum_{\cH\in\sH_{\cR}}S_{01}'(\cH),
\end{equation*}
where 
\begin{align*}
S_{01}'(\cH)&:=\sum_{(a_1,a_2,a_3)\in\cC\cap\cH'\cap \cG}h(q(a_1,a_2,a_3))I_\cH (a_1,a_2,a_3),\\
I_\cH (a_1,a_2,a_3)&:=\int_{A_0}^{A_0(1+\eta_1)}
\frac{\d a_0}{\widetilde{N}_P(a_0,a_1,a_2,a_3)}=\frac{A_0\eta_1(1+o(1))}{\widetilde{N}_P(A_0,A_1,A_2,A_3)}.
\end{align*}
We recall from \eqref{qij} that if $(a_1,a_2,a_3)\in \cG$ then $q_1(a_1,a_2,a_3)$ and $q_2(a_1,a_2,a_3)$ factor as $\prod_{i=1}^6 q_{1i}$ and $q_{21}q_{22}$ respectively with $q_{21},q_{11},q_{12},q_{13},q_{14},q_{15}$ primes satisfying $q_{ij}\ge X^{\theta_{ij}}$. In particular, we see that for any choice of $a_1,a_2,a_3$ there are $O(1)$ choices of $q_{ij}$ such that $q_i(a_1,a_2,a_2)=\prod_{j}q_{ij}$. Thus, summing over these representations, we find
\begin{equation*}\begin{split}
S_{01}(\cH)&\gg \frac{A_0\eta_1}{\widetilde{N}_P(A_0,A_1,A_2,A_3)}S_{02}(\cH),
\end{split}
\end{equation*}
say, with $S_{02}(\cH)$ as given by the lemma and $I_\cC$ defined in \eqref{defIC}. This gives the result.
\end{proof}
%
%
%
\subsection{Preparation for the application of Theorem \ref{DistriNorm}}
\label{Prep}%
%
Following \cite[Section 6.2]{B15} or \cite[Section 15]{D15}, we do several manipulations in order to take care of the different coprimality conditions and the  multiplicative weight $h(q)$. In our situation it is important that we are slightly more careful than these previous works. We do not impose congruence conditions to moduli larger than $(\log{X})^{O(1)}$ since this would cause issues related to Siegel zeros (the argument of the previous papers would introduce a congruence constraint of modulus $X^{t_0}$ for some $t_0>0$). This means we need to be careful not to lose the fact that when $(a_0,a_1,a_2,a_3)\in\cH$, the $a_i$ are in small intervals. 
%
%
 Let 
\begin{equation}
  Z:=(\log X)^{\lambda_0}, \qquad Z':=X^{\alpha_0/10000},
 \label{eq:ZDef}
 \end{equation}
 where $\alpha_0$ is the constant used to define the set $\cK$ (which will be chosen sufficiently small later on) and $\lambda_0$ is a fixed constant (which will be chosen sufficiently large). From the bound \eqref{goodcube}, we certainly note that since $\alpha_0<1$ we have
 \begin{equation}
 Z^{1000}<Z'^{100}<\min(A_0,A_1,A_2,A_3).\label{eq:ZSize}
 \end{equation}
 %
%
For brevity we will write
\begin{equation}\label{defNH}
N_\cH =\widetilde{N}_P(A_0,A_1,A_2,A_3).
\end{equation}
 \begin{lmm}[Removing the condition $(q,q_3)=1$]\label{lmm:S03}
 Let $S_{02}(\cH)$ be as in Lemma \ref{lmm:S02}. Then we have
 \[
 S_{02}(\cH)=S_{03}(\cH)+O\Bigl(\frac{ \eta_1^3  A_1A_2A_3}{Z^{3/4}}\Bigr),
 \]
 where
 \[
      S_{03}(\cH):=\sum_{d\le Z}\mu_(d)\sum_{\substack{{q_{ij}\in [X^{\theta_{ij}},X^{\theta_{ij}+\tau_{ij}}]}\\ (i,j)\in I_\cC\\ q_{21}\equiv 1\Mod{D_{q_2}}}}\sum_{\substack{{(a_1,a_2,a_3)\in\cH'}\\ {\prod_{j=1}^6q_{1j}|q_1(a_1,a_2,a_3)}\\ {q_{21}|q_2(a_1,a_2,a_3)}\\ d|q(a_1,a_2,a_3)\\ d|q_3(a_1,a_2,a_3) \\
{ (q(a_1,a_2,a_3),q_3(a_1,a_2,a_3))=1}\\ {(q,a_2a_3)=1}\\
(a_2,a_3)=30,\ a_1\equiv 1\Mod{30}\\ a_2,a_3\equiv 30\Mod{900}}}h(q(a_1,a_2,a_3)).
 \]
 \end{lmm}
%
%
\begin{proof}
 First, we detect the condition $(q,q_3)=1$ via Möbius inversion
 \begin{equation*}
     S_{02}(\cH)=\sum_{\substack{{q_{ij}\in [X^{\theta_{ij}},X^{\theta_{ij}+\tau_{ij}}]}\\ (i,j)\in I_\cC \\ q_{21}\equiv 1\Mod{D_{q_2}} }}
\sum_{\substack{{(a_1,a_2,a_3)\in\cH'}\\ {\prod_{j=1}^6q_{1j}|q_1(a_1,a_2,a_3)}\\ {q_{21}|q_2(a_1,a_2,a_3)}\\
{ (q(a_1,a_2,a_3),q_3(a_1,a_2,a_3))=1}\\ {(q,a_2a_3)=1}\\
{(a_2,a_3)=30,\ a_1\equiv 1\Mod{30}}\\
{a_2,a_3\equiv 30\Mod{900}}}}h(q(a_1,a_2,a_3))\sum_{\substack{{d|q(a_1,a_2,a_3)}\\ {d|q_3(a_1,a_2,a_3)}}}\mu (d).
 \end{equation*}
 
 We split $S_{02}(\cH)$ into three sums,  
 \[
 S_{02} (\cH)=S_{03}(\cH)+U_{21}(\cH)+U_{22}(\cH),
 \]
where $S_{03} (\cH)$ is the contribution of the terms in $S_{02}(\cH)$ with $d\le Z$, $U_{21}(\cH)$ is the contribution from $Z<d\le Z'$ and $U_{22}(\cH)$ is the contribution from $d>Z'$. We note that $S_{03}(\cH)$ is as given in the lemma, so we are left to bound $U_{21}(\cH)$ and $U_{22}(\cH)$.
 
First we  bound $U_{21}$.  Recall that $q_3(a_1,a_2,a_3)=-a_1a_3+a_2^2-c_3a_2a_3-c_2a_3^2$, so the condition  $q_3\equiv 0\Mod d$ implies that 
 $a_1\equiv \ov{a_3} (a_2^2-c_3a_2a_3-c_2a_3^2)\Mod d$. (We restrict ourselves to $(a_3,q(a_1,a_2,a_3))=1$ so $(a_3,d)=1$.) Inserting
 this into the condition $q(a_1,a_2,a_3)\equiv 0\Mod d$ and multiplying by $a_3^6$ gives
 $Q(a_2,a_3):=q(a_2^2-c_3a_2a_3-c_2a_3^2,a_2a_3,a_3^2)\equiv 0\Mod d$, for a polynomial $Q(a_2,a_3)$ which is of degree $12$ in $a_2$ (and non-zero). For any given $a_3$
 the number of roots of $Q(a_2,a_3)\Mod{d}$ is $O(12^{\omega (d)})$.  For any choice of $a_1,a_2,a_3$ under consideration, there are $O(1)$ choices of primes $q_{ij}\in [X^{\theta _{ij}},X^{\theta_{ij}+\tau_{ij}}]$ with $q_{ij}|q_1(a_1,a_2,a_3)q_2(a_1,a_2,a_3)$. Thus, letting $b(a_2,a_3)=\ov{a_3} (a_2^2-c_3a_2a_3-c_2a_3^2)$, and noting $Z'<A_i^{0.99}$ (recall \eqref{eq:ZSize}), we deduce
 \begin{align*}
 U_{21}(\cH)& \ll\sum_{Z<d\le Z'}\sum_{a_3\in [A_3,A_3(1+\eta_1)]}\sum_{\substack{a_2\in [A_2,A_2(1+\eta_1)]\\ Q(a_2,a_3)\equiv 0\Mod d}}\sum_{\substack{a_1\in [A_1,A_1(1+\eta_1)]\\ a_1\equiv  b(a_2,a_3)\Mod d}}1\\
&\ll A_1A_2A_3\eta_1^3\sum_{Z<d<Z'}\frac{12^{\omega (d)}}{d^2}\ll A_1A_2A_3\eta_1^3Z^{-3/4}.
 \end{align*}
 
 We now consider $U_{22}$. Since $Q(a_2,a_3)\equiv 0\Mod{d}$, if $Q(a_2,a_3)\ne 0$ there are $O(X^\epsilon)$ choices of $d$ given $a_2,a_3$.
 We have $Q(a_2,a_3)=0$ if and only if $\exists (i,j)$ such that 
 \begin{equation*}
     (a_2^2-c_3a_2a_3-c_2a_3^2)+(r_i+r_j)a_2a_3+a_3^2 (r_i^2+r_ir_j+r_j^2)=0,
 \end{equation*}
which rearranges to
 \begin{equation*}
     a_2^2+ a_2a_3 (r_i+r_j-c_3)+a_3^2 (r_i^2+r_ir_j+r_j^2-c_2)=0.
 \end{equation*}
 Since $a_3\not =0$, $a_2/a_3$ is a root of $X^2+(r_i+r_j-c_3)X+r_i^2+r_ir_j+r_j^2-c_2$ and there are at most two such roots. Thus for each choice of $a_2$ there are at most 2 choices of $a_3$ such that $Q(a_2,a_3)=0$. Moreover, in this case we still have $d|q_3(a_1,a_2,a_3)\ne0$, so there are $O(X^\epsilon)$ choices of $d$ given $a_1,a_2,a_3$. We deduce that (using $Z'\ll A_1,A_3$)
 \begin{align*}
        U_{22} (\cH)&\ll \sum_{d>Z'}\mu^2 (d)\sum_{\substack{(a_1,a_2,a_3)\in\cH'\\ Q(a_2,a_3)\ne 0\\ a_1\equiv b(a_2,a_3)\Mod d}}1+\sum_{d>Z'}\sum_{\substack{(a_1,a_2,a_3)\in\cH'\\ Q(a_2,a_3)= 0 \\ d|q_3(a_1,a_2,a_3)}}\mu ^2 (d)\\        
        &\ll \sum_{a_2\in [A_2,A_2 (1+\eta_1)]}\sum_{\substack{a_3\in  [A_3,A_3(1+\eta_1)] \\ Q(a_2,a_3)\ne 0}}\sum_{\substack{d>Z'\\ d|Q(a_2,a_3)\\ \mu^2 (d)=1}}\sum_{\substack{a_1\in[A_1, A_1 (1+\eta_1)]\\ a_1\equiv b(a_2,a_3)\Mod d}}1\\
        &\qquad +\sum_{a_1\ll A_1, a_2\ll A_2}\sum_{\substack{0<a_3\ll A_3\\ Q(a_2,a_3)=0}}\sum_{d|q_3(a_1,a_2,a_3)}1\\
         &\ll \frac{A_1}{Z'}\sum_{\substack{a_2\in  [A_2,A_2(1+\eta_1)]\\ a_3\in  [A_3,A_3(1+\eta_1)]\\
         Q(a_2,a_3)\not =0}}
         \tau (Q(a_2,a_3))+A_1A_2X^\varepsilon\ll \frac{A_1A_2A_3X^\varepsilon}{Z'}.
 \end{align*}
 This gives the result.
 \end{proof}
 %
%
\begin{lmm}[Removing the condition $(q,a_2a_3)=1$]\label{lmm:S04}
Let $S_{03}(\cH)$ be as given in Lemma \ref{lmm:S03}. Then we have
\[
S_{03}(\cH)=S_{04}(\cH)+O\Bigl(\frac{\eta_1^3  A_1 A_2 A_3}{Z^{1/2}} \Bigr),
\]
where 
\[
S_{04}(\cH):=\sum_{\substack{d\le Z\\ s_2s_3\le Z}}\mu  (d)\mu (s_2s_3)\sum_{\substack{{q_{ij}\in [X^{\theta_{ij}},X^{\theta_{ij}+\tau_{ij}}]}\\
\forall\,(i,j)\in I_\cC\\ q_{21}\equiv 1\Mod{D_{q_2}}}}\sum_{\substack{{(a_1,s_2a_2',s_3a_3')\in\cH'}\\ {\prod_{j=1}^6q_{1j}|q_1(a_1,s_2a_2',s_3a_3')}\\ {q_{21}|q_2(a_1,s_2a_2',s_3a_3')}\\
{[d,s_2s_3]|q(a_1,s_2a_2',s_3a_3')}\\ d|q_3(a_1,a_2,a_3)\\
{(s_2a_2',s_3a_3')=30,\ a_1\equiv 1\Mod {30}}\\
s_2a_2',s_3a_3'\equiv 30\Mod{900}}}h(q(a_1,s_2a_2',s_3a_3')).
\]
\end{lmm}
 %
%
\begin{proof}
We remove the condition $(a_2a_3,q(a_1,a_2,a_3))=1$ via Mobius inversion, giving
 \begin{equation*}
     S_{03}(\cH )=\sum_{d\le Z}\mu (d)\sum_{\substack{{q_{ij}\in [X^{\theta_{ij}},X^{\theta_{ij}+\tau_{ij}}]}\\
(i,j)\in I_\cC\\ q_{21}\equiv 1\Mod{D_{q_2}}}}\sum_{\substack{{(a_1,a_2,a_3)\in\cH'}\\ {\prod_{j=1}^6q_{1j}|q_1(a_1,a_2,a_3)}\\ {q_{21}|q_2(a_1,a_2,a_3)}\\
{(a_2,a_3)=30,\ a_1\equiv 1\Mod {30}}\\
{a_2,a_3\equiv 30\Mod{900}}\\
{d|q_3(a_1,a_2,a_3)}}}\sum_{\substack{{s|a_2a_3}\\ {[d,s]|q(a_1,a_2,a_3)}}}\mu (s)h(q(a_1,a_2,a_3)).
 \end{equation*}
 We write $s$ as $s=s_2s_3$ with $s_2|a_2$ and  $s_3|a_3$, and write $a_2=s_2a_2'$, $a_3=s_3a_3'$.
Let $U_3 (\cH)$ denote the contribution given by the $s>Z$  and $S_{04}(\cH)$
the remaining contribution with $s\le Z$. Thus we are left to bound $U_3(\cH)$.

Since each $q_i(a_1,s_2a_2',s_3a_3')$ has a finite number of prime factors in $[X^{\theta _{ij}},X^{\theta_{ij}+\tau_{ij}}]$, there are $O(1)$ choices of the $q_{ij}$, so
\begin{equation*}
    U_3(\cH)\ll \sum_{d\le Z}\mu^2 (d)\sum_{Z<s_2s_3\ll N_\cH^{1/4}}\mu^2(s_2s_3)
    \sum_{\substack{{(a_1,s_2a_2',s_3a_3')\in\cH'}\\ {[d,s]|q(a_1,a_2's_2,a_3's_3)}\\ {d|q_3(a_1,a_2's_2,a_3's_3)}}}1.
\end{equation*}
The form $q$ is monic of degree $6$ in $a_1$ (by \eqref{defq1}, \eqref{defq2}) and $[d,s_2s_3]$ is squarefree, so given $s_2,s_3,a_2',a_3'$ there are $O(6^{\omega([d,s_2s_3])})$ choices of $a_1\Mod{[d,s_2s_3]}$ such that $q(a_1,a_2's_2,a_3's_3)=0\Mod{[d,s_2s_3]}$.
  Since $(a_1,a_2's_2,a_3's_3)\in\cH'$ we obtain
  \begin{equation*}\begin{split}
      U_3 (\cH)&\ll \sum_{d<Z}\mu^2 (d)\sum_{Z<s_2s_3\ll N_\cH^{1/4}}\mu^2 (s_2s_3)6^{\omega ([s_2s_3,d])}\Big ( \frac{\eta_1A_2}{s_2}+1\Big )\Big (\frac{\eta_1A_3}{s_3}+1\Big )\Big (\frac{\eta_1A_1}{[s_2s_3,d]} +1\Big )\\
      &\ll Z N_\cH^{1/4+\varepsilon}+Z N_\cH^{1/4+\varepsilon}(|A_1|+|A_2|+|A_3|)+
      \frac{A_1A_2A_3\eta_1^2}{\min (A_1,A_2,A_3)}Z X^\varepsilon
      \\
      &+
\eta_1^3 A_1 A_2 A_3\sum_{d<Z}\sum_{s>Z}\frac{6^{\omega([d,s])}}{s[s,d]}.
      \end{split}
  \end{equation*}
  This final term is seen to be $O(\eta_1^3A_1A_2A_3(\log{Z})^{O(1)} /Z)$. Since $\max (A_1,A_2,A_3)\ll N_\cH^{1/4}$ and $Z=(\log{X})^{O(1)}$, this gives
  \[
  U_3(\cH)\ll \frac{\eta_1^3 A_1A_2A_3}{Z^{1/2}}+N_\cH^{1/2+\epsilon}.
  \]
  This gives the result.
  \end{proof}
%
%
%
%
\begin{lmm}[Simplifying the function $h$]\label{lmm:S05}
Let $S_{04}(\cH)$ be as in Lemma \ref{lmm:S04}. Then we have
\[
S_{04}(\cH)=S_{05}(\cH)+O\Bigl(\frac{\eta_1^3 A_1A_2A_3}{Z}\Bigr),
\]
where
\[
S_{05}(\cH):=\sum_{\substack{u\le Z^{20}\\ d\le Z\\ s_2s_3\le Z}}\mu (d)\mu  (s_2s_3)\ell(u)\sum_{\substack{{q_{ij}\in [X^{\theta_{ij}},X^{\theta_{ij}+\tau_{ij}}]}\\
\forall\,(i,j)\in I_\cC\\ q_{21}\equiv 1\Mod{D_{q_2}}}}\sum_{\substack{{(a_1,s_2a_2',s_3a_3')\in\cH'}\\ {\prod_{j=1}^6q_{1j}|q_1(a_1,s_2a_2',s_3a_3')}\\ {q_{21}|q_2(a_1,s_2a_2',s_3a_3')}\\
{[d,s_2s_3,u]|q(a_1,s_2a_2',s_3a_3')}\\ d|q_3(a_1,a_2,a_3)\\
{(s_2a_2',s_3a_3')=30,\ a_1\equiv 1\Mod{30}}\\
s_2a_2',s_3a_3'\equiv {30}\Mod{900}}}1,
\]
and $\ell$ is the multiplicative function defined by
\begin{equation*}
    \ell (p^\nu):=\begin{cases}
    \frac{g(p)-2}{p-g(p)}, & \text{if}\ p>q_0\text{ and }\nu=1,\\
    -1, & \text{if}\ 7\le p\le q_0\text{ and }\nu=1,\\
   -h(p), & \text{if}\ \nu =2,\\
    0, & \text{if}\ \nu\ge 3,
    \end{cases}
\end{equation*}
with $q_0$ given by \eqref{defq0}.
\end{lmm}
%
%
\begin{proof}
Recalling \eqref{eq:HDef}, we see that $h=\1 *\ell$ where $\ell$ is as given by the lemma. In particular,
\[
   h(q(a_1,s_2a_2',s_3a_3'))= \sum_{u|q(a_1,s_2a_2',s_3a_3')}\ell (u).
\]
Since $a_1\equiv 1\Mod{30}$ and $30| (a_2,a_3)$, $(u,30)=1$.
We substitute this into our definition of $S_{04}(\cH)$, and consider separately the contribution $S_{05}(\cH)$ from $u<Z^{20}$ and the contribution $U_4(\cH)$ from $u>Z^{20}$.

Since $\ell (u)=0$ when there exists $p$ such that $p^3|u$, we may write $u=v^2w$ with $\mu^2 (vw)=1$. Since $U_4(\cH)$ has $u>Z^{20}$, it suffices to separately bound the contribution of terms $U_{41}(\cH)$ with $w>Z^{10}$ and the contribution $U_{42}(\cH)$ of terms with $v^2>Z^{10}\ge w$. 

First we bound $U_{41}(\cH)$ with $w>Z^{10}$. Since $q_0>10$, we see that $|\ell (u)|\le 10^{\omega (vw)}/w$. Following an entirely analogous argument to our bound for $U_2(\cH)$ in Lemma \ref{lmm:S03}, we can find that 
\begin{equation*}\begin{split}
   U_{41}(\cH) &\ll Z\sum_{\substack{s_2s_3<Z\\ \mu^2 (s_2s_3)=1}}
    \sum_{\substack{w\ge Z^{10} \\ {wv^2 |q}}}\mu^2 (vw)
    \frac{(60)^{\omega (v w)}}{w}
    \Big (\frac{\eta_1 A_1}{wv^2}+1\Big )\frac{\eta_1^2 A_2A_3}{s_2s_3}\\
    &\ll A_2A_3X^\varepsilon + \eta_1^3 A_1A_2A_3 (\log X)Z^{-3}\ll \frac{A_1A_2A_3}{Z}.
    \end{split}
\end{equation*}
Thus we are left to bound $U_{42}(\cH)$ involving terms with $v\ge Z^5$. We see
\begin{equation*}
    U_{42}(\cH)\le V'(\cH)+\sum_{(i,j)\in I_\cC}V_{ij}(\cH),
\end{equation*}
where $V_{ij}(\cH)$ denotes those terms with $q_{ij}|v$ for some $q_{ij}\in [X^{\theta_{ij}},X^{\theta_{ij}+\tau_{ij}}]$,  and $V'(\cH)$ denotes those terms with $(\prod_{(i,j)\in I_\cC} q_{ij},v)=1$ for all $q_{ij}\in [X^{\theta_{ij}},X^{\theta_{ij}+\tau_{ij}}]$, 
$(i,j)\in I_\cC$.

First we consider $V_{21}(\cH)$. By \eqref{eq:Con5}, we  have $\sum_{i=1}^6\theta_{1i}+\theta_{21}>1+\alpha_0$. We recall $q_1(a_1,a_2a_3)\ll X^{1+\alpha_0}$ for all $(a_1,a_2,a_3)\in\cH$ and that $\prod_{j=1}^6 q_{1j}|q_1(a_1,a_2,a_3)$ with $\prod_{j=1}^6 q_{1j}\gg X^{\sum_{j=1}^6 \theta_{1j}}$. Therefore we must have that $(q_{21},q_1(a_1,a_2,a_3))=1$. 
Since $\alpha_0<1/19$ by \eqref{eq:Con8} and $q_{21}^2\le X^{2\theta_{21}+2\tau_{21}}\le X^{1/4-7\alpha_0/8}\le \min(A_1,A_2,A_3)$ by \eqref{eq:Con13} and
 \eqref{goodcube}, we deduce that 
\begin{equation*}\begin{split}
    V_{21}(\cH )&\ll X^\varepsilon \sum_{q_{21}\in [X^{\theta_{21}}, X^{\theta_{21}+\tau_{21}}]}
    \sum_{\substack{{(a_1,a_2,a_3)\in\cH'}\\ {q_{21}^2|q_2 (a_1,a_2,a_3)}}}1\\
    & \ll X^{-\theta_{21}+\varepsilon}A_1A_2A_3.
    \end{split}
\end{equation*}
We now consider $V_{1j}(\cH)$. As with $V_{21}(\cH)$, we can't have $q_{1j}^2|q_1(a_1,a_2,a_3)$ by size considerations and \eqref{eq:Con5}. Therefore if $q_{1i}^2|q(a_1,a_2,a_3)$ then $q_{1i}| (q_1(a_1,a_2,a_3),q_2(a_1,a_2,a_3))$, and so Lemma \ref{pgcdq1q2} shows that $P((a_2-c_3a_3)\ov{a_3})\equiv 0\Mod{q_{1i}}$. Again, we have that $q_{1j}\le \min(A_1,A_2,A_3)$. Thus we have
\begin{equation*}\begin{split}
    V_{1i}(\cH)&\ll X^\varepsilon\sum_{q_{1i}\in [X^{\theta_{1i}}, X^{\theta_{1i}+\tau_{1i}}]}\sum_{\substack{{(a_2,a_3)\in [A_2,A_2(1+\eta_1)]\times [A_3,A_3(1+\eta_1)]}\\
{P((a_2-c_3a_3)\ov{a_3})\equiv 0\Mod{q_{1i}}}}}
\sum_{\substack{{a_1\in [A_1,A_1(1+\eta_1)]}\\ {q_{1i}| q_1(a_1,a_2,a_3)}}}1\\
&\ll X^{-\theta_{1i}+\varepsilon}A_1A_2A_3.
\end{split}
\end{equation*}
Finally, we are left to bound $V'(\cH)$.
Each $v$ counted in $V'(\cH)$ may factored as $v=v_1v_2v_3$, with 
\begin{equation*}
v_1:=\prod_{\substack{{p|v}\\ {p^2|q_1(a_1,a_2,a_3)}}}p,
\qquad v_2:=\prod_{\substack{{p|v/v_1}\\ {p^2|q_2(a_1,a_2,a_3)}}}p,\qquad v_3:=\frac{v}{v_1v_2}.
\end{equation*}
Since $v$ was squarefree, we see that $v_1,v_2,v_3$ are pairwise coprime and squarefree.

By Lemma \ref{pgcdq1q2} again, $P((a_2-c_3a_3)\ov{a_3})\equiv 0\Mod {v_3}$. In $V'(\cH)$, $v$ is coprime with all the $q_{ij}$, and so for any $a_1,a_2,a_3\in\cH$
\begin{align}
v_1^2v_3&\ll \frac{q_1(a_1,a_2,a_3)}{\prod_{j=1}^5 q_{1j}}\ll X^{1+\alpha_0-\sum_{j=1}^5\theta_{1j}}<\eta_1 A_2,\\
v_2^2v_3&\ll \frac{q_2(a_1,a_2,a_3)}{q_{21}}\ll X^{(1+\alpha_0)/2-\theta_{21}}.
\end{align}
Thus we have
\[
    V'(\cH)\ll \sum_{\substack{{s,d<Z}\\ {\mu^2 (d)\mu^2(s)=1}}}\sum_{w<Z^{10} } \frac{10^{\omega (w)}}{w}\sum_{\substack{v_1 v_2 v_3>Z^5 \\ \mu^2(wv_1v_2v_3)=1}}
    \sum_{\substack{{a_2
\in [A_2,A_2(1+\eta_1)]}\\ {a_3\in [A_3,A_3(1+\eta_1)]}\\
{P((a_2-c_3a_3)\ov{a_3})\equiv 0\Mod {v_3}}}}    
\sum_{\substack{{a_1\in [A_1,A_1(1+\eta_1)]}\\ {v_1^2v_3|q_1(a_1,a_2,a_3)}\\
{v_2^2v_3|q_2(a_1,a_2,a_3)}}}1.
\]
Let $d_1\in\mathbb{Z}[a_2,a_3]$ denote the discriminant of $q_1$ (viewing $q_1$ as a polynomial in $a_1$), and $d_2\in\mathbb{Z}[a_2,a_3]$ denote the discriminant of $q_2$. By Lemma \ref{pgcdq1q2}, we see that the inner sum restricts $a_1$ to one of $O(6^{\omega(v_1v_2v_3)})$ residue classes modulo $v_1^2v_2^2v_3/(v_1,d_1(a_2,a_3))(v_2,d_2(a_2,a_3))$. Thus
\begin{align}\label{suma1}
 \sum_{\substack{{a_1\in [A_1,A_1(1+\eta_1)]} \\ {v_1^2v_3|q_1(a_1,a_2,a_3)}\\
{v_2^2v_3|q_2(a_1,a_2,a_3)} }}1    &\ll 6^{\omega (v_1v_2v_3)}
\Big ( \frac{\eta_1 A_1(v_1,d_1(a_2,a_3))(v_2,d_2(a_2,a_3))}{v_1^2v_2^2v_3}+1\Bigr).
\end{align}

Let $I_1=\{ (1,3), (1,4), (2,3), (2,4)\}$   be the set of the indexes  $(i,j)$ such that $(r_i,r_j)$ is involved in the factorisation of $q_1$.
We note that 
\begin{align*} d_1 (a_2,a_3)&=\prod_{ \substack{(i,j),(k,l)\in I_1\\ (i,j)\not = (k,l)}}(a_2 (r_i+r_j-r_k-r_\ell)+ a_3(r_i^2+r_ir_j+r_j^2-r_k^2-r_\ell^2-r_kr_\ell))\\
d_2 (a_2,a_3)&=-(a_2(r_1+r_2-r_3-r_4)+a_3(r_1^2+r_1r_2+r_2^2-r_3^2-r_3r_4-r_4^2))^2.
\end{align*}

We remark that the coefficient in $a_2^{12}$ in $d_1$ is non zero because we can't have $r_i+r_j -r_k-r_l=0$ for two different $(i,j),(k,\ell ) \in I_1$.
The case $\{ i,j\}\cap \{ k,\ell\}\not =\emptyset$ is clear, the other case was  noticed in Remark (ii) after the proof of Lemma \ref{t1t2}. 

For $d_2$, it may be the case that $r_1+r_2-r_3-r_4=0$. However in this case we can't also have  
$r_1^2+r_1r_2+r_2^2-r_3-r_3r_4-r_4^2=0$ since this would imply that $r_1+r_2=r_3+r_4$ and $r_1r_2=r_3r_4$ which is not possible when the roots of $P$ are distinct. Thus either the coefficient of $a_2$ in $d_2$ is non-zero or the coefficient of $a_3$ is non-zero.

To estimate the sum over $v_1,v_2,v_3,a_2,a_3$  of the terms with $d_1 (a_2,a_3),d_2 (a_2,a_3)$ in 
\eqref{suma1}, we write $w_i  =(v_i , d_i (a_2,a_3))$ for $i=1,2$ and next forget the coprimality between $v_i/w_i$ and $d_i (a_2,a_3)/w_i$. This sum is thus bounded by

\[
\sum_{\substack{v_1v_2v_3\ge Z^5\\ v_1^2v_3\le X^{1+\alpha_0-\sum_{j=1}^6 \theta_{1j}}\\
v_2^2v_3\le X^{(1+\alpha_0)/2 -\theta_{21}}\\ \mu^2 (v_1v_2v_3)=1}}
\frac{6^{\omega (v_1v_2v_3)}}{v_1^2v_2^2v_3}\sum_{\substack{w_1|v_1\\ w_2|v_2}}w_1w_2\sum_{\substack{a_2\in [A_2,A_2 (1+\eta_1)]\\ a_3\in [A_3,A_3 (1+\eta _1)\\ d_1 (a_2,a_3)\equiv 0\Mod {w_1}\\ d_2 (a_2,a_3)\equiv 0\Mod {w_2}\\ P((a_2-c_3a_3)\ov{a_3})\equiv 0\Mod{v_3}}}1.
\]
If the coefficient in $a_2^2$ in $d_2 (a_2,a_3)$ is non zero, then the inner sum over $a_2$, $a_3$ is
$$\ll \eta_1 A_3 \Big (1+\frac{\eta_1A_2}{w_1w_2v_3}\Big )12^{\omega (w_1w_2v_3)},$$
otherwise the condition $w_2| d_2 (a_2,a_3)$ is equivalent to $w_2| d_P a_3$ for some $d_P\in\Z$ depending only of $P$ (we recall that $w_2$ is square free) and thus the inner sum over $a_2,a_3$ is bounded by
$$\ll  \Big ( 1+\frac{\eta_1 A_3}{w_2}\Big )\Big ( 1+\frac{\eta _1A_2}{w_1v_3}\Big )12^{\omega (v_1v_2v_3)}
\ll 12^{\omega (w_1w_2v_3)}\Big (1+\frac{\eta _1A_3}{w_2}+\frac{\eta_1 A_2}{w_1v_3}+\frac{\eta_1^2A_2A_3}{w_1w_2v_3}\Big ).$$
Finally we obtain that 
\begin{align*}
    V'(\cH)&\ll Z^2(\log Z)^{10}
    \sum_{\substack{v_1v_2v_3\ge Z^5\\ v_1^2v_3\le X^{1+\alpha_0-\sum_{j=1}^6 \theta_{1j}}\\
v_2^2v_3\le X^{(1+\alpha_0)/2 -\theta_{21}}\\ \mu^2 (v_1v_2v_3)=1}}
6^{\omega (v_1v_2v_3)}\Big [ \eta_1^2 A_2A_3\\ &+\sum_{\substack{w_1|v_1\\ w_2|v_2}}\frac{12^{\omega (w_1w_2v_3)}\eta_1A_1w_1w_2}{v_1^2v_2^2v_3^2}\Big (
\eta_1 A_3+\frac{\eta_1A_2}{w_1v_3}+\frac{\eta_1^2A_2A_3}{w_1w_2v_3}\Big )\Big ]\\
& \ll Z^{-3}(\log Z)^{10}\eta_1^3 A_1A_2A_3+Z^3\frac{A_1A_2A_3}{\min (A_1,A_2,A_3)}X^{3(1+\alpha_0)/4-\sum_{(i,j)\in I_\cC}\theta_{ij}/2}.
\end{align*}
By \eqref{goodcube} and \eqref{eq:Con5} we see that $X^{3(1+\alpha_0)/4-\sum_{ij\in I_\cC}\theta_{ij}/2}\le \min (A_1,A_2,A_3)X^{-\epsilon}$. Putting everything together then gives the result.
\end{proof}

%
%
\begin{lmm}[Removing $(a_2,a_3)/30=1$]\label{lmm:S06}
Let $S_{05}(\cH)$ be as given in Lemma \ref{lmm:S05}. Then we have
\[
S_{05}(\cH)=S_{06}(\cH)+O\Bigl(\frac{\eta_1^3 A_1A_2A_3}{Z}\Bigr),
\]
where
\[
S_{06}(\cH):=\sum_{\substack{t\le Z^{50}\\ u\le Z^{20}\\ d\le Z\\ s_2s_3\le Z\\ (t,30)=1}}\mu (d)\mu (s_2s_3)\ell(u)\mu (t)\sum_{\substack{{q_{ij}\in [X^{\theta_{ij}},X^{\theta_{ij}+\tau_{ij}}]}\\
\forall\,(i,j)\in I_\cC\\ q_{21}\equiv 1\Mod{D_{q_2}}}}\sum_{\substack{{(a_1,a_2,a_3)\in\cH'}\\ {\prod_{j=1}^6q_{1j}|q_1(a_1,a_2,a_3)}\\ {q_{21}|q_2(a_1,a_2,a_3)}\\
{[d,s_2s_3,u]|q(a_1,a_2,a_3)}\\ d|q_3(a_1,a_2,a_3)\\ [t,s_2]|a_2 \\ [t,s_3]|a_3 \\
a_2,a_3\equiv 30\Mod{900} \\ a_1\equiv 1\Mod{30}}}1.
\]
\end{lmm}
%
%
\begin{proof}
Since we have $a_2,a_3\equiv{30}\Mod{900}$, we can detect $(a_2,a_3)|30$ using Möbius inversion $\1_{(a_2,a_3)|30}=\sum_{\substack{t|(a_2,a_3)\\ {(t,30)=1}}}\mu (t)$ and separately consider the contribution $S_{06}(\cH)$ from terms with $t\le Z^{50}$ and the contribution $U_5(\cH)$ from terms with $t>Z^{50}$. Since there are $O(1)$ choices of the $q_{ij}$ given a choice of $a_1,a_2,a_3$, we see that
\begin{align*}
U_5(\cH)&\ll \sum_{t>Z^{50} }\sum_{\substack{u\le Z^{20}\\ d,s_2s_3\le Z}}\sum_{\substack{(a_1,a_2,a_3)\in \cH'\\ t|(a_2,a_3)}}O(1)\\
&\ll Z^{30}\sum_{Z^{50}<t<\min(A_2,A_3)}\eta_1 A_1 \Bigl(\frac{\eta_1 A_2}{t}+1\Bigr)\Bigl(\frac{\eta_1 A_3}{t}+1\Bigr)\ll \frac{\eta_1^3A_1A_2A_3}{Z}.
\end{align*}
This gives the result.
\end{proof}
%
%
\subsection{Application of Theorem \ref{DistriNorm}}
\label{Lm}%
%
\begin{lmm}[Application of Theorem \ref{DistriNorm}]\label{lmm:S07}
Let $S_{06}(\cH)$ be as in Lemma \ref{lmm:S06}. Then we have
\[
S_{06}(\cH)\gg \eta_1^3 A_1A_2A_3.
\]
\end{lmm}
%
%
\begin{proof}[Proof of Lemma \ref{lmm:S07} assuming Theorem \ref{DistriNorm}]

Recalling the definition of $S_{06}$ from Lemma \ref{lmm:S06}, we remark that the different conditions modulo $30$  on $a_1,a_2,a_3$ imply that $(q(a_1,a_2,a_3),30)=1$ and thus we may impose that $(ds_2s_3tu,30)=1$.
Splitting $(a_1,a_2,a_3)$ into residue classes $\Mod{[t,u,d,s_2,s_3]}$, we see that
\begin{align}
S_{06}(\cH)&=\sum_{\substack{t\le Z^{50}\\ u\le Z^{20}\\ d\le Z\\ s_2s_3\le Z\\ (ds_2s_3tu,30)=1}}\mu (d)\mu (s_2s_3)\ell(u)\mu (t)\sum_{\u_0\in \cS(d,s_2,s_3,t,u)}S_{07}(\u_0,[d,s_2,s_3,t,u]),
\label{eq:S6}
\end{align}
where
\begin{align*}
S_{07}(\u_0,m)&:=\sum_{\substack{{q_{ij}\in [X^{\theta_{ij}},X^{\theta_{ij}+\tau_{ij}}]}\\
\forall\,(i,j)\in I_\cC\\ q_{21}\equiv 1\Mod{D_{q_2}}}}\sum_{\substack{{(a_1,a_2,a_3)\in\cH'}\\ {\prod_{j=1}^6q_{1j}|q_1(a_1,a_2,a_3)}\\ {q_{21}|q_2(a_1,a_2,a_3)}\\
{(a_1,a_2,a_3)\equiv \u_0\Mod{m}}\\
a_2,a_3\equiv 30\Mod{900},\ a_1\equiv 1\Mod{30} }}1,\\
    \cS ( d,s_2,s_3,t,u)&:=\Bigl\{ (u_1,u_2,u_3)\Mod{[d,s_2s_3,t,u]}:\,[d,s_2s_3,u] | q (u_1,u_2,u_3),\\
    &\qquad d|q_3(u_1,u_2,u_3),\, [s_2,t]|u_2,\,
    [s_3,t]|u_3\Bigr\}.
\end{align*}
We now apply  Theorem \ref{DistriNorm} on incomplete norms with $K=\Q (r_1+r_3)$, $\nu_1=1$, $\nu_2=r_1+r_3$, $\nu_3=r_1^2+r_3^2+r_1r_3$ and $\nu_4$ such that 
$\nu_4$ is in the ring of integers of $K$ and $(\nu_1 , \nu_2,\nu_3 ,\nu_4)$ is a $\Q$-basis of $K$.
 By Theorem \ref{DistriNorm} (taking $X_i=A_i$, $\ell=5$, $\ell'=3$, $\theta_i=\theta_{1i}\frac{\log{X}}{\log{A_1}}$, $\theta_i'=(\theta_{1i}+\tau_{1i})\frac{\log{X}}{\log{A_1}}$, $\tau=\theta_{21}\frac{\log X}{\log A_1}$, $\tau'=(\theta_{21}+\tau_{21})\frac{\log X}{\log A_1}$), we have that
\[
S_{07}(\u_0,m)=(1+o(1))\frac{\eta_1^3 A_1A_2A_3}{30^5 m^3\varphi (D_{q_2})}\prod_{(i,j)\in I_\cC}\log\Bigl(1+\tau_{ij}/\theta_{ij}\Bigr).
\]
Here we have used the fact that \eqref{eq:Xi1} and \eqref{eq:Xi2} hold by \eqref{goodcube}. Similarly \eqref{eq:th1} holds by \eqref{eq:Con1}, \eqref{eq:th2} holds by \eqref{eq:Con10}, \eqref{eq:th3} holds by \eqref{eq:Con3}, \eqref{eq:th4} holds by \eqref{eq:Con5}, \eqref{eq:th5} holds by \eqref{eq:Con4} and \eqref{eq:tau'} holds by \eqref{eq:Con13} and \eqref{eq:Con14} and by noticing that 
$\frac{4}{1+\alpha_0}\le \frac{\log X}{\log A_1}\le\frac{4}{1+\alpha_0/2}$. Substituting this into our expression \eqref{eq:S6} for $S_{06}$, we find that
\begin{equation}
    S_{06} (\cH) = (1+o(1))\frac{\eta_1^3}{30^5} \frac{A_1A_2A_3}{\varphi (D_{q_2})}\prod_{(i,j)\in I_\cC}\log\Bigl(1+\tau_{ij}/\theta_{ij}\Bigr)    \sum_{m\le Z^{72} }\frac{L(m)}{m^3},
    \label{eq:S6Asymp}
\end{equation}
where
\begin{equation*}
    L(m):=\sum_{\substack{ d\le Z \\ s_2s_3\le Z\\  u\le Z^{20}\\ (ds_2s_3tu,30)=1 \\ t<Z^{50}\\ [d,s_2s_3,t,u]=m }}\mu (d)\mu (s_2s_3)\mu (t)\ell (u)|\cS (d,s_2,s_3,t,u)|.
    \end{equation*}
We wish to remove the upper bound constraints on $d,s_2,s_3,u,t,m$ so we can understand $\sum_{m}L(m)/m^3$ via an Euler product. Let 
\begin{align*}
L^*(m)&:=\sum_{\substack{[d,s,t,u]=m\\ (ds_2s_3tu,30)=1}}|\mu (d)\mu(s)\mu (t)\ell (u)|\sum_{s_2s_3=s}|\cS (d,s_2,s_3,t,u)|,\\
\tilde L (m)& := \sum_{\substack{[d,s,t,u]=m\\ (ds_2s_3tu,30)=1}}\mu (d)\mu (s)\mu (t)\ell (u)\sum_{s_2s_3=s}|\cS (d,s_2,s_3,t,u)|,
\end{align*}
which are multiplicative functions of $m$. We note that $L^*(m)\ge \max(|L(m)|,|\tilde{L}(m)|)$ for all $m$ and that $\tilde{L}(m)=L(m)$ for $m\le Z$. From the support of $\mu,\ell$ we have $L^*(p^k)=0$ for $k\ge 3$. We easily check that $L^* (p)\le 2^5 p$  and $L^* (p^2)\le 3p^2$ for $p>q_0$ since $|\ell (p)|\le 2/(p-2)$ in this range. We deduce that $L^*(m)/m^3\ll \tau(m)^5/m^2$. 
We note that $\tilde L (p^k)=0$ for $k\ge 2$ and $2\le p\le q_0$, and 
that $\tilde L (p^k)=0$ for any $k\ge 1$ when $p=2,3,5$.
We find
\begin{align*}
\sum_{m\le Z^{72} }\frac{L(m)}{m^3}&=\sum_{m\le Z}\frac{\tilde{L}(m)}{m^3}+O\Bigl(\sum_{m>Z}\frac{L^*(m)}{m^3}\Bigr)\\
&=\sum_{m}\frac{\tilde{L}(m)}{m^3}+O\Bigl(\sum_{m>Z}\frac{\tau(m))}{m^2}\Bigr)\\
&=\prod_{7\le p\le q_0}\Big (1+\frac{\tilde{L}(p)}{p^3}\Big )\prod_{p>q_0}\Bigl(1+\frac{\tilde{L}(p)}{p^3}\Bigr)+O\Bigl(\frac{1}{Z^{1/2}}\Bigr).
\end{align*}
From our bounds on $L^*$ we see that $\prod_{p>q_0}(1+\tilde{L}(p)/p^3)\gg 1$ and the product over $p\le q_0$ converges. We wish to show that the product converges to a strictly positive constant, and so need to check that $1+\tilde{L} (p)/p^3$ doesn't vanish for some small prime $p$ with $7\le p\le q_0$.
If $p|[d,s_2,s_3]$ then for  $u=1$ or $p$, we have
$$|\cS (d,s_2,s_3,u,1)|=|\cS (d,s_2,s_3,u,p)|.$$
Since $\ell(p)+\ell (1)=0$ when $7\le p\le q_0$,
we deduce
\[
\sum_{[d,s_2,s_3]=p}\sum_{[d,s_2,s_3,t,u]=p}
\mu (d)\mu (s_2s_3)\mu (t)\ell (u)
|\cS (d,s_1,s_2,t,u)|=0.
\]
  The value $\tilde L (p)$ is then  
$$\tilde L (p)=1-p-|\{ (u_1,u_2,u_3)\Mod p : p|q(u_1,u_2,u_3)\}|.$$
Then $$1+\tilde L(p)/p^3\ge (p^3-6p^2-p+1)/p^3 >0,$$
when $p\ge 7$.  
 Thus $\sum_{m\le Z^{72}}L(m)/m^3\gg 1$, and so substituting this into \eqref{eq:S6Asymp} and using the fact $\tau_{ij}/\theta_{ij}\gg 1$ we obtain the result.
\end{proof}
%
%
%
\subsection{Proof of Proposition \ref{prpstn:S0}}
%
%
\begin{proof}[Proof of Proposition \ref{prpstn:S0} assuming Theorem \ref{DistriNorm}]
By Lemmas \ref{lmm:S01}, \ref{lmm:S02}, \ref{lmm:S03}, \ref{lmm:S04}, \ref{lmm:S05}, \ref{lmm:S06} and \ref{lmm:S07} in turn, we see that
\[
S_0\gg \frac{1}{\log{X}}\sum_{\cH\in\sH_\cR}\frac{A_0A_1A_2A_3 \eta_1^4}{\widetilde{N}_P(A_0,A_1,A_2,A_3)}+O\Bigl(\frac{1}{Z^{1/2}}\Bigr).
\]
(Note that in this application  Lemma \ref{lmm:S07} we are assuming Theorem \ref{DistriNorm}, and that we have $12\theta_0+22\alpha_0<1$ required for Lemma \ref{lmm:S01} since we are taking $\theta_0$ sufficiently small and assuming that $\alpha_0$ satisfies \eqref{eq:Con7}.) We note that
\[
\sum_{\cH\in\sH_\cR}\frac{A_0A_1A_2A_3 \eta_1^4}{\widetilde{N}_P(A_0,A_1,A_2,A_3)}=\sum_{\cH\subset\cR}\frac{A_0A_1A_2A_3 \eta_1^4}{\widetilde{N}_P(A_0,A_1,A_2,A_3)}-\sum_{\substack{\cH\subset \cR\\ \cH \text{ bad} }}\frac{A_0A_1A_2A_3 \eta_1^4}{\widetilde{N}_P(A_0,A_1,A_2,A_3)}.
\]
If $\cH$ is bad, then $\max(A_1,A_2,A_3)^4\eta_1^{1/10}\ge  q_1(A_1,A_2,A_3)$ or there exists $i\in\{ 0,1,2,3\}$
such that $|A_i|<\eta_1\max (|A_0|,|A_1|,|A_2|, |A_3|)$. The first inequality implies that there exists $(i,j)\in I_\cC$
such that 
\[ L_{i,j}(A_1,A_2,A_3):=|A_1+(r_i+r_j)A_2+(r_i^2+r_ir_j+r_j^2)A_3|\ll \eta_1^{1/40}\max (A_1,A_2,A_3).
\]
Thus, by partial summation
\begin{align*}
\sum_{\substack{\cH\subset \cR\\ \cH \text{ bad} }}\frac{A_0A_1A_2A_3 \eta_1^4}{\widetilde{N}_P(A_0,A_1,A_2,A_3)}&\ll
\sum_{(i,j)\in I_\cC}\sum_{\substack{A=2^\ell \\ X^{1-\frac{7\alpha_0}{8}}\ll A\ll X^{\frac{1+\alpha_0}{4}} }} \sum_{\substack{(a_0,a_1,a_2,a_3)\in \cR\\ L_{i,j}(a_1,a_2,a_3 )\le \eta_1^{1/40}A \\\max(a_0,a_1,a_2,a_3)\ll A}}\frac{1}{A^4}\\
&+ \sum_{i=0}^3\sum_{\substack{A=2^\ell \\ X^{1-\frac{7\alpha_0}{8}}\ll A\ll X^{\frac{1+\alpha_0}{4}} }}
\sum_{\substack{(a_0,a_1,a_2,a_3)\in \cR\\ a_i\le \eta_1 A \\\max(a_0,a_1,a_2,a_3)\ll A}}\frac{1}{A^4}\\
&\ll \eta_1^{1/40}\log{X}.
\end{align*}
Similarly, we find by partial summation
\begin{align*}
\sum_{\substack{\cH\subset \cR }}\frac{A_0A_1A_2A_3 \eta_1^4}{\widetilde{N}_P(A_0,A_1,A_2,A_3)}&=(1+o(1)) \sum_{\substack{(a_0,a_1,a_2,a_3)\in \cR}}\frac{1}{\widetilde{N}_P(a_0,a_1,a_2,a_3)}\\
&\gg \log{X}.
\end{align*}
Putting everything together now gives Proposition \ref{prpstn:S0}.
\end{proof}

Thus we are left to establish Theorem \ref{DistriNorm}.

%
%
%
%
%
%
\section{Incomplete norm forms}
%
%
%
%
In this section we perform our initial reductions to reduce the proof of Theorem \ref{DistriNorm} to that of establishing Proposition \ref{prpstn:TSieve} and Proposition \ref{prpstn:T1}. We roughly follow the argument of \cite{May15b} in this section, but require a number of small technical modifications.

Let  $K$ be a quartic number field, $\cO_K$ its integer ring, $Cl_K$ its class group. Let $\nu_1,\nu_2,\nu_3,\nu_4 \in \cO_K$  such that $\v=(\nu_1,\nu_2,\nu_3,\nu_4)$ is a $\Q$-basis of $K$. We suppose for convenience that $\nu_1=1$ and $K=\Q (\nu_2)$. We then define $\cO_\v =\Z [ \nu_1,\nu_2,\nu_3,\nu_4]$ the order generated by $\v$.

We let $N(\cdot)=N_K(\cdot)$ be the norm on $K$, and note that this is a different norm to $N_P$ on $\Q(r_1)$ encountered earlier.

There exists an integral basis of $\cO_K$, $\w =(\omega_1 ,\omega_2,\omega_3 ,\omega_4)$ and some integers
$w_{ij}$,  $1\le i\le j\le 4$, such that
\begin{equation}
\label{basew}
\nu_j=\sum_{i=1}^j w_{ij}\omega_i\quad (j=1,2,3,4).
\end{equation}
(cf. for example \cite[Proposition 2.11]{Nark}).
 
%
%
\subsection{From $\cO_K$ to  $\cO_\v$ and vice-versa}
%
%

We denote by $L_{\w \v}=(w _{ij})_{1\le i,j\le 4}$ the matrix of $\v$ in $\w$ so that
for all $1\le j\le 4$, 
$\nu_j=\sum_{i=1}^4w _{ij}\omega_i$.

By \eqref{basew} this matrix is upper triangular and  the absolute value of its  determinant is 
\begin{equation}
\label{W}
W=|w_{11}w_{22}w_{33}w_{44}|\in\Z^*
\end{equation}
%
%
\begin{lmm}\label {coordinates}
For all  $\alpha \in \cO_K$, there exist $a_1,a_2,a_3,a_4\in\Z$,
with 
\begin{equation*}
\alpha =\frac{1}{W}\sum_{i=1}^4a_i\nu_i
\end{equation*}
Conversely, there exists a subset $\cV_0\subset \{ 0,\ldots , W-1\}^4$ such that for all $\a\in\Z^4$ we have
\begin{equation*}
\frac{1}{W}\sum_{i=1}^4a_i\nu_i\in \cO_K\Leftrightarrow \exists \u\in \cV_0 : \a\equiv\u \Mod W.
\end{equation*}
\end{lmm}
%
%
\begin{proof}
Let $\alpha\in\cO_K$. There exist $(a_1,a_2,a_3,a_4)\in\Z^4$
and $(a'_1,a_2',a_3',a_4')\in\Q^4$
such that  $\alpha =\sum_{i=1}^4a_i\omega_i=\sum_{i=1}^4a'_i\nu_i$.
With our previous notation,
\begin{equation*}
\begin{pmatrix} a'_1 \\ a'_2\\ a'_3\\ a'_4\\
\end{pmatrix}
=(L_{\w\v})^{-1}
\begin{pmatrix} a_1 \\ a_2\\ a_3\\ a_4\\
\end{pmatrix}.
\end{equation*}
The matrix $(L_{\w\v})^{-1}$ is of type $\frac{1}{W} (w '_{ij})_{1\le i,j\le 4}$
where the coefficients $w '_{ij}$ are integers.
This implies the first part of the lemma.

The second part of the lemma is also a direct consequence of the change of basis formula.
With our previous notation we have
\begin{equation*}
\sum_{j=1}^4a_j\nu_j=\sum_{i=1}^4\Big (\sum_{j=1}^4w_{ij}a_j\Big )\omega_i.
\end{equation*}
Then for any $\a=(a_1,a_2,a_3,a_4)\in\Z^4$, $\frac{1}{W}\sum_{i=1}^4a_i\nu_i\in\cO_K$ if and only if
for all $1\le i\le 4$, we have
\begin{equation*}
\sum_{j=1}^4w _{ij}a_j\equiv 0\Mod W.
\end{equation*} 
The set $\cV_0$ is the the subset of $\{ 0,\ldots ,W-1\}^4$ formed by all the solutions of these congruences.
\end{proof}
%
%
%
%
\begin{lmm}\label{Wsg}
Let $\ga$ be a principal ideal. Then there is a generator $\alpha$ of $\ga$ such that 
\begin{equation*}
|\alpha^\sigma |\ll N(\ga )^{1/4}
\end{equation*}
for all embeddings $\sigma : K \hookrightarrow \C$. Furthermore there exists $V>0$ depending only on $\v$ such that
\begin{equation*}
\alpha =\frac{1}{W}\sum_{i=1}^4a_i\nu_i
\end{equation*}
for some integers $a_i\ll N(\ga )^{1/4}$.
\end{lmm}
%
%
\begin{proof}
The first part is a particular case of \cite[Lemma 4.3]{May15b}. 
The last part follows also from this lemma combined with  Lemma \ref{coordinates}.
\end{proof}
%
%

\medskip

%
%
\begin{lmm}\label{Ovq}
Let $\cC$ be an hypercube of side length $\delta_0 B$ which contains a point $\bb_0\in\Z^4$ such that $\| \bb_0\| \ll B$.
We suppose that  $\gb_0 =(W^{-1}\sum_{i=1}^4 (\bb_0 )_i\nu_i)$ is an integral ideal  whose norm satisfies 
$N(\gb_0)=B_0^4\gg B^4$. Let  $q$ such that $W|q$ and $10qW\le \delta_0B$.

Then there exists  a set  $\cW(\b_0)$ of  $W^4$  elements $\beta_0'\in\cO_K$
with $\beta_0'=W^{-1}\sum_{i=1}^4 (\bb _0')_i\nu_i$  and with $\bb_0'\in\cC$,
such that  for all $\bb\in\cC$, 
$\bb\equiv\bb_0\Mod q$ if and only if $\beta =\frac{1}{V}\sum_{i=1}^4 b_i\nu_i\in\cO_K$ and there exists $\beta_0'\in\cW(\b_0)$ with
$\beta\equiv\beta_0'\Mod q$.
\end{lmm}
%
%
\begin{proof}
This is  variant of an argument used in the proof of \cite[Lemma 9.4]{May15b}.

Let $\beta_0 :=\frac{1}{W}\sum_{i=1}^4 (\b_0)_i\nu_i$. 
For all $\v=(v_1,\ldots ,v_4)\in\{ 0,\ldots ,W-1\}^4$, there exists $\u =\u (\b_0 , \v)\in\Z^4$
such that $\b_0+q(\v +W\u)\in\cC$ since $qW\le \delta_0 B$, the side length of $\cC$. We will prove that the set 
\[
\cW:=\Big\{ \beta_0'=\frac{1}{W}\sum_{i=1}^4 b_i'\nu_i \ \text{with}\ \b'= \b_0+q(\v +W\u (\b_0,\v)), \v\in
\{ 0,\ldots ,W-1\}^4\Big\}
\]
satisfies the conclusion of the lemma.

First we suppose that $\bb \equiv \bb_0\Mod q$. This implies that there exist four integers $m_1,m_2,m_3,m_4$
such that $b_i=(\b_0)_i+qm_i$.  
We get 
$$\beta :=\frac{1}{W}\sum_{i=1}^4\b_i\nu_i=\frac{1}{W}\sum_{i=1}^4 ((\bb_0)_i +m_iq)\nu_i=\beta_0+\frac{q}{W}\sum_{i=1}^4m_i\nu_i.$$
 Since $W|q$, this implies that $\beta\in\cO_K$.
If we choose $\beta_0'=\beta_0+\frac{q}{W}\sum_{i=1}^4 v_i\nu_i+$
with $0\le v_1,\ldots ,v_4<W$ such that $v_i\equiv m_i\Mod W$ then we would have  
$\beta =\beta_0' +\frac{q}{W}\sum_{i=1}^4 ( m_i-v_i+Wu_i)\nu_i$, and thus $\beta\equiv \beta_0'\Mod q$.

Now we prove the reciprocal assertion.  We suppose that there exists $\beta'_0\in\cW$  such that $\beta\equiv \beta'_0\Mod q$. Then 
 $\beta=\beta_0' +q\gamma$ for some $\gamma\in\cO_K$. There exists
$g_1,g_2,g_3,g_4\in\Z$ such that $\gamma=\frac{1}{W}\sum_{i=1}^4 g_i\nu_i$.
For each $i=1,2,3,4$, we have
$\frac{b_i}{W}=\frac{(\b_0)_i +q(v_i+Wu_i+g_i)}{W}.$ This implies that $\bb\equiv \bb_0 \Mod q$.
\end{proof}
%
%

\medskip
For any ideal $\gd$ of $\cO_K$, we define the function $\varrho_\v$ by 
\begin{equation}
\varrho_\v (\gd):=\frac{|\{\a\in [1,N(\gd )^3] : \gd | (a_1\nu_1+a_2\nu_2+a_3\nu_3)\}|}{N(\gd)^2}.\label{eq:RhoVDef}
\end{equation}

 This function  satisfies the following properties.
%
%
\begin{lmm}\label{rhov}
\begin{enumerate}
\item For all degree one prime ideals $\gp$ with $(N(\gp), W)=1$, we have
$\varrho_\v (\gp)=1$.
\item We have 
\begin{equation*}
\Big|\Big\{ \x\in [1,p^2]^3 : p^2| N\Big ( \sum_{i=1}^3 x_i\nu_i\Big )\Big\}\Big | \ll p^4.
\end{equation*}
\item For any ideal $\goe$ such that $N(\goe )$  is a power of $p$, we have
\begin{equation*}
\frac{\varrho_\v (\goe)}{N(\goe)}\ll  \frac{1}{p^2}
\end{equation*}
unless $\goe$ is a degree $1$ prime ideal above $p$.
\item For any ideals $\ga ,\gb$, $\varrho_\v (\ga \gb)=\varrho_\v (\ga)\varrho_\v (\gb)$ if $(N(\ga ),N(\gb))=1$.
\item For $k\ge 3$, we have
\begin{equation*}
\Big |\Big \{ \x\in [1,p^k]^3 : p^k| N\Big (\sum_{i=1}^3x_i\nu_i\Big )\Big \}\Big |\ll k p^{11k/4}.
\end{equation*}
\end{enumerate}
\end{lmm}
%
%
\begin{proof} The first four assertions are essentially given by \cite[Lemma 7.7]{May15b}, except that they work with a basis $\nu_1,\nu_2,\nu_3,\nu_4$ in place of $1,\theta ,\theta^2,\theta^3$ which has a negligible effect on the proof. Indeed, by \eqref{basew}  the $\Q$-vector space spanned by $\nu_1, \nu_2,\nu_3$ is the same
as the one spanned by $\omega_1,\omega _2,\omega_3$, and the change-of-basis matrix between the basis $\nu_1,\nu_2,\nu_3,\nu_4$ and $\omega_1,\omega_2,\omega_3,\omega_4$ has determinant $W$. Thus when $(N(\gd ),W)=1$ we have
$$\varrho_\v (\gd )=\frac{|\{\a\in [1,N(\gd )^3] : \gd | (a_1\omega_1+a_2\omega_2+a_3\omega_3)\}|}{N(\gd)^2},$$
and so it is sufficient to prove these four statements with the basis $\w$ in place of $\v$. The proof is then the same as in \cite{May15b}.

We are left to establish assertion 5. Since $N(\nu_1)\ne 0$, for any choice of $x_2,x_3$, $g_{x_2,x_3}(x_1):=N(x_1\nu_1+x_2\nu_2+x_3\nu_3)$ is a non-zero polynomial of degree 4 in $x_1$. Thus, given $x_2,x_3$, if $N(x_1\nu_1+x_2\nu_2+x_3\nu_3)\equiv 0\Mod{p^k}$, we see that $\|x_1-\alpha\|_p\ll p^{-k/4}$ for one of the 4 roots $\alpha$ of $g_{x_2,x_3}$ over $\overline{\Q_p}$. Thus there are $O(p^{3k/4})$ choices of $x_1\in[1,p^k]$ for each choice of $x_2,x_3$. This gives the result.
\end{proof}
%
%


Let $\gamma_K$ be the residue in $s=1$ of $\zeta_K$ and we define $\Sft$ to be the Euler product\footnote{This definition of $\Sft$ is slightly different as the one given in \cite{May15b}. In the present paper $\Sft$ doesn't depend on some modulus $q^*$ or $m$.} 
\begin{equation}
\Sft :=\prod_\gP \Big ( 1-\frac{\varrho_\v (\gP)}{N(\gP)}\Big )\Big ( 1-\frac{1}{N(\gP)}\Big )^{-1}.\label{eq:SingularDef}
\end{equation} 
%
%
\begin{lmm}\label{Perron}There exists a constant $c>0$ such that 
for  any ideal $\gI$ of $\cO_K$, $m\in\N$, $R\ge 2$ we have
\begin{equation*}
\sum_{\substack{{N(\gd)<R}\\ {(\gd ,\gI)=1}\\ {(N(\gd),m)=1}}} \frac{\mu (\gd )\varrho_\v (\gd)}{N(\gd )}\log \frac{ R}{N(\gd )}
=\frac{\Sft}{\gamma_K}\prod_{\gP | (m)\gI}\Big ( 1-\frac{\varrho _\v (\gP)}{N(\gP)}\Big )^{-1}+O\Big ( 2^{4\omega (\gI)}\exp (-c\sqrt{\log R})\Big ).
\end{equation*}
\end{lmm}
%
%
\begin{proof}
The proof is exactly the same as in \cite[Lemma 8.5]{May15b}. \cite[Lemma 8.5]{May15b} states the result with $N(J)^{o(1)}$ in place of $2^{4\omega(J)}$, but following the proof we see that the error term can be taken as $\exp (-c\sqrt{\log R})\prod_{\gP|J}( 1-\frac{1}{N (\gP)^{3/4}})^{-1}$, which is clearly sufficient for our slightly stronger bound.
   \end{proof}
   \begin{lmm}
\label{rhodrhoI}
 For any $2\le R\le x$ we have
\begin{equation*}
\sum_{N(\gd )\le R}\mu^2 (\gd)\sum_{N(\gI)\le x}\frac{\rho_v (\gd \gI)}{N(\gd\gI)}\ll (\log x)^8.
\end{equation*}   
\end{lmm}
\begin{proof}
By Rankin's trick, we have
\[
\sum_{N(\gd )\le R}\mu^2 (\gd)\sum_{N(\gI)\le x}\frac{\rho_v (\gd \gI)}{N(\gd\gI)}\le \prod_{N(\gP)\le x}\Bigl(1+2\sum_{k\ge 1}\frac{\varrho_\v (\gP^k)}{N(\gP^k)}\Bigr).
\]
By Lemma  \ref{rhov}, if $\gP$ is a degree 1 prime ideal above $p$ then the term in parentheses is $1+2/p+O(1/p^2)$, and if $\gP$ is of degree more than 1 above $p$ then this is $1+O(1/p^2)$. The result now follows from the Prime Ideal Theorem.
\end{proof}
%
%

%
%
\subsection{Multiplication in $\cO_\v$}
%
%
\begin{dfntn}\label{mult}
For any vectors $\bd, \be \in\Z^4\setminus\{\zero\}$, we define $\bd\diamond\be$  as be the vector  
$\b\in\Q^4$ such that 
\begin{equation*}
\sum_{i=1}^4b_i\nu_i=\sum_{i=1}^4 d_i\nu_i\times \sum_{i=1}^4 e_i\nu_i
\end{equation*}
For $1\le i\le 4$ we denote by $(\bd\diamond\be)_i$ the coordinate $b_i$.
\end{dfntn}
%
%
This operation is helpful to detect the elements of $\cO_\v$ with a fourth coordinate equal to zero.
The following lemma turns the problem of detecting this zero coordinate into a question about lattices.
%
%
\begin{lmm}\label{de40}
For any $\bd\in\Z^4\setminus\{ \zero\}$
let $\Lambda_\bd$ be the subset of $\Z^4$ defined by 
\begin{equation*}
\Lambda_\bd =\{ \be\in\Z^4 : (\bd\diamond\be )_4=0\}.
\end{equation*}
Then $\Lambda_\bd$ is a lattice of rank $3$ and $\det (\Lambda _\bd)\asymp \| \bd\|/D $,where $D$ is the GCD of the components of $\bd$.
\end{lmm}
%
%
\begin{proof}
The argument is essentially a special case of \cite[Lemma 7.2]{May15b} .
We will expose it in a more pedestrian way. For all $1\le i,j\le 4$ there exist rational numbers $\lambda_{i,j,k}$, $1\le k\le 4$ such that
\begin{equation*}
\nu_i\nu_j=\sum_{k=1}^4\lambda_{ijk}\nu_k.
\end{equation*}
For all $\bd ,\be\in\Z^4$, 
\begin{equation*} 
\sum_{i=1}^4(\bd\diamond\be)_i\nu_i=\sum_{k=1}^4\Big (\sum_{i,j=1}^4\lambda_{ijk}d_ie_j\Big )\nu_k
\end{equation*}
Identifying the fourth coordinate, we deduce for all  $\bd\in\Z^4\setminus\{\zero\}$, 
\begin{equation*}
\Lambda_\bd =\Big \{ \be\in\Z^4 : 
\sum_{j=1}^4\Big (\sum_{i=1}^4 \lambda_{ij4}d_i\Big )e_j=0\Big \}.
\end{equation*}
The terms $\sum_{i=1}^4 \lambda_{ij4}d_i$, for $j=1,2,3,4$ correspond to the coefficients of the fourth row of the matrix
in basis $\v$ 
of the multiplication by $d=d_1\nu_1+d_2\nu_2+d_3\nu_3+d_4\nu_4$.
Since $\bd\not =\zero$, this matrix is invertible  and at least one of these coefficients is non zero.
This shows that $\Lambda_\bd$ has rank $3$.
By \cite{HB84}, the determinant of $\Lambda_\d$ is equal to the determinant of the dual lattice  that is for us the lattice spanned by the vector \begin{equation}\label{defT}T(\bd):= 
\begin{pmatrix}
\sum_{i=1}^4 \lambda_{i14}d_i\\
\sum_{i=1}^4 \lambda_{i24}d_i\\
\sum_{i=1}^4 \lambda_{i34}d_i\\
\sum_{i=1}^4 \lambda_{i44}d_i\\
\end{pmatrix}.
\end{equation}
Since the components of this vector have size $O(\max_{1\le i\le 4}|d_i|)$, $\det (\Lambda_d)\ll \|\bd \|$. 
\end{proof}
%
%

\medskip

%
%
\begin{lmm}\label{div}
For any $m\in\N$ and $X\ge 3$, we have
\begin{equation*}
\sum_{\max (|x_1|,|x_2|,|x_3|)\ll X}
\tau\Big (\sum_{i=1}^3 x_i\nu_i\Big )^m\ll X^3(\log X )^{O_m(1)}.
\end{equation*}
\end{lmm}
%
%
\begin{proof}
The proof is the same as that of \cite[Lemma 4.2]{May15b} which concerns the case $\nu_i=\theta^{i-1}$.
The only place where this change could have an importance is for the bound of the sums  with any $\gd$ such that $N(\gd )\ll X^{1/n}$
\begin{equation*}
\sum_{\substack{{\max (|x_1|,|x_2|,|x_3|)\ll X}\\ {\gd|(\sum_{i=1}^3x_i\nu_i)}}}1.
\end{equation*}
Since the $\nu_i$ are linear combinations of some $\theta^{j}$, $j=0,1,2,3$ for $\theta$ such that $K=\Q (\theta )$,
the condition $\gd |  (\sum_{i=1}^3x_i\nu_i)$ can be split in the $x_i$ into arithmetic progression $\Mod {N(\gd )}$,
and thus the argument of \cite{May15b} combined with Lemma \ref{rhov} apply also in our case.
\end{proof}
%
%
%
%
\begin{lmm} Let $\bd\in\Z^4\setminus\{\zero\}\cap [-D,D]^4$ and 
$\Lambda_\bd$ as in Lemma \ref{de40}.  Let $\z_1 (\bd)$ denote a shortest non-zero vector in $\Lambda_\bd$. Then we have $\|\z_1 (\bd )
\| \ll D^{1/3}$ and
\begin{equation*}
|\{ \bd\in [1,D]^n : \| \z_1 (\bd )\| \le Z\}|\ll D^{3+o( 1)}Z^3.
\end{equation*}
Furthermore we have
\begin{equation*}
\sum_{\| \bd\|\le D}\frac{1}{\| \z _1 (\bd )\| ^2}\ll D^{10/3 +o(1)}.
\end{equation*}
\end{lmm}
%
%
\begin{proof}
The proof is exactly the same as the proof of \cite[Lemma 7.3]{May15b} except that we have a slightly different definition for $\diamond$, and so require Lemmas \ref{de40} and \ref{div} instead of \cite[Lemma 4.2]{May15b} and \cite[Lemma 7.2]{May15b}.
\end{proof}
%
%
\begin{lmm}\label{WTI}
Let $\gd$ be an ideal of $\cO_K$ with $(N(\gd) ,q)=1$. Let $\cR\subset [-X,X]^3$ as in the Proposition \ref{prpstn:TypeI} below.
Then we have
\begin{equation*}
|\big \{ \a\in\Z^3\cap\cR:\gd |\big (\sum_{i=1}^3 a_i\nu_i\big ),\ \a\equiv\a_0\Mod q\}|=\frac{\varrho_\v (\gd )\vol (\cR)}{N(\gd)q^3}+O(N(\gd )^4X^2).
\end{equation*}
\end{lmm} 
%
%
\begin{proof}
The proof is identical as the proof of \cite[Lemma 7.4]{May15b} with $\v$ in place of $(1,\ldots ,\theta ^{n-1})$. In fact,  the arguments of \cite{May15b} give a slightly stronger error term of $O(X^2\varrho_\v(N(\delta)(qN(\delta))^{-2}+\varrho_\v (N(\delta )))$.
\end{proof}
%
%
\subsection{Sums of Type I}
%
%
We now state  a similar result to \cite[Proposition 7.5]{May15b}
%
%
\begin{prpstn}\label{prpstn:TypeI}
Let $\cR\subset [-X,X]^3$ be a region such that any line  parallel to the coordinate axes intersects $\cR$ in $O(1)$ intervals.
For any given $\u_0\in\Z^3$  and $q\le\sqrt{X}$ we define
\begin{equation*}
\Gamma =\Big \{ \sum_{i=1}^3 a_i\nu_i : \a\in\Z^3\cap\cR,\ \a\equiv\u_0\Mod q\Big \}.
\end{equation*}

Let $\Gamma_\gd=\{ \kappa\in\Gamma : \gd |(\kappa)\}$. 
Then we have 
\begin{equation}\label{STI}
\sum_{\substack{{N(\gd)\in [D,2D]}\\{(N(\gd ),q)=1}}}\Bigg | |\Gamma_\gd |-\frac{\varrho_\v (\gd)\vol (\cR)}{q^3N(\gd )}\Bigg |\ll
X^2q^{1+o(1)}D^{1/3+o(1)}+Dq^{4+o(1)}.
\end{equation}
\end{prpstn}
%
%
\begin{proof} We follow the proof of \cite[Proposition 7.5]{May15b}, but now  we work with a general order $\cO_\v$ in place of $\Z [\theta]$. This involves minor modifications at the beginning of the argument; the last steps require no modification. For brevity we emphasise just the key points requiring modification and only sketch the rest of th argument.

We split the summation on the ideals $\gd$ according to their class in $Cl_K$.
Let $\cC$ be a given class and consider the contribution of all the $\gd\in\cC$. 
Since the $\gd$ in the summation in \eqref{STI} are coprime with $q$,
we can fix a representative integral ideal $\gc\in \cC$ such that $(N(\gc ),q)=1$ and with $N(\gc )=q^{o(1)}$.
The ideal $\gd\gc^{-1}(N(\gc))$ is a principal ideal of $\cO_K$. By Lemma \ref{Wsg} we can find a generator 
of the form $\delta=\frac{1}{W}\sum_{i=1}^4d_i\nu_i$ where the $d_i$ are integers such that 
$|d_i|\ll D^{1/n}q^{o(1)}$. Then $\delta_\gc:=\frac{1}{W N(\gc)}\sum_{i=1}^4d_i\nu_i$ is a generator of the principal fractional ideal $\gd\gc^{-1}$. In \cite{May15b} it is proved that $|\sigma_0(\delta_\gc)|\gg D^{1/4}q^{o(1)}$ for all embeddings $\sigma_0$.

 Let $\alpha\in\Gamma_\gd$, so $(\alpha)=\ga'\gd$ for some integral ideal $\ga '$. Since $(\alpha )=\ga '\gc\gd\gc^{-1}$ and $(\alpha )$ and $\gd\gc^{-1}=(\delta_\gc)$ are principal, $\ga '\gc$ is principal too, so
  $\ga'\gc=(\beta )$ for some generator $\beta\in\cO_K$.
 By Lemma \ref{coordinates}, we can take $\beta=\frac{1}{W}\sum_{i=1}^4b_i\nu_i$ where $\b =(b_1,b_2,b_3,b_4)\in\Z^4$ satisfies
 $(\b \Mod W) \in \cV_0$. Then $(\alpha)=(\beta )(\delta_\gc)$. 
Let $\bd =(d_1,d_2,d_3,d_4)$. We have $W^2N(\gc) \beta\delta_\gc=\sum_{k=1}^4(\bd\diamond \b)_k\nu_k$.
 \begin{equation*}
 \beta\delta_\gc=\sum_{k=1}^4\frac{1}{W^2N(\gc )}\Big (\sum_{i,j=1}^4\ell_{i,j,k}b_id_i\Big )\nu_k.
 \end{equation*}
 The coefficient of $\nu_i$ are integers if and only  $b_1,b_2,b_3,b_4$ satisfy some congruences modulo $W^2N(\gc)$.
We also need to impose that $\gc | (\beta)$. This is also equivalent to some congruences conditions modulo $W^2 N(\gc)$ for $b_1,b_2,b_3,b_4$. 
 Let $q_1=[q,W^2N(\gc)]$ and $\cV_0'\subset \{ 0\ldots , q_1-1\}^4$ the set of r classes satisfying all these conditions and furthermore such that 
 \begin{equation*}
 (\d\diamond\b )_4 \equiv 0\Mod{q_1}\ \text{and}\ \frac{(\d\diamond\b)_i}{W^2N(\gc)}\equiv {(\u_0)_i}\Mod q\ \text{for}\
 1\le i\le 3.
 \end{equation*}
Thus, for $\gd\in\cC$, we are interested in 
 \begin{equation*}
 |\Gamma_\gd |=\sum_{\b_0\in V_0'}\sum_{\substack{{\b\in\Z^4}\\ {\b\equiv\b_0\Mod {q_1}}\\ {\delta_\gc \beta\in \Gamma}}}1.
 \end{equation*}
 The rest of the proof follows \cite{May15b}.  Let $\Lambda _\d$ be the lattice introduced in Lemma \ref{de40}.
 We write $\b=\b ^{(1)}+q_1\b^{(2)}$ where $\b^{(1)}$ is some vector of $\Lambda_{\d}$ such that 
 $\b^{(1)}\equiv \b_0\Mod {q_1}$ (when such $\b^{(1)}$ exists) and $\b^{(2)}\in\Lambda_{\gd_\gc}$
 \begin{equation*}
 |\Gamma_\gd |=\sideset{}{''}{\sum}_{\b_0\in V_0'}{}\sum_{\substack{{\b^{(2)}\in\Lambda_\d}\\ {\b\equiv\b_0\Mod {q_1}}\\ {\delta_\gc \beta_1+
 q_1\delta_\gc\beta_2\in \Gamma}}}1,
 \end{equation*}
 where $\sum ''$ indicates that the $\b_0$ are as above  but furthermore such that there exists a vector $b^{(1)}$ in the lattice $\Lambda_\gd$ and $\beta_j=\frac{1}{W}\sum_{i=1}^4
 b_i^{(j)}\nu_i$ for $j=1,2$. The argument now follows the proof of \cite[Proposition 7.5]{May15b} precisely, except that we apply Lemmas \ref{div}, \ref{WTI} for the basis $\v$ in place of \cite[Lemmas 7.3 and 7.4]{May15b}.
 \end{proof}
 %
%

\subsection{Initial steps in the Type II sum}

%
%
 If $\a\in\cA_{q_1\cdots q_\ell}(\u_0,m,p)$ then there exists $d\in\N$ such that $N(a_1\nu_1+a_2\nu_2+a_3\nu_3)=d\prod_{i=1}^\ell q_i$.
 The conditions on $q_i$ imply that $(m,q_1\cdots q_\ell )=1$ but in general, it is not clear that $(d,m)=1$.
 This may gives some complications in the application of Proposition \ref{prpstn:TypeI}.
 Let us write $m_0=(d,m^\infty)$ and recall the notation 
 $\cX=\prod_{i=1}^3 [X_i, X_i (1+\eta_1)[$ from \eqref{eq:cXDef}. In almost cases, $m_0$ is small. The contribution of the $\a\in\cX$, such that $\a\equiv\u_0\Mod m$ and $m_0>D_0$
 with  $D_0=\eta_\cX^{-2}\eta_1^{-1}$ with $\eta_\cX$ defined by \eqref{defeta2} below,  is less than 
  \begin{equation*}\begin{split}
  \sum_{\substack{ {m_0| m^\infty} \\ 
  {m_0>D_0}
  }} &
  \sum_{\substack{ {\a\in\cX}\\
  {\a\equiv\u_0\Mod m}\\
 {N\big (\sum_{i=1}^3a_i\nu_i\big )\equiv 0\Mod{m_0}} 
  }} 1
  \ll \eta_1^3\prod_{i=1}^3X_i\sum_{\substack{ {m_0| m^\infty} \\ 
  {m_0>D_0}
  }}
\frac{4^{\omega (m_0)}}
{m_0m^2}\\
&\qquad\ll \frac{\eta_1^3\prod_{i=1}^3X_i}{m^2\sqrt{D_0}}\sum_{\substack{ {m_0| m^\infty} \\ 
  {m_0>D_0}
  }}
\frac{4^{\omega (m_0)}}
{\sqrt{m_0}}\ll \frac{\eta_1^3\prod_{i=1}^3X_ic^{\Omega (m)}}{m^2\sqrt{D_0}},
\end{split}
\end{equation*}
for some $c>0$ large enough.
This contribution is sufficiently small.

We now suppose that $m_0\le D_0$.
 
 Let 
 \begin{equation}\label{defMm0}\begin{split}
 \cM (m_0)&=\Big\{ \v_0\in [1,mm_0]^3 : \v_0\equiv{u_0}\Mod m,\\
 &\ N(\sum_{i=1}^3  (\v_0)_i\nu_i)\equiv 0\Mod {m_0},\ 
 \big (m, \frac{N(\sum_{i=1}^3  (\v_0)_i\nu_i}{m_0}\big )=1\Big\}.\end{split}
 \end{equation}
 Then for every $\a\in\cA (\u_0,m)$ such that $m_0 =(N(a_1\nu_1+a_2\nu_2+a_3\nu_3),m^\infty)$, 
 there exists exactly one $\v_0\in\cM (m_0)$ such that $\a\equiv \v_0\Mod {mm_0}$.
 
 We deduce that 
 \begin{equation}
 \sum_{\substack{{\a\in\cX} \\ {\a\equiv \u_0\Mod m}\\
  {(N(\sum_{i=1}^3 a_i\nu_i) ,m^\infty)\le D_0}}}1=\sum_{\substack {{m_0|m^\infty}\\ {m_0\le D_0}}}\sum_{\v_0\in\cM (m_0)}|\cA (\v_0 ,mm_0)|.
  \label{eq:m0Split}
  \end{equation}

  Any $\a\in \cA (\v_0,mm_0)$ is such that the associated ideal $(\sum_{i=1}^3 a_i\nu_i)$ may be factored as 
  $(\sum_{i=1}^3 a_i\nu_i)=\gM_0 \gJ$ with $N(\gM_0)=m_0$ and $(N(\gJ),m)=1$. This property will  simplify some GCD considerations in the next sections.  Let 
 \begin{equation}\label{defm'}
 m':=m_0m
\end{equation}
denote this extended modulus.

 %
%
  
 \subsection{Switching to ideals with norms in small boxes}
 
 %
%
 
 We introduce the sets of principal ideals of $\cO_K$
 \begin{equation}\label{AX}
 \cAI =\Big\{ \big (\sum_{i=1}^3 a_i\nu_i \big ) :  \a\in\cX\Big\}.
 \end{equation}
For any $\ga\in\cAI $ there is exactly one $(a_1,a_2,a_3)\in\cX$ such that $\ga =(a_1\nu_1+a_2\nu_2+a_3\nu_3)$. 
We justify this in a similar way as  in \cite[Proof of Lemma 5.2 assuming Proposition 5.1 pp. 13-14]{May15b}. 

If $\alpha = \sum_{i=1}^3 a_i\nu_i$  and $\beta =\sum_{i=1}b_i\nu_i$ with $\a,\b\in\cX$ are such that $(\alpha )=(\beta)$
then $\beta\alpha^{-1}$ is a unit of $\cO_K$.
But $|\sigma ( \alpha )|\ll X$  for all embedding $\sigma$ and since $\alpha =N(\alpha)\prod_{\sigma\not = Id}\sigma (\alpha )^{-1}$ we have 
$|a_1\nu_1+a_2\nu_2+a_3\nu_3|\gg \eta_1^{1/10}X$ by \eqref{eq:Xi2}
and then
\begin{equation*}
\frac{\beta}{\alpha}  =1 +\frac{\beta-\alpha}{\alpha}= 1 + O (\eta_1 ^{9/10}).
\end{equation*}
If $\alpha\not =\beta$ then $\beta\alpha^{-1}$ can't be   a unit because the length  between two  units is $\gg 1$ and we have a contradiction.

Next we consider the sets 
\begin{equation*}
\cAI (\v_0 , m', p)=\Big\{ \big (\sum_{i=1}^3a_i\nu_i)\in\cAI  : \a\equiv {\v_0}\Mod {m'}\ \text{and}\  p|f(a_1,a_2,a_3)\Big\}
\end{equation*}
and for any ideal $\gd$, 
\begin{equation*}
\cAI _\gd (\v_0,m', p)=\{\ga \in\cAI (\v_0 , m', p) :\gd |\ga\}.
\end{equation*}

Let $N_0^4=\min_{(\alpha )\in\cAI}N(\alpha)$.
Let $\eta_\cX$ and $\eta_2$ defined by 
\begin{equation}\label{defeta2}
\eta_\cX=\frac{1}{(\log X)^A},\qquad \eta_2=\eta_\cX^{10000\ell^2}.
\end{equation}

By the definition of $\cX$, $N(a_1\nu_1+a_2\nu_2+a_3)\in [N_0^4, N_0^4(1+O(\eta_1))]$ for all 
$(a_1\nu_1+a_2\nu_2+a_3)\in \cAI$. We can choose $O(\eta_2^{-1}\eta_1)$ reals $X_0$ with $X_0^4\in [N_0^4, N_0^4(1+O(\eta_1))]$ so that
 the sets 
\begin{equation*}\label{A'}
\cAI (X_0,\v_0 ,m',p)=\Big\{ \big (\sum_{i=1}^3 a_i\nu_i\big )\in\cAI (\v_0 ,m',p) : N\big (\sum_{i=1}^3 a_i\nu_i\big  )
\in [X_0^4, X_0^4+\eta_2X_0^4[\Big\},
\end{equation*}
form a partition of $\cAI (\v_0,m',p)$.
 Next we introduce the sets 
\begin{equation*}
\cAI_\gd (X_0,\v_0, m',p)=\Big\{ \big (\sum_{i=1}^3 a_i\nu_i\big )\in\cAI (X_0,\v_0 ,m',p) : \gd | \big (\sum_{i=1}^3 a_i\nu_i\big )
\Big\}.
\end{equation*}

By \eqref{eq:th4}, there exists $\varepsilon>0$ such that $X^{\sum_{i=1}^\ell \theta_i +\min (\theta_0,\ldots , \theta_\ell)}>X^{4+\varepsilon}$ and by \eqref{eq:th2} the intervals $[X^{\theta_i},X^{\theta_i'}]$ do not overlap. Thus for each $\ga \in\cAI$ 
such that $N(\ga )\equiv 0\Mod{q_1\cdots q_\ell}$ with 
$X^{\theta_i}\le q_i\le X^{\theta_i'}$, is divisible by exactly one prime ideal $\gP_i$ with 
$N(\gP_i)\in [X^{\theta_i}, X^{\theta'_i}]$ (for all $1\le i\le \ell$).

We are now ready to settle the connection between the set $\cA_q ( \v_0, m',p)$ in Theorem \ref{DistriNorm} and the  sets of ideals just defined above. For any primes $q_1,\ldots ,q_\ell$ with $q_i\in [X^{\theta_i}, X^{\theta_i'}]$, we have
\begin{equation}\label{corresp}
|\cA_{q_1\cdots q_\ell} ( \v _0, m',p)|=\sum_{X_0}\sum_{N(\gP_i)=q_i}|\cAI _{\gP_1\cdots \gP_\ell}(X_0,\v_0,m',p)|.
\end{equation}

Any ideal  $(a_1\nu_1+a_2\nu_2+a_3\nu_3)$ counted in \eqref{corresp} may be factored as 

\begin{equation}\label{facto}
 (a_1\nu_1+a_2\nu_2+a_3\nu_3)=\gM_0\gI\prod_{i=1}^\ell \gP_i,
\end{equation}
where each $\gP_i$ is a prime ideal with norm in $[X^{\theta_i}, X^{\theta'_i}]$ and $\gI$ is an ideal with 
\begin{equation}\label{lesI}
N(\gI)\in  \cI_0:=\Bigl[\frac{X_0^4X^{-\sum_{i=1}^\ell \theta_i'}}{m_0},\frac{ X_0^4(1+\eta_2)X^{-\sum_{i=1}^\ell\theta_i}}{m_0}\Bigr]=[I_1,I_2],
\end{equation}
say.

We  choose now  $O(\eta_2^{-1}\log X)$  reals $I\in \cI_0$ such that $\cI_0$ is covered by
 the union of the intervals $[I, I(1+\eta_2)[$. Let $\hat \cI_0$ denote the set of these reals $I$.

 Since we have $(N(\sum_{i=1}^3a_i\nu_i)/m_0 ,m)=1$ when $\a\equiv\v_0\Mod{m'}$, we have 
 $(m',N(\gI)\prod_{i=1}^\ell N(\gP_i))=1$.

 For brevity we will write  $\cAI (\v_0,m',p)$ in place of $\cAI (X_0,  \v_0,m',p)$ when the context will be clear.

To have  a precise control of the size of the norms of some ideals,
we  cover  each interval $[\theta_i, \theta'_{i}]$    by $O(\eta_\cX^{-2})$ distinct intervals  of size $O(\eta_\cX^2)$
so that,
\begin{equation}\label{UnionRl}
\prod_{i=1}^\ell [\theta_i,\theta'_i]=\cup_{\bl \in E}\cR (\bl),
\end{equation}
where $E$ is some subset of $\N^\ell$ and each $\cR (\bl)$ is of type $\cR (\bl)=\prod_{i=1}^\ell [t_i,t'_i )$
with $|t_i'-t_i|\ll \eta_\cX^2$  (except that in  the  intervals with $t'_i=\theta_i'$  we take the whole   segment $[t_i,\theta_i' ]$),   (cf \cite[section 8 p.45]{May15b}).

We write $\cR (\bl)=\prod \cR_1 (\bl)\times \cR_2 (\bl)$ with $\cR_2(\bl)$ representing the first $\ell '$ coordinates and $\cR_1 (\bl)$ the final  $\ell -\ell'$ coordinates.

 For a polytope $\cR\subset\R^s$ (for some $s$), we define 
 \begin{equation*}
 \1_\cR (\ga )=\begin{cases}
 1, & \ga =\gp _1\cdots \gp _s\ {\rm with}\ N(\gp _i)=X^{e_i}\ {\rm and }\ (e_1,\ldots ,e_s )\in\cR,\\
 0, & {\rm otherwise.}
 \end{cases}
 \end{equation*}
Thus we need to study the quantity 
\begin{equation}
T(\cR(\bl )):=\sum_{X^\tau\le p\le X^{\tau'}}\sum_{\gI\in\cI }\sum_{\gM_0\gI\ga\in \cA (\v_0,m',p)}\1_{\cR (\bl)}(\ga),
\label{eq:TRDef}
\end{equation}
with $\cI :=\{ \gI : N(\gI)\in [I,I+\eta_2I[, (N(\gI),m)=1\}$.

\subsection{Approximation weights}

We recall that $\eta_2=\eta_\cX^{10000\ell ^2}$.
A key idea of \cite{May15b} is to approximate the indicator $\1_{\cR_2}$ by a weight $\tilde \1_{\cR_2}$
which will be more easy to control. 
For $\cS\subset \R^s$ we consider the function 
\begin{equation}
c_\cS (t)=\iint_{\substack{{(e_1,\ldots ,e_s)\in\cS}\\ {\sum_{i=1}^s e_i\in I_t}}}
\frac{de_1\cdots de_{s }}{\eta_2^{1/2}\prod_{i=1}^s e_i},\label{eq:CDef}
\end{equation}
\begin{equation*}
{\rm where}\ I_t :=\Big [\frac{\log t}{\log X},\frac{\log (t+\eta_2^{1/2}t)}{\log X}\Big ].
\end{equation*}
In this previous definition we have $\sum_{i=1}^\ell e_i\in I_t$ if and only if $X^{\sum_{i=1}^\ell e_i}\in [t,t (1+\sqrt{\eta_2}]$.
This function is so that $c_\cS (N(\ga))$ corresponds to the probability for an ideal of norm close to $N(\ga)$ to 
have a prime factorisation compatible  with $\cS$  (cf. \cite[section 8]{May15b}).
We recall below some properties of this function that we will frequently use later on.
\begin{lmm}\label{cR}
\begin{itemize}
\item If $\cS=\prod_{i=1}^s [s_i,s'_i]$ is an hyperrectangle with $\min s_i >\varepsilon_0 >0$ and $\ell >1$, then 
\begin{equation*}
c_\cS (t+\delta)-c_\cS (t)\ll\frac{\delta}{t}
\end{equation*} 
\item If $\cS=\prod_{i=1}^s [s_i,s'_i]$ is an hyperrectangle  with $\min s_i >\varepsilon_0 >0$    then  
\begin{equation*}
c_\cS (t)\ll_{\varepsilon_0} \frac{1}{\log X}.
\end{equation*}
\end{itemize}
\end{lmm}
\begin{proof}
The first part is a particular case of \cite[Lemma 8.3  (iii)]{May15b}.
The proof of the second point is a direct computation analogous to \cite{May15b} :
\begin{equation*}
c_\cS (t)\le\frac{1}{\sqrt{\eta_2}}\iint_{\substack{{e_i\in [s_i , s_i']}\\ {1\le i\le s -1}}}\Big [\int_{e_s\in I_t-\sum_{i=1}^{s-1}e_i}\frac{\d e_s}{s_s}\Big ]\prod_{i=1}^{s -1}\frac{\d e_i}{s_i}.
\end{equation*}
The integral over $e_\ell$ is $O(\sqrt{\eta_2}(\log X)^{-1})$ and the contribution of the other integrals is $O(1)$.
\end{proof}

Let $\varepsilon_{00}>0$ and $R:=X^{\varepsilon_{00}}$. The approximate weights of $\1_{\cR_2}$ are defined by
\begin{equation}
\tilde \1_{\cR_2}(\gb):=c_{\cR_2}(N(\gb ))\sum_{\gd |\gb}\lambda_\gd,\label{eq:TildeDef}
\end{equation}
where
\begin{equation*}
\lambda _\delta := \begin{cases} 
\mu (\gd )\log \frac{R}{N(\gd )}, & N(\gd)<R,\\
 0, & {\rm otherwise.}
\end{cases}
 \end{equation*}
\begin{rmk}
Our weights are somewhat simpler than the one introduced in \cite{May15b}, because we don't need to take care of the perturbations caused by a possible exceptional character $\chi ^*$. (Ultimately we will only require estimates with moduli up to a fixed power of $\log{X}$, whereas in \cite{May15b} larger moduli needed to be considered due to losses occurring in high dimensions.)
\end{rmk}
We now write
\[ 
T(\cR)=T_{sieve}(\cR )
+T_1(\cR ),
\] 
where 
\begin{align}
T_{sieve}(\cR)&:=\sum_{\substack{{p\in [P_1,P_2]}\\ {p\equiv 1\Mod {D_f}}}}\sum_{\gI\in\cI}\sum_{\gM_0\gI\ga\gb\in\cAI(\v_0,m',p)}\1_{\cR_1}(\ga)\tilde\1_{\cR_2}(\gb),\label{eq:TSieveDef}\\
 T_1(\cR)&:=\sum_{\substack{{p\in [P_1,P_2]}\\ {p\equiv 1\Mod{D_f}}}}\sum_{\gI\in\cI}\sum_{\gM_0\gI\ga\gb\in\cAI (\v_0 ,m',p)}\1_{\cR_1}(\ga)(\1_{\cR_2}(\gb)-\tilde\1_{\cR_2}(\gb)),\label{TsT1}
\end{align}
and
\[
  P_1:=X^{\tau}, \qquad
     P_2:=X^{\tau'}.
\]

For brevity again we will write $T_{sieve}(\cR)$ and $T_1 (\cR)$ in place of $T_{sieve}(\cR , \v_0)$ and $T_1 (\cR ,\v_0)$ when $\v_0$ is clear from the context. We see that Theorem \ref{DistriNorm} follows immediately from the following two propositions.
%
%
\begin{prpstn}[Estimate for $T_{sieve}$]\label{prpstn:TSieve}
If we have
\[\epsilon_{00}<\sum_{j=1}^{\ell'}\theta_j-1-12\tau',\]
then
\begin{align*}
T_{sieve}(\cR)=(2+O(\eta_2^{1/2}))\log({1+\eta_2})|\cAI(X_0)| c_{\cR_1\times \cR_2}(X_0^4/mI)\frac{g(m')}{m'{}^3}\frac{\log(P_2/P_1)}{\varphi(D_f)},
\end{align*}
with
\begin{align*}
    g(m')=\prod_{\gP| (m')}\Big ( 1-\frac{\varrho_\v (\gP)}{N(\gP)}\Big )^{-1}.
\end{align*}    
\end{prpstn}
%
%
\begin{prpstn}[Bound for $T_1(\cR)$]\label{prpstn:T1}
Let $\cR=\cR_1\times\cR_2$ and $T_1(\cR)$ be as above. If we have
\[
\tau'<\min\Bigl(\frac{4-2\theta'_{1}-\ldots-2\theta'_{\ell '}}{100},\frac{\theta_1+\cdots+\theta_{\ell'}-1}{100}\Bigr),
\]
then for any $A>0$ we have
\[
T_1(\cR)\ll_A \frac{|\cAI(X_0)|}{(\log{X})^A}. 
\]
\end{prpstn}
%
%
We remark that we are assumming the general setup in Propositions \ref{prpstn:TSieve} and \ref{prpstn:T1}; in particular, the constants $\theta_1,\theta_1',\dots,\theta_\ell,\theta_{\ell}'$ determining $\cR$ are assumed to satisfy \eqref{eq:th1}-\eqref{eq:th5}.

%
%
We will establish Proposition \ref{prpstn:TSieve} in Section \ref{sec:TSieve} and the harder Proposition \ref{prpstn:T1} in Section \ref{sec:T1}. The presence of the sum over primes $p\in [P_1,P_2]$ introduces few additional complications to $T_{sieve}$ and Proposition \ref{prpstn:TSieve}, but quite significant additional technical details to $T_1$ and Proposition \ref{prpstn:T1}. Assuming these propositions for now, we can establish Theorem \ref{DistriNorm} by putting all our manipulations together.

\begin{proof}[Proof of Theorem \ref{DistriNorm} assuming Propositions \ref{prpstn:TSieve} and \ref{prpstn:T1}]
We recall from \eqref{eq:m0Split} and \eqref{AX}
that
\begin{equation}
\cA_{q_1\cdots q_\ell}(\u_0,m,p)=\sum_{\substack{m_0|m^\infty \\ m_0\le D_0}}\sum_{\v_0\in\cM(m_0)}|\cAI_{q_1\cdots q_\ell}(\v_0,m',p)|+O\Bigl(\frac{\eta_1^3\prod_{i=1}^3 X_i}{m^2 D_0^{1/2}}\Bigr).
\label{eq:AAt}
\end{equation}
We concetrate on the $\cAI$ terms. We recall from \eqref{corresp}, \eqref{lesI} and \eqref{eq:TRDef} that
\begin{align*}
&|\cAI_{q_1\cdots q_\ell} ( \v _0, m',p)|=\sum_{X_0}\sum_{N(\gP_i)=q_i}|\cAI _{\gP_1\cdots \gP_\ell}(X_0,\v_0,m',p)|\\
&=\sum_{X_0}\sum_{\substack{\cI \\ \sqcup \cI=\cI_0}}\sum_{N(\gP_i)=q_i}\sum_{\substack{\gJ\\ N(\gJ)\in\cI \\ \gM_0\gJ\prod_{i=1}^\ell\gP_i\in \tilde{A}(\v_0,m',p)}}1\\
&=\sum_{X_0}\sum_{\substack{\cI \\ \sqcup \cI=\cI_0}}\sum_{\substack{\cR_1,\cR_2 \\ \prod_{i=1}^\ell [\theta_i,\theta_i']=\sqcup \cR_1\times\cR_2}}\sum_{\substack{\gJ\\ N(\gJ)\in\cI}}\sum_{\ga}\1_{\cR_1}(\ga)\sum_{\substack{\gb \\ \gJ\gM_0\ga\gb\in \tilde{A}(\v_0,m',p)}}\mathbf{1}_{\cR_2}(\gb).
\end{align*}

By assumption of Theorem \ref{DistriNorm}, we have $\tau'<(\sum_{i=1}^{\ell'}\theta_i-1)/100$, and so choosing $\epsilon_{00}$ sufficiently small means that the hypothesis of Proposition \ref{prpstn:TSieve} is satisfied. Thus, summing over $p\in [P_1,P_2]$ and applying Propositions \ref{prpstn:TSieve} and \ref{prpstn:T1} (with a suitably large constant $A$) gives
\begin{align}
&\sum_{p\in[P_1,P_2]}|\cAI_{q_1\cdots q_\ell} ( \v _0, m',p)|=\sum_{X_0}\sum_{\substack{\cI \\ \sqcup \cI=\cI_0}}\sum_{\substack{\cR=\cR_1\times\cR_2 \\ \prod_{i=1}^\ell [\theta_i,\theta_i']=\sqcup \cR_1\times\cR_2}}\Bigl(T_{sieve}(\cR)+T_1(\cR)\Bigr)\nonumber\\
&=(2+O(\eta_2^{1/2}))\eta_2\frac{\log(P_2/P_1)}{\phi(D_f)}\frac{g(m')}{m'{}^3}T_3+O\Bigl(\frac{\prod_{i=1}^3X_i}{(\log{X})^{A-O(1)}}\Bigr),\label{eq:AtT3}
\end{align}
where
\[
T_3:=\sum_{X_0}\sum_{\substack{\cI \\ \sqcup \cI=\cI_0}}\sum_{\substack{\cR=\cR_1\times\cR_2 \\ \prod_{i=1}^\ell [\theta_i,\theta_i']=\sqcup \cR_1\times\cR_2}}|\cAI(X_0)|c_{\cR}(X_0^4/mI).
\]
Here we used that there are at most $O(\eta_\cX^{-2})$ subsets $\cR $, $O(\eta_2^{-1}\eta_1)$ reals $X_0$ and  $O(\eta_2^{-1}\log X)$ reals $I$ to bound the contribution from $T_1$ by  Proposition \ref{prpstn:T1}.

We now concentrate on $T_3$. Since the subsets  $\cR$ form a partition of $\cT:=\prod_{i=1}^\ell [\theta_i,\theta_i']$, we find
$$\sum_{\substack{\cR=\cR_1\times\cR_2 \\ \prod_{i=1}^\ell [\theta_i,\theta_i']=\sqcup \cR_1\times\cR_2}}c_{\cR}(X_0^4/mI)=c_\cT (X_0^4/mI),$$
so
\[
T_3=\sum_{X_0}|\cAI(X_0)|\sum_{\substack{\cI \\ \sqcup \cI=\cI_0}}c_\cT (X_0^4/mI).
\]
By Lemma \ref{cR} applied to $c_\cT$ we have for all  $I\in\cI_0$
 \begin{equation*}\begin{split}
 c_\cT (X_0^4/mI)&=\frac{1}{\eta_2}\int_{I}^{I(1+\eta_2)}\frac{c_\cT (X_0^4/mv)}{v}dv+O(\eta_2).\\
 \end{split}
 \end{equation*}
 Expanding the definition \eqref{eq:CDef} of $c_{\cT}$ and swapping the order of summation and integration, we find
 \begin{align*}
\sum_{I\in\hat\cI_0}\int_{I}^{I(1+\eta_2)}&\frac{c_\cT (X_0^4/mv)}{v}dv\\
&=\frac{1}{\eta_2^{1/2}}
\iint\limits_{\substack{{e_i\in [\theta_i,\theta_i']}\\ {1\le i\le\ell}}}
\sum_{I\in\hat\cI_0}\int\limits_{\substack{v\in [I,I(1+\eta_2)[\\ v\in \Big [\frac{X_0^4}{m\prod_{i=1}^\ell X^{e_i}}, \frac{X_0^4(1+\sqrt{\eta_2})}{m\prod_{i=1}^\ell X^{e_i}}\Big ]}}
\frac{dv}{v}\prod_{i=1}^\ell \frac{\d e_i}{e_i}\\
&=\frac{1}{\eta_2^{1/2}}\iint\limits_{\substack{ e_i\in [\theta_i,\theta_i'] \\ 1\le i\le\ell }}\Bigl(\int_{X_0^4/(m\prod_{i=1}^\ell X^{e_i})}^{X_0^4(1+\sqrt{\eta_2})/(m\prod_{i=1}^\ell X^{e_i})}\frac{dv}{v}\Bigr)\prod_{i=1}^\ell \frac{\d e_i}{e_i}\\
&=\frac{\log(1+\sqrt{\eta_2})}{\eta_2^{1/2}}\prod_{i=1}^\ell \log\Big (\frac{\theta_i'}{\theta_i}\Big ).
\end{align*}
We note that this is independent of $X_0$, so we find
\begin{align}
T_3&=\frac{\log(1+\sqrt{\eta_2})}{\eta_2^{3/2}}\prod_{i=1}^\ell \log\Big (\frac{\theta_i'}{\theta_i}\Big )\sum_{X_0}|\cAI (X_0)|+O\Bigl(\log{X}\sum_{X_0}|\cAI (X_0)|\Bigr)\nonumber\\
&=\frac{(1+O(\sqrt{\eta_2}))}{\eta_2}\prod_{i=1}^\ell \log\Big (\frac{\theta_i'}{\theta_i}\Big )|\cAI (X)|.\label{eq:T33}
\end{align}
Putting together \eqref{eq:AAt}, \eqref{eq:AtT3} and \eqref{eq:T33} we find
\begin{align}
\sum_{p\in[P_1,P_2]}&\sum_{\substack{q_1,\dots q_\ell \\ q_i\in [X^{\theta_i},X^{\theta_i'}]}}\cA_{q_1\cdots q_\ell}(\u_0,m,p)\nonumber\\
&=2\frac{\log{\frac{P_2}{P_1}}}{\phi(D_f)}\prod_{i=1}^\ell \log\Big (\frac{\theta_i'}{\theta_i}\Big )|\cAI (X)|\sum_{\substack{m_0|m^\infty \\ m_0\le D_0}}\sum_{\v_0\in\cM(m_0)}\frac{g(m')}{m'{}^3}\nonumber\\
&\qquad+O\Bigl(\eta_2^{1/2}\prod_{i=1}^3 X_i\Bigr)+O\Bigl(\frac{\eta_1^3\prod_{i=1}^3 X_i}{D_0^{1/2}}\Bigr).
\label{eq:MainEst}
\end{align}
Finally it remains to estimate the inner double sum. The summand is independent of $\v_0$, so recalling that $m'=m_0m$ we are left to estimate
\begin{equation}\label{Sm0}
\sum_{\substack{{m_0<D_0}\\
{m_0|m^\infty}}}\frac{|\cM (m_0)|}{(mm_0)^3}\prod_{\gP| (m)}\Big (\sum_{k=2}^\infty \frac{\varrho_\v(\gP^k)}{N(\gP^k)}\Big )^{-1}.
\end{equation}
By \eqref{defMm0}, 
$$|\cM(m_0)|\le m_0^2m,$$ and thus for any given $m$ 
the sum over $m_0$ converges. We may therefore extend it to all $m_0\ge 1$ cost of an admissible error term.
Next we note that the sets of the $\a\in [1,X]^3$
with $\a\equiv \u_0\Mod m$ can be partitioned into sets of the $\a\in [1,X]^3$ such that $\a\equiv{\v_0}\Mod {mm_0}$, with $m_0\le X^2$ and $\v_0\in\cM (\u_0)$, and so
\begin{equation*}
    \begin{split}
  \sum_{\substack{{m_0<D_0}\\
{m_0|m^\infty}}}\frac{|\cM (m_0)|}{(mm_0)^3}&=
(1+O(D_0^{-1/4}))\sum_{\substack{{m_0<X^2}\\
{m_0|m^\infty}}}\frac{|\cM (m_0)|}{(mm_0)^3}\\
&= \frac{(1+O(D_0^{-1/4})}{X^3+O(X^2)}
\sum_{\substack{{m_0<X^2}\\
{m_0|m^\infty}}}\sum_{\v_0\in\cM (m_0)}
\sum_{\substack{{\a\in [1,X]^3}\\ {\a\equiv{\v_0}\Mod{m_0}}}}1\\
&=\frac{(1+O(D_0^{-1/4}))}{X^3+O(X^2)}
\sum_{\substack{{\a\in [1,X]^3}\\ {\a\equiv{\u_0}\Mod{m_0}}}}1\\
&=\frac{1}{m^3}(1+O(D_0^{-1/4})).
\end{split}
\end{equation*}
Substituting this into \eqref{eq:MainEst} and recalling $D_0=\eta_\cX^{-2}\eta_1^{-1}\gg (\log{X})^{2A}$, $\eta_2\ll (\log{X})^{-2A}$ and $|\cAI(X)|=\eta_1^3X_1X_2X_3+O(\eta_2X_1X_2X_3)$ gives Theorem \ref{DistriNorm}.
\end{proof}
%
%
%
\section{Proposition \ref{prpstn:TSieve}: The term $T_{sieve}(\cR)$.}\label{sec:TSieve}
%
%
In this part we obtain an analogue of \cite[Lemma 8.6]{May15b} by expanding the sieve terms and applying Proposition \ref{prpstn:TypeI}. 

If $\ga$  and $\gb$ are  some ideals satisfying $\1_{\cR_1} (\ga )=\1_{\cR_2}(\gb)=1$ then $\ga$ and $\gb$
factor into prime ideals as $\ga =\prod_{i=\ell'+1}^\ell \gP_i$, $\gb=\prod_{i=1}^{\ell '} \gP_i$
with $N(\gP_i)\in [X^{t_i},X^{t_i}(1+O(\eta_\cX^2\log X))]$ for $1\le i\le\ell$.
In particular, 
\[
N(\ga)\in [A^4, A^4(1+O(\eta_\cX^2\log X))],\quad N(\gb)\in  [B^4, B^4(1+O(\eta_\cX^2\log X))]
\]
where
\[
A^4:=X^{\sum_{i=\ell'+1}^\ell t_i},\quad  B^4:=X^{\sum_{i=1}^{\ell '}t_i}.
\]
%
%
\begin{lmm}\label{lmm:TSplit}
Let $B^4>X^{1+\epsilon}R$ and $\cR=\cR_1\times \cR_2$. Then we have
\[
T_{sieve}(\cR)=M_1(\cR)+E_1(\cR)
\]
where $M_1(\cR)$ is given by
\begin{align*}
M_1(\cR)&:=\sum_{\substack{ {p\in [P_1,P_2]}\\ {p\equiv 1\Mod{D_f}}}}
\sum_{\gI\in\cI}\sum_{\ga}\1_{\cR_1}(\ga )\sum_{\substack{{N(\gd )<R}\\ {N(\gd),m)=1}}}\lambda_\gd\nonumber\\
&\qquad\times c_{\cR_2}\Big (\frac{X_0^4}{m_0N(\a \gI)}\Big )|\cAI_{\ga\gd\gI}(\v_0(\y),m',p)|,
\end{align*}
and $E_1(\cR)$ satisfies
\[
\sum_{\cR}\sum_{X_0}|E_1(\cR)|\ll \eta_2^{1/2}\eta_1^{-2\ell}(\log X)^{11}\prod_{i=1}^3X_i.
\]
\end{lmm}
%
%
\begin{proof}
We substitute our definition \eqref{eq:TildeDef} of $\tilde \1_{\cR_2}$ into our expression \eqref{eq:TSieveDef} for $T_{sieve}$, and write $\gu=\gM_0\gI\ga\gb$. This gives
\begin{equation}\label{MR}\begin{split}
T_{sieve}(\cR)&=\sum_{\substack{{p\in[P_1,P_2]}\\ {p\equiv 1\Mod{D_f}}}}\sum_{\gI\in\cI}\sum_{\ga}\1_{\cR_1}(\ga )\sum_{N(\gd )<R}\lambda_\gd
\sum_{\substack{{\gu\in\cAI(\v_0,m',p)}\\ {\gM_0\gI\ga\gd |\gu}\\ }}
c_{\cR_2}(N(\gu /\a \gI\gM_0)).\end{split}
\end{equation}
  If $\gu\in \cAI (\v_0,m',p)$ then $N(\gu )\in [X_0^4,X_0^4(1+\eta_2)]$. By  Lemma \ref{cR}, this implies
$c_{\cR_2}(N(\gu /\a \gI\gM_0))=c_{\cR_2}(X_0^4/m_0N(\a \gI))+O(\eta_2).$ Thus we write
\begin{equation}
T_{sieve}(\cR)=M_1(\cR)+O(E_1(\cR)),
\end{equation}
where $M_1(\cR)$ is as given in the Lemma and
\begin{align}
E_1(\cR)&:= \eta_2
\sum_{\substack{{p\in[P_1,P_2]}\\ {p\equiv 1\Mod{D_f}}}}\sum_{\gI\in\cI}\sum_{\ga}\1_{\cR_1}(\ga)\sum_{N(\gd )<R}|\lambda _\gd||\cAI _{\gI\ga\gd}(\v_0,m',p)|.\label{cc}
\end{align}
We concentrate on $E_1(\cR)$. For any $(\sum_{i=1}^3x_i\nu_i)\in\cAI $, the number of primes $p\in [P_1,P_2]$ such that $p| f(x_1,x_2,x_3)$ is finite. This allows us to remove the summation over $p$ and replace $|\cAI _{\gI\ga\gd}(\v_0,m',p)|$ with $|\cAI _{\gI\ga\gd}(\v_0,m',1)|$ in $E_1(\cR)$ at the cost of a factor $O(1)$.

 We then apply Proposition \ref{prpstn:TypeI} to estimate $|\cAI _{\gI\ga\gd\gM_0}(\v_0,m',1)|$, recalling that $N(\gI\ga\gd\gM_0)\ll X^4 R/B^4$ and $m'\ll (\log X)^{O(1)}$. This gives
\begin{equation}\label{resteT1}\begin{split}
E_1(\cR)&\ll \eta_2
\sum_{\gI\in\cI}\sum_{\ga}\1_{\cR_1}(\ga)\sum_{\substack{{N(\gd )<R} \\ {(N(\gd ),m')=1}}}|\lambda _\gd|\frac{|\cAI (X_0,\cX) |\varrho_\v (\ga\gd\gI ) }{N(\ga\gI\gd)(m')^3}\\
&+X^{o(1)}\sum_{\substack{{N(\gd )<R}\\ {(N(\gd ),m)=1}}}|\lambda _\gd|\Bigl(X^{2}\Bigl(\frac{X^4 R}{B^4} \Bigr)^{1/3}+\frac{X^4 R}{B^4}X\Bigr).
\end{split}
\end{equation}
Crudely, if $B^4>X^{1+\epsilon}R$, we see the second term in \eqref{resteT1} contributes to \eqref{resteT1}
\begin{equation}\label{E12}
\ll X^{3+o(1)}\Bigl(\frac{X R^4}{B^4}+\Bigl(\frac{X R^4}{B^4}\Bigr)^{1/3}\Bigr)\ll X^{3-\epsilon/4}.
\end{equation}
By an Euler product upper bound and Lemma \ref{rhodrhoI}, we see that the first term contributes to \eqref{resteT1}
\begin{align}
&\ll \eta_2(\log{X}) |\cAI(X_0,\v_0,m')|\sum_{N(\gI),N(\ga),N(\gd)<X}\frac{|\rho_\v(\ga\gd\gI)|}{N(\ga\gd\gI)}\nonumber\\
&\ll\eta_2(\log{X})^{9} |\cAI(X_0,\v_0,m')|.\label{E11}
\end{align}
Thus, substituting \eqref{E12} and \eqref{E11} into \eqref{resteT1} we find for $B^4>X^{1+\epsilon}R$
\[
E_1(\cR)\ll \eta_2(\log{X})^{9} |\cAI(X_0,\v_0,m')|+X^{3-\epsilon/4}.
\]
Summing this over all hyperrectangles $\cR$ and all relevant $X_0$, we find
\begin{align}
\sum_{\cR}\sum_{X_0}E_1(\cR)&\ll \eta_2(\log{X})^{9}\sum_{\cR}\sum_{X_0} |\cAI(X_0,\v_0,m')|+X^{3-\epsilon/5}.\nonumber\\
&\ll \eta_2\eta_\cX^{-2\ell}(\log{X})^4|\cAI (\v_0,m')|+X^{3-\epsilon/5}\nonumber\\
&\ll \eta_2^{1/2}\eta_1^{-2\ell}(\log X)^{11}\prod_{i=1}^3X_i+X^{3-\epsilon/5}.
\end{align}
This gives the result.
\end{proof}
%
%

Thus we have to evaluate $M_1(\cR)$.
%
%
\begin{lmm}\label{lmm:M1}
Let $B^4>X^{1+\epsilon}RP_2^{12}$ and let $M_1(\cR)$ be as given by Lemma \ref{lmm:TSplit}. Then we have
\[
M_1(\cR)=(2+O(\eta_2^{1/2}))(\log{1+\eta_2})|\cAI(X_0)| c_{\cR_1\times \cR_2}(X_0^4/mI)\frac{g(m')}{m'{}^3}\frac{\log(P_2/P_1)}{\varphi(D_f)}.
\]
\end{lmm}
%
%
\begin{proof}
First we want to apply Proposition \ref{prpstn:TypeI} to estimate $|\cAI_{\ga\gd\gJ}(\v_0,m',p)|$. To do this we split according to residue classes $\Mod{p}$. For any $(y_1,y_2,y_3)$ such that $f(y_1,y_2,y_3)\equiv 0\Mod p$ let $\ut_0 (\y)$ be a solution of the two equations $\ut_0(\y)\equiv \y\Mod p$ and $\ut_0 (\y)\equiv \v_0\Mod {m'}$. Thus
\[
|\cAI_{\ga\gd\gI}(\v_0,m',p)|=\sum_{\substack{y_1,y_2,y_3\Mod{p}\\ f(y_1,y_2,y_3)\equiv 0\Mod{p}}}|\cAI_{\ga\gd\gI}(\ut_0(\y),p m',1)|
\]
We recall that $p\le P_2$ and $N(\gI\ga\d)\ll X R/B$. Therefore, by Proposition \ref{prpstn:TypeI}, we can replace $|\cAI_{\ga\gd\gI}(\ut_0(\y),p m',1)|$ with $\rho_{\v}(\ga\gd\gI)|\cA(X_0,\chi)|/p^3m'{}^3N(\ga\gd\gI)$ in $M_1(\cR)$ at the cost of a term bounded by
\[
\sum_{p\le P_2}\sum_{\substack{y_1,y_2,y_3\Mod{p}\\ f(y_1,y_2,y_3)\equiv 0\Mod{p}}}\Bigl(X^{2+o(1)}\Bigl(\frac{X R}{B^4}\Bigr)^{1/3}P_2+X^{o(1)} \frac{X R}{B^4}P_2^4\Bigr).
\]
This is $O(X^{3-\epsilon/4})$ provided $B^4>X^{1+\epsilon}RP_2^{12}$.

Since the function $\varrho_\v$ is multiplicative, $(\ga, \gd\gI)=1$,  and $\ga$ is a product of degree one prime ideals of large enough norm, by Lemma \ref{rhov}, we have 
$\varrho _\v (\ga\gd\gI)/N(\ga\gd\gI)=\varrho_\v (\gd\gI)/(N(\gd\gI)N(\ga))$. Thus
\begin{equation}
M_1(\cR)=M_2(\cR)+O(X^{3-\epsilon/4}),
\label{eq:M1}
\end{equation}
where
\begin{align}
M_2(\cR)&:=  |\cAI(X_0,\cX)|\Bigl(\sum_{\substack{ {p\in [P_1,P_2]}\\ {p\equiv 1\Mod{D_f}}}}\frac{n_p}{p}  \Bigr)\Bigl(\sum_{\gI\in\cI}Z_1(\gI)Z_2(\gI)\Bigr),\label{eq:M2}\\
Z_1(\gI)&:=\sum_\ga \1_{\cR_1}(\ga)\frac{c_{\cR_2}(X_0^4/(m_0N(\a \gI)))}{N(\ga)},\\
Z_2(\gI)&:=\sum_{\substack{{N(\gd )<R}\\ {(N(\gd),m')=1}}}
\lambda_\gd 
\frac{\varrho_\v (\gd\gI)}{N(\gd\gI)(m')^3}\\
n_p&:=\frac{1}{p^2}|\{y_1,y_2,y_3\Mod{p}:\,f(y_1,y_2,y_3)\equiv 0\Mod{p}\}|.\label{eq:NpDef}
\end{align}
First we simplify $Z_1(\gI)$. Since this is a sum of a smooth function over products of $\ell$ prime ideals in a bounded region, this can be estimated using the Prime Ideal Theorem. Following the arguments of \cite[Section 8, proof of Lemma 8.6]{May15b} we find that
\begin{equation*}
Z_1(\gI) =c_{\cR_1\times \cR_2} (X_0^4/m_0N(\gI))+O(\eta_2).
\end{equation*}
We recall that $\cI=\{\gI:(N(\gI),m)=1,\,N(\gI)\in [I,I+\eta_2 I]\}$.

Thus, by Lemma \ref{cR}, we have
\begin{equation}
Z_1(\gI) =c_{\cR_1\times \cR_2} (X_0^4/m_0 I)+O(\eta_2).
\label{eq:Z1}
\end{equation}
Now we consider $Z_2(\gI)$. By Lemma \ref{Perron} we find that
\begin{equation}
Z_2(\gI)=\frac{h (\gI)g((m'))\Sft}{\gamma_K m'{}^3} +O(16^{\omega ((m)\gI)}\exp (-c\sqrt{\log R})) ,
\label{eq:Z2}
\end{equation}
where
\begin{align*}
 g((m))&:=\prod_{\gP |(m)}\Big ( 1-\frac{\varrho_\v (\gP)}{N(\gP)}\Big )^{-1}, \\
h(\gJ)&:=\prod_{\gP|\gJ}\Bigl(1-\frac{\rho_\v(\gP)}{N(\gP)}\Bigr)^{-1}\prod_{\gP_2^e||\gJ}\Bigl(\frac{\rho_\v(\gP^e)}{N(\gP^e)}-\frac{\rho_\v(\gP^{e+1})}{N(\gP^{e+1})}\Bigr).
\end{align*}
Putting together \eqref{eq:Z1} and \eqref{eq:Z2}, we see that
\begin{equation}
\sum_{\gI\in\cI}Z_1(\gI)Z_2(\gI)=\frac{g((m'))\Sft c_{\cR_1\times \cR_2}(X_0^4/m_0I)}{\gamma_K m'{}^3}\sum_{\gJ\in\cI}h(\gJ)+O(\eta_2^2 I).
\label{eq:Z1Z2Sum}
\end{equation}
Since $h(\gI)$ is multiplicative, the sum can be calculated by a contour computation
\begin{align}
\sum_{\gJ\in\cI}h(\gJ)&=\frac{1}{2\pi i}\int_{2-i\infty}^{2+i\infty} \frac{I^s ((1+\eta_2)^s-1)}{s}\sum_{\gI} \frac{h(\gI)}{N(\gI)^s}ds\nonumber\\
&=\text{Res}_{s=0}\Bigl(\frac{I^s ((1+\eta_2)^s-1)}{s}\sum_{\gI} \frac{h(\gI)}{N(\gI)^s}\Bigr)+O(\exp(-c\sqrt{\log{R}})).\label{eq:hContour}
\end{align}
We see that the residue is given by
\begin{align}
&\gamma_K\log(1+\eta_2)\prod_{\gP}\Bigl(1+h(\gP)+h(\gP^2)+...\Bigr)\Bigl(1-\frac{1}{N(\gP)}\Bigr)\nonumber\\
&=\gamma_K\log(1+\eta_2)\prod_{\gP}\Biggl(1+\Bigl(1-\frac{\rho(\gP)}{N(\gP)}\Bigr)^{-1}\Bigl(\sum_{e\ge 1}\Bigl(\frac{\rho(\gP^e)}{N(\gP^e)}-\frac{\rho(\gP^{e+1})}{N(\gP^{e+1})}\Bigr)\Bigr)\Biggr)\Bigl(1-\frac{1}{N(\gP)}\Bigr)\nonumber\\
&=\gamma_K\log(1+\eta_2)\prod_{\gP}\Bigl(1-\frac{\rho(\gP)}{N(\gP)}\Bigr)^{-1}\Bigl(1-\frac{1}{N(\gP)}\Bigr)\prod_{\gP}\Biggl(\Bigl(1-\frac{\rho(\gP)}{N(\gP)}\Bigr)+\frac{\rho(\gP)}{N(\gP)}\Biggr)\nonumber\\
&=\gamma_K\frac{\log(1+\eta_2)}{\Sft}.\label{eq:hResidue}
\end{align}
Putting together \eqref{eq:Z1Z2Sum} \eqref{eq:hContour} and \eqref{eq:hResidue} we see that
\begin{equation}
\sum_{\gI\in\cI}Z_1(\gI)Z_2(\gI)=\frac{g((m'))(\log{1+\eta_2})c_{\cR_1\times \cR_2}(X_0^4/m_0I)}{m'{}^3}+O(\eta_2^2 I).
\label{eq:Z1Z2Sum2}
\end{equation}
Finally, we recall the definition \eqref{eq:NpDef} of $n_p$. Since $f$ is the product of two linear factors when $p\equiv 1\Mod{D_f}$, we have $n_p=2+O(1/p)$ for all $p\in[P_1,P_2]$. Thus
\begin{equation}
\sum_{\substack{p\in [P_1,P_2]\\p\equiv 1\Mod{D_f}}}\frac{n_p}{p}=\frac{(2+O(\eta_2))\log(P_2/P_1)}{D_f}.
\label{eq:fBound}
\end{equation}
Putting together \eqref{eq:M1}, \eqref{eq:M2}, \eqref{eq:Z1Z2Sum2} and \eqref{eq:fBound} now gives the result.
\end{proof}
%
%
We are now in a position to establish Proposition \ref{prpstn:TSieve}.
%
%
\begin{proof}[Proof of Proposition \ref{prpstn:TSieve}]
We see that putting together Lemma \ref{lmm:TSplit} and \ref{lmm:M1} gives
\[
T_{sieve}=(2+O(\eta_2^{1/3}))\eta_2|\cAI(X_0)| c_{\cR_1\times \cR_2}(X_0^4/mI)\frac{g(m')}{m'{}^3}\frac{\log(P_2/P_1)}{\varphi(D_f)}.
\]
provided $B^4>X^{1+\epsilon}RP_2^{12}$. Recalling that $R=X^{\epsilon_{00}}$, $P_2=X^{\tau'}$, $B^4=X^{\sum_{i=1}^{\ell'}t_i}\ge X^{\sum_{i=1}^{\ell'}\theta_i}$ we see that this condition is satisfied provided
\[
\sum_{i=1}^{\ell'}\theta_i>1+\epsilon_{00}+12\tau'
\]
and $\epsilon$ is taken sufficently small. This gives the result.
\end{proof}
%
%

\section{Proposition \ref{prpstn:T1}: The term $T_1(\cR)$}\label{sec:T1}

%
%
In this section we use the dispersion method to bound $T_1(\cR)$ and establish Proposition \ref{prpstn:T1}. Let us recall the expression of $T_1(\cR)$
\begin{equation*}
 T_1(\cR)=\sum_{\substack{{p\in [P_1,P_2]}\\ {p\equiv 1\Mod {D_f}}}}\sum_{\gI\in\cI}\sum_{\gM_0\gI\ga\gb\in\cAI(\v_0,m',p)}\1_{\cR_1}(\ga)(\1_{\cR_2}(\gb)-\tilde\1_{\cR_2}(\gb)).
\end{equation*}
To simplify some notation we will write 
\begin{equation}
\tilde g(\gb ):=\1_{\cR_2}(\gb)-\tilde\1_{\cR_2}(\gb).\label{eq:gTildeDef}
\end{equation}
We first split the sum over $\gb$ into ideal classes $\cC\in Cl_K$. Let $\gc\in\cC$ and $\gc'=(N(\gc)/\gc)$. Since the ideals  in the set $\cA$ are principal, the ideals $\gM_0\gI\ga\c$ and $\gb\gc'$ are principal. Therefore they are respectively of the form
$(\alpha), (\beta)$ with $\gM_0\gI\gc | (\alpha )$, $\gc'|(\beta)$ with $V\alpha = a_1\nu_1+a_2\nu_2+a_3\nu_3+a_4\nu_4$, $V\beta =b_1\nu_1+b_2\nu_2+b_3\nu_3+b_3\nu_4$,
where $a_1,a_2,a_3,b_1,b_2,b_3,b_4\in\Z$ and with $\a,\b$ lying in the fundamental domain $\cD$. We will write $\a =(a_1,a_2,a_3,a_4)$, $\b=(b_1,b_2,b_3,b_4)$.
In order to handle the modulo $m$ condition  between $\gb$ and $\gI\ga$ we plit the sums according to some congruence classes on $\alpha$, $\beta$ modulo $m'$. Together this gives
\begin{equation}
T_1(\cR)=\sum_{\cC\in Cl_K}\sum_{\substack{{\a_0,\b_0 \Mod{m'}}\\ {N(\gc )(\a_0\diamond \b_0)_i\equiv (\v_0)_i\Mod {m'},\ {\rm for}\ i=1,2,3,4}}}\tilde T_\gc(\cR,\a_0,\b_0),\label{eq:T1}
\end{equation}
with $(a_0)_4=0$ since $(\a\diamond\b )_4=0$ and 
 $\gc\in\cC$ is a well chosen representant, and $\gc'$ as above
\begin{equation*}
\tilde T_\gc(\cR,\a_0,\b_0)=\sum_{\substack{{p\in [P_1,P_2]}\\ {p\equiv 1\Mod{D_f}}}}\sum_{\gI\in\cI}\sum_{\substack{{\a\equiv\a_0\Mod{m'}}\\ {\b\equiv \b_0\Mod {m'}}\\ {\gM_0\gI\gc|(\alpha ), \gc'| (\beta)}\\ {(\alpha\beta)/(N(\gc ))\in\cAI ( p)}}}
\1_{\cR_1}(\frac{(\alpha )}{\gM_0\gI\gc})\tilde g ((\beta)/\gc)), 
\end{equation*}
with now $\cAI (p)=\cAI (\cX,\zero, 1, p)$.

We recall that our previous conditions imply that there exists $A$, $B$ such that 
$N( \ga)\in [A^4,2A^4]$, $N(\gb) \in [B^4,2B^4]$.

We will use the notation of
\cite[p. 80 and 71]{May15b}:
\begin{align}
\cR_{X_0}&:=\Bigl\{ \x\in\R^4 : x_i\in [X_i, X_i (1+\eta_1)], i=1,2,3, \ x_4=0,\nonumber\\
&\qquad\qquad N(\sum_{i=1}^3x_i\nu_i)\in [X_0^4,X_0^4(1+\eta_2)]\Bigr\},\label{defRb1b2}\\
\cR_{\b_1 ,\b_2} &:=\Bigl\{ \a\in\R^4 : \| a\| \in [A,2A], \a\diamond\b_1\in\cR_{X_0},\ \a\diamond\b_2\in\cR_{X_0}\Bigr\}.\nonumber
\end{align}

Let $\cF$ be a fundamental domain such that if $\1_{\cR_1} (\ga)=1$ and $\ga\gI=V^{-1}\sum_{i=1}^4 a_i\nu_i$  then if $\alpha\in  \cF$, $a_i\ll A$ and similarly, $b_i\ll B$ for all $1\le i\le 4$.

We will concentrate on ideals $\gb$ with not too many divisors. For this we introduce a slight variant of $\tilde g$
\begin{equation*}
g_\b=\begin{cases}
\1_{\cR_2}(\gb)-\tilde\1_{\cR_2}(\gb) &\text{if}\ \tau (\gb )\le \varepsilon_0^{-2},\\
0 & \text{otherwise}.
\end{cases}
\end{equation*}

Following \cite[section 11]{May15b} except that we apply Lemma \ref{div} , we prove that we can replace $\tilde g_\b$ by $g_\b$ with a  error term less 
$O(\varepsilon_0^2X_0^3 (\log X)^{O(1)})$.
This modification will permit us to bound the terms $g_\b$ by $O(\varepsilon^{-2})$.

Thus now we have to concentrate on sums
\begin{equation}
 T_\gc(\cR,\a_0,\b_0)=\sum_{\substack{{p\in [P_1,P_2]}\\ {p\equiv 1\Mod{D_f}}}}\sum_{\gI\in\cI}\sum_{\substack{{\a\equiv\a_0\Mod{m'}}\\ {\b\equiv \b_0\Mod {m'}}\\ {\gM_0\gI\gc|(\alpha ), \gc'| (\beta)}\\ {\gI(\alpha\beta)/(N(\gc ))\in\cAI(p)}\\
 \a\diamond\b\in\cR_{X_0}}}\1_\cF (\ga)
\1_{\cR_1}(\frac{(\alpha )}{\gM_0\gI\gc}) g_{(\beta)/\gc}.
\label{eq:Tc}
\end{equation}
%
%
\subsection{Cosmetic reductions}\label{sec:Cosmetic}

%
%
For $T>0$, we denote by
 $\cC_T$  the subset of $\R^4$ defined by 
$$\cC_T=\{ \a\in\R^4 : N(\ga)\in [T^4,T^4(1+\eta_\cX^2)]\}.$$
By Weber's Theorem \cite{Web99}, we have $$|\cC_T|= \lambda_K T^4\eta_\cX^2+ O(T^3),$$
for some $\lambda_K$ depending only on $K$.
It will make some later technicalities simpler if we introduce the restriction $p\nmid N(\bb)$ to the terms in $T_1$. By Proposition \ref{prpstn:TypeI} and the divisor bound, we can do this at the cost of an error term of size
\[
\ll\sum_{p\in [P_1,P_2]}
\sum_{\substack{{\b\in \Z^4\cap C_B}\\ {N(\gb)\equiv 0\Mod p}}}\sum_{\substack{{\gb\gu\in \cAI (0,1,p)}}}X^\varepsilon\ll X^\varepsilon(X^{3}P_1^{-1}+X^2B^{4/3}P_2^2+B^4P_2^5).
\]
This is acceptably small provided
\begin{equation}
B<\frac{X^{3/4-\epsilon}}{P_2^{3/2}}.\label{eq:BSize}
\end{equation}
When $p\equiv 1\Mod{D_f}$, the function $f\Mod{p}$ factors as the product of two linear factors. Thus the condition $p| f(\aa\diamond \bb)$ is equivalent to $p|\vv_p\cdot(\aa\diamond\bb)$ or $p|\ww_p\cdot(\aa\diamond\bb)$ for two non-zero vectors $\vv_p,\ww_p\in \ZZ^4$. There are $O(p^5)$ choices of $\aa_p,\bb_p\Mod{p}$ such that $(\aa_p\diamond\bb_p)=\v_p\cdot(\aa_p\diamond\bb_p)=w_p\cdot(\aa_p\diamond\bb_p)$ whenever $p$ is sufficiently large in terms of $f$. Therefore, as above, provided \eqref{eq:BSize} holds, the contribution of the $\aa$, $\bb$ such that 
 $p|\vv_p\cdot(\aa\diamond\bb)$ and  $p|\ww_p\cdot(\aa\diamond\bb)$ is bounded by
\[ X^\varepsilon\sum_{p\in [P_1,P_2]}\sum_{\substack{\a_p,\b_p\in \{ 1,\ldots ,p\}^4\\ {(\a\diamond\b )_4\equiv 0\Mod p}\\ {p|\v_p\cdot (\a\diamond\b)}\\ {p|\w_p\cdot (\a\diamond\b)}}}
\sum_{\substack{{\b\in\Z^4\cap C_B}\\ {\b\equiv\b_p\Mod p}}}\sum_{\substack{{\gu\in\cAI (0,1,p)}\\ {\gb |\gu }\\ {\u\equiv (\a_p\diamond\b_p)\Mod p}}}1
\ll X^{3+\varepsilon}P_1^{-1}.
\]
Putting this together, we see that it suffices for us to estimate for each $C\in Cl_K$ and each $\aa_0,\bb_0 \Mod{m'}$ the sums
\begin{equation}
T_3(\mathcal{R}):=\sum_{\substack{p\in [P_1,P_2]\\ p\equiv 1\Mod{D_f}}}\sum_{\gI\in\cI}\sum_{\substack{\bb\in\ZZ^4\cap \Cc_B\\ \bb\equiv \bb_0\Mod{m'} \\ p\nmid N(\bb)}}\sum_{\substack{\aa\in\ZZ^4\cap \Cc_A\\ (\aa\diamond \bb)\in\cR_{X_0}\\ p|\vv_p\cdot (\aa\diamond\bb) \\ \aa\equiv \aa_0\Mod{m'} \\
{\gI\gc |\ga }\\ \ga\in\cAI (p)}}\1_{\Rc_1}(\frac{\ga}{\gM_0\gI\gc})g_{(\beta)/\gc}.
\label{eq:T3Def}
\end{equation}
%
%
\subsection{Dispersion method}
%
%
We swap the order of summation, and apply Cauchy-Schwarz. The ideals $\gI$ and $\ga/\gI$ are coprime since $N(\gI)<X^{\theta_i}$ for all $1\le i\le\ell$. 
In the application of Cauchy-Schwarz we can group these ideals together. We recall that the set $\cR_X$
is defined in \eqref{setN}.
This gives 
\[
T_3^2\le \eta^4 A^4  \sum_{\substack{\aa\in\mathbb{Z}^4\cap \Cc_A\\ \aa\equiv \aa_0\Mod{m'} }}\Bigl(\sum_{\substack{p\in [P_1,P_2]\\ p\equiv 1\Mod{D_f} }}\sum_{\substack{\bb\in \ZZ^4\cap\Cc_B\\
\aa\diamond\bb\in\cR_{X_0}\\ p|\vv_p\cdot (\aa\diamond \bb)\\ \bb\equiv \bb_0\Mod{m'}\\ p\nmid N(\bb)}}g_\bb\Bigr)^2.
\]
Thus we see that
\begin{align}
T_3^2\ll \eta_\cX^2 A^4 T_4\label{eq:T3}
\end{align}
where, with the notation \eqref{defRb1b2}
\[
T_4:=\sum_{\substack{p_1,p_2\in [P_1,P_2]\\ p_1\equiv p_2\equiv 1\Mod{D_f} }}\sum_{\substack{\bb_1,\bb_2\in \ZZ^4\cap\Cc_B\\ \bb_1\equiv \bb_2\equiv \bb_0\Mod{m'}\\ p_1\nmid N(\bb_1),\, p_2\nmid N(\bb_2)}}g_{\bb_1}g_{\bb_2}\sum_{\substack{\aa\in \ZZ^4\cap \Cc_A\\ \aa\in\cR_{\b_1,\b_2}\\
 p_1|\vv_{p_1}\cdot (\aa \diamond \bb_1)\\ p_2|\vv_{p_2}\cdot (\aa\diamond \bb_2)}}1.
\]
Thus we wish to show that $T_4=o(\eta_\cX^6 A^2 B^6)$.
%
%
\subsection{Collinear $\bb_1,\bb_2$}
We separate the situation when $\bb_1$ and $\bb_2$ are collinear (in which case we have $\wedge(\bb_1,\bb_2)=0$ where $\wedge(\xx,\yy)$ is the $L^2$ norm of the six $2\times 2$ subdeterminants of the $2\times 4$ matrix with columns $\xx$ and $\yy$. Thus we have
\begin{equation}
T_4=T_5+T_6,\label{eq:T4}
\end{equation}
where $T_5$ is those terms with $\wedge(\bb_1,\bb_2)=0$ and $T_6$ is those terms with $\wedge(\bb_1,\bb_2)\ne 0$.

We first concentrate on $T_5$. 
%
%
\begin{lmm}\label{lmm:T5}
\[
T_5\ll A^3 B^3 (\log{X})^{O(1)}.
\]
\end{lmm}
%
%
\begin{proof}
Let $\cc$ be the shortest non-zero vector with integer components which is collinear with $\bb_1$ (this is $\bb_1$ divided by the $\gcd$ of its components). Then we see that $\bb_1=\lambda_1\cc$ for some $\lambda\in \ZZ$, and since $\bb_2$ is collinear with $\bb_1$, we also have that $\bb_2=\lambda_2\cc$ for some $\lambda_2\in \mathbb{Z}$. Thus we see that
\[
T_5\ll \sum_{\substack{\cc\in \ZZ^4\\ \|\cc\|\ll B}}\sum_{\lambda_1,\lambda_2\ll B/\|\cc\|}\sum_{\substack{\aa\in \ZZ^4\cap\Cc_A\\ (\aa\diamond \cc)_4=0}}\sum_{\substack{p_1,p_2\in [P_1,P_2]\\ p_1|f(\lambda_1\aa\diamond \cc)\\ p_2|f(\lambda_2\aa\diamond\cc)}}1.
\]
We see that the inner sum is $O(1)$ since $P_1\gg B^\epsilon$ and $f(\lambda_1\aa\diamond\cc)\ll B^{O(1)}$. We then split the size of $\|\cc\|$ into dyadic ranges, giving 
\[
T_5\ll (\log{X})\sup_{C\ll B}\frac{B^2}{C^2}\sum_{\substack{\cc\in \ZZ^4\\ \|\cc\|\asymp C}}\sum_{\substack{\aa\in \ZZ^4\cap\Cc_A\\ (\aa\diamond \cc)_4=0}}1.
\]
We now let $\zz=(\aa\diamond\cc)$. By the divisor bound, given $\zz$ there are $O(\tau_K(\mathfrak{z}))$ choices of $\aa,\cc$. Thus we see that
\[
T_5\ll (\log{X})\sup_{C\ll B} \frac{B^2}{C^2}\sum_{z_1,z_2,z_3\ll AC}\tau_K(z_1\nu_1+z_2\nu_2+z_3\nu_3)\ll A^3B^3(\log{X})^{O(1)}.
\]
\end{proof}
%
%
Thus we are left to bound $T_6$.
%
%
\subsection{Lattice counts}
%
%
We now concentrate on the inner sum. Let $\Lambda_{\bb_1,\bb_2}$ and $\Lambda_{\bb_1,\bb_2,p_1,p_2}$ denote the lattices
\begin{align*}
\Lambda_{\bb_1,\bb_2}&:=\{\xx\in\ZZ^4:\,(\xx\diamond \bb_1)_4=(\xx\diamond \bb_2)_4=0\},\\
\Lambda_{\bb_1,\bb_2,p_1,p_2}&:=\{\xx\in\Lambda_{\bb_1,\bb_2}:\,p_1|\vv_{p_1}\cdot (\xx\diamond \bb_1),\,p_2|\vv_{p_2}\cdot (\xx \diamond \bb_2)\}.
\end{align*}
Thus the inner sum in $T_6$ is 
\[
\sum_{\substack{\aa \in \Cc_A\cap\Lambda_{\bb_1,\bb_2,p_1,p_2}\\
\a\in\cR_{\b_1 ,\b_2}\\
}}1.
\]
If $\bb_1,\bb_2$ are not collinear, then $\Lambda_{\bb_1,\bb_2}$ is a lattice of rank 2, and so it has a Minkowski-reduced basis $\{\zz_1,\zz_2\}$. Without loss of generality we may assume that $\|\zz_1\|\le \|\zz_2\|$. Thus we have that
\[
\sum_{\aa \in \Cc_A\cap\Lambda_{\bb_1,\bb_2,p_1,p_2}\cap\cR_{\b_1 ,\b_2}}1=\sum_{\substack{\lambda_1,\lambda_2\in\mathbb{Z}\\ \lambda_1\zz_1+\lambda_2\zz_2\in\Cc_A\cap\cR_{\b_1 ,\b_2}\\
\lambda_1 c_1+\lambda_2 c_2\equiv 0\Mod{p_1}\\ \lambda_1c_3+\lambda_2c_4\equiv 0\Mod{p_2} }}1,
\]
for some constants $c_1,c_2,c_3,c_4$ depending only on $\bb_1,\bb_2$, $p_1$ and $p_2$. The condition $\lambda_1\zz_1+\lambda_2\zz_2\in\Cc_A\cap\cR_{\b_1,\b_2}$ forces $\lambda_1,\lambda_2$ to lie in a region $\Rc'_{\bb_1,\bb_2}\subseteq \RR^4$. Since $\|\lambda_1\zz_1+\lambda_2\zz_2\|\asymp |\lambda_1|\|\zz_1\|+|\lambda_2|\|\zz_2\|$ and $\Cc_A$ only contains vectors of norm $O(A)$, we see that lying in $\Rc_{\bb_1,\bb_2}$ forces $\lambda_1\ll A/\|\zz_1\|$ and $\lambda_2\ll A/\|\zz_2\|$, so $\Rc'_{\bb_1,\bb_2}$ has volume $O(A^2/\|\zz_1\|\|\zz_2\|)$.

By Davenport's Theorem on counting lattice points (\cite[Lemma 7.1]{May15b} for example), we have that
\[
\sum_{\substack{(\lambda_1,\lambda_2)\in \Rc'_{\bb_1,\bb_2} \\ \lambda_1c_1+\lambda_2c_2=0\Mod{[p_1,p_2]} }}1=\frac{\vol(\Rc'_{\bb_1,\bb_2})}{f_{\bb_1,\bb_2,p_1,p_2}}+O\Bigl(\frac{A}{\|\zz_1\|}\Bigr)
\]
where $f_{\bb_1,\bb_2,p_1,p_2}=[\Lambda_{\bb_1,\bb_2}:\Lambda_{\bb_1,\bb_2,p_1,p_2}]$ is the index of the lattices, given explicitly in terms of $c_1,c_2,c_3,c_4,p_1,p_2$ by
\[
f_{\bb_1,\bb_2,p_1,p_2}=
\begin{cases}
1,\qquad &c_1\equiv c_2\equiv 0\Mod{p_1}\text{ and }c_3\equiv c_4\equiv 0\Mod{p_2},\\
p_2,&c_1\equiv c_2\equiv 0\Mod{p_1}\text{ and }c_3,c_4\text{ not both }0\Mod{p_2},\\
p_1,&c_3\equiv c_4\equiv 0\Mod{p_2}\text{ and }c_1,c_2\text{ not both }0\Mod{p_1},\\
p_1, &p_1= p_2\text{ and }c_1c_4\equiv c_2c_3\Mod{p_1}\text{ and }c_1,c_2\text{ not all }0\Mod{p_1},\\
p_1p_2,&\text{otherwise}.
\end{cases}
\]
We split $T_6$ into the contribution from the main term $\vol(\Rc'_{\bb_1,\bb_2})/f_{\bb_1,\bb_2,p_1,p_2}$ and the error term $O(A/\|\zz_1\|)$. This gives
\begin{equation}
T_6=T_8+O(T_7),\label{eq:T6}
\end{equation}
where
\begin{align*}
T_7&:=\sum_{\substack{p_1,p_2\in [P_1,P_2]\\ p_1\equiv p_2\equiv 1\Mod{D_f} }}\sum_{\substack{\bb_1,\bb_2\in \ZZ^4\cap\Cc_B\\ \bb_1\equiv \bb_2\equiv \bb_0\Mod{m'}\\ p_1\nmid N(\bb_1),\,p_2\nmid N(\bb_2)\\\wedge(\bb_1,\bb_2)\ne 0}}\frac{A}{\|\zz_1\|},\\
T_8&:=\sum_{\substack{p_1,p_2\in [P_1,P_2]\\ p_1\equiv p_2\equiv 1\Mod{D_f} }}\sum_{\substack{\bb_1,\bb_2\in \ZZ^4\cap\Cc_B\\ \bb_1\equiv \bb_2\equiv \bb_0\Mod{m'}\\ p_1\nmid N(\bb_1),\,p_2\nmid N(\bb_2)\\\wedge(\bb_1,\bb_2)\ne 0}}\frac{g_{\bb_1}g_{\bb_2}\vol(\Rc'_{\bb_1,\bb_2})}{f_{\bb_1,\bb_2,p_1,p_2}}.
\end{align*}
We first show that the contribution $T_7$ from the error term is small. 
%
%
\begin{lmm}\label{lmm:T7}
\[
T_7\ll AB^7 P_2^2.\]
\end{lmm}
%
%
\begin{proof}
We note that $\zz_1\in \ZZ^4$ with $\|\zz_1\|^2\le \|\zz_1\|\cdot\|\zz_2\|\ll \det(\Lambda_{\bb_1,\bb_2})\ll B^2$. Thus $\|\zz_1\|\ll B$. Thus we can rearrange the summation to give
\[
T_7 \ll P_2^2\sum_{\bb_1,\bb_2\in \ZZ^4\cap \Cc_B}\frac{A}{\|\zz_1\|}\ll A P_2^2\sum_{\substack{\zz_1\in\ZZ^4\\ \|\zz_1\|\ll B}}\frac{1}{\|\zz_1\|}\Bigl(\sum_{\substack{\bb\in \ZZ^4\cap\Cc_{B}\\ (\bb\diamond \zz_1)_4=0}}1\Bigr)^2
\]
The condition $(\bb\diamond \zz_1)_4=0$ forces $\bb$ to lie in a rank 3 lattice of determinant $\|\zz_1\|$. Thus the inner sum is $O(B^3/\|z_1\|+B^2)$. Thus we obtain the bound
\[
T_7 \ll A P_2^2\sum_{\substack{\zz_1\in\ZZ^4\\ \|\zz_1\|\ll B}}\Bigl(\frac{B^6}{\|\zz_1\|^3}+\frac{B^4}{\|\zz_1\|}\Bigr)\ll AB^7 P_2^2.
\]
This gives the result.
\end{proof}
%
%
Thus we are left to show that
\[
T_8=o(\eta_\cX^4 A^2 B^6).
\]
%
%
\subsection{Further lattice estimates}
%
%
We recall that $\Rc'_{\bb_1,\bb_2}$ is the region $\lambda_1,\lambda_2\in \RR^2$ such that $\lambda_1\zz_1+\lambda_2\zz_2\in \Cc_B\cap\cR_{\b_1 ,\b_2}$. We see that this has volume $\vol(\Rc''_{\bb_1,\bb_2})/\det(\Lambda_{\bb_1,\bb_2})$, where $\det(\Lambda_{\bb_1,\bb_2})$ is the determinant of the lattice (that is, the 2-dimensional area of parallelogram generated by $\zz_1,\zz_2$) and $\Rc''_{\bb_1,\bb_2}$ is the 2-dimensional region formed by intersecting $\Cc_B$ with the $\zz_1,\zz_2$ plane 
\[T_8=
 \sum_{\substack{p_1,p_2\in [P_1,P_2]\\ p_1\equiv p_2\equiv 1\Mod{D_f}}}\sum_{\substack{\bb_1,\bb_2\in \ZZ^4\cap\Cc_B\\  \bb_1\equiv \bb_2\equiv \bb_0\Mod{m'} \\ \wedge(\bb_1,\bb_2)\ne 0\\ p_1\nmid N(\bb_1),\,p_2\nmid N(\bb_2)}}\frac{g_{\bb_1}g_{\bb_2}\vol(\Rc''_{\bb_1,\bb_2})}{f_{\bb_1,\bb_2,p_1,p_2}\det(\Lambda_{\bb_1,\bb_2})}.
\]
We first establish a few simple estimates.
%
%
\begin{lmm}
\[
\sum_{\substack{p_1,p_2\in[P_1,P_2]\\ p_1\equiv p_2\equiv 1\Mod{D_f} }}\frac{\vol(\Rc_{\bb_1,\bb_2})}{f_{\bb_1,\bb_2,p_1,p_2}}\ll \vol(\Rc'_{\bb_1,\bb_2})+O\Bigl(\frac{AP_2^2}{\|\zz_1\|}\Bigr).
\]
\end{lmm}
%
%
\begin{proof}
%
%
We have that
\begin{align*}
\sum_{\substack{p_1,p_2\in[P_1,P_2]\\ p_1\equiv p_2\equiv 1\Mod{D_f}}}\sum_{\substack{\aa\in \ZZ^4\cap \Cc_A\\ \aa\in\cR_{\b_1 ,\b_2}\\ p_1|\vv_{p_1}\cdot \aa\diamond \bb_1\\ p_2|\vv_{p_2}\cdot \aa \diamond\bb_2}}1&=\sum_{\substack{\aa\in \ZZ^4\cap \Cc_A\\ \aa\in\cR_{\b_1 ,\b_2}}}\sum_{\substack{p_1,p_2\in[P_1,P_2]\\ p_1\equiv p_2\equiv 1\Mod{D_f}\\ p_1|\vv_{p_1}\cdot \aa\diamond \bb_1\\ p_2|\vv_{p_2}\cdot \aa \diamond\bb_2}}1\\
&\ll \sum_{\substack{\aa\in \ZZ^4\cap \Cc_A\\ (\aa\in\cR'_{\b_1 ,\b_2}=0}}1\\
&\ll \vol(\Rc'_{\bb_1,\bb_2})+O\Bigl(\frac{A}{\|\zz_1\|}\Bigr).
\end{align*}
On the other hand, we know that
\[
\sum_{\substack{p_1,p_2\in[P_1,P_2]\\ p_1\equiv p_2\equiv 1\Mod{D_f}}}\sum_{\substack{\aa\in \ZZ^4\cap \Cc_A\\ \aa\in\cR_{\b_1 ,\b_2}\\ p_1|\vv_{p_1}\cdot \aa\diamond \bb_1\\ p_2|\vv_{p_2}\cdot \aa \diamond\bb_2}}1=\sum_{\substack{p_1,p_2\in[P_1,P_2]\\ p_1\equiv p_2\equiv 1\Mod{D_f}}}\Bigl(\frac{\vol(\Rc'_{\bb_1,\bb_2})}{f_{\bb_1,\bb_2,p_1,p_2}}+O\Bigl(\frac{A}{\|\zz_1\|}\Bigr)\Bigr).
\]
Putting these together gives the result.
\end{proof}
%
%
\begin{lmm}\label{lmm:Cylinder}
Let
\[
\Cc_{C,d;\cc_1,\bb_2}:=\#\Bigl\{\bb_1\in\ZZ^4\cap\Cc_B:\,\wedge(\bb_1,\bb_2)\sim \frac{B^2}{C},\,\bb_1\equiv \cc_1\Mod{d}\Bigr\}.
\]
Then we have
\[
\#\Cc_{C,d;\cc_1,\cc_2}\ll \Bigl(1+\frac{B}{d}\Bigr)\Bigl(1+\frac{B}{C d}\Bigr)^3.
\]
\end{lmm}
%
%
\begin{proof}
The condition $\wedge(\bb_1,\bb_2)\sim B^2/C$ forces $\bb_1$ to lie in a cylinder $\Cc$ with axis of length $O(B)$ proportional to $\bb_2$, and with radius $O(B/C)$. We then see that we can cover this cylinder with 
\[
\ll \Bigl(1+\frac{B}{d}\Bigr) \Bigl(1+\frac{B}{Cd}\Bigr)^3
\]
 different hypercubes $\Bc$ of side length $d$. Finally, there is at most one choice of $\bb_1$ in a hypercube $\Bc$ of side length $d$ which satisfies $\bb_1\equiv \cc_1\Mod{d}$, which gives the result.
\end{proof}

For any $\cc_1,\cc_2\in\Z^4$, the notation $\cc_1\propto\cc_2$ indicates that the two vectors are proportional. 
%
%
\begin{lmm}\label{lmm:BasicSum}
If $\cc_1\not\propto\cc_2\Mod{p}$ then
\[
\sum_{\substack{\bb_1,\,\bb_2\in \ZZ^4\cap\Cc_B \\ \wedge(\bb_1,\bb_2)\ne 0\\ \mathrm{primitive}\\ \bb_1\equiv \cc_1\Mod{p}\\ \bb_2\equiv \cc_2\Mod{p}}}\frac{1}{\det(\Lambda_{\bb_1,\bb_2})}\ll \frac{B^6}{p^8}+B^{17/3}.
\]
\end{lmm}
%
%
\begin{proof}
We recall that $\Lambda_{\bb_1,\bb_2}$ is the lattice in $\ZZ^4$ of $\xx$ with $(\xx\diamond\bb_1)_4=(\xx\diamond\bb_2)_4=0$. By \cite[Lemma 10.1]{May15b}, this has determinant $\wedge(\bb_1,\bb_2)/D_{\bb_1,\bb_2}$, where $\wedge(\bb_1,\bb_2)$ is the $L^2$ norm of the six $2\times 2$ subdeterminants of the matrix with columns $\bb_1,\bb_2$, and $D_{\bb_1,\bb_2}$ is the greatest common divisor of these six subdeterminants. Note that this implies $\bb_1\propto\bb_2\Mod{D_{\bb_1,\bb_2}}$, so since $\bb_1,\bb_2$ are primitive we must have $D_{\bb_1,\bb_2}\le B$ when $\wedge(\bb_1,\bb_2)\ne 0$.

We consider separately those $\bb_1,\bb_2$ with $\wedge(\bb_1,\bb_2)\ll B$, those with $B\ll \wedge(\bb_1,\bb_2)\ll B^{4/3}$, and those $\bb_1,\bb_2$ with
\[
D_{\bb_1,\bb_2}=d,\qquad \wedge(\bb_1,\bb_2)\sim B^2/C\]
for each $1\le d\le B$ and $1\le C\le B^{2/3}$ with $C$ running through powers of $2$.

If $\wedge(\bb_1,\bb_2)\ll B$ then $\bb_1$ lies within $O(1)$ of the line proportional to $\bb_2$, and so there are $O(B)$ choices of $\bb_1$. Since $\det(\Lambda_{\bb_1,\bb_2})\ge 1$, these terms contribute a total (ignoring the congruence conditions $\Mod{p}$ for an upper bound)
\[
\ll \sum_{\|\bb_2\|\ll B}O(B)\ll B^5.
\] 
If $\wedge(\bb_1,\bb_2)\in[B, B^{4/3}]$ then we separately consider those with $\wedge(\bb_1,\bb_2)\sim B^2/C$ for $C\in[B^{2/3},B]$ running through powers of 2, and again drop the congruence constraints. By Lemma \ref{lmm:Cylinder} there are
\[
\ll B\Bigl(1+\frac{B}{C}\Bigr)^3\ll \frac{B^4}{C^3}
\]
choices of $\bb_1$ given $\bb_2$. If $\wedge(\bb_1,\bb_2)\sim B^2/C$ then $\det(\Lambda_{\bb_1,\bb_2})\ge B/C$ (since $D_{\bb_1\bb_2}\le B$). Thus these terms contribute
\[
\ll \sum_{C=2^j\in[B^{2/3},B]}\sum_{\substack{\bb_2\in \ZZ^4\cap\Cc_B}}\frac{C}{B}\frac{B^4}{C^3 }\ll B^{17/3}.
\]

Thus we are left to consider the terms with $\wedge(\bb_1,\bb_2)\sim B^2/C$ for some $C\le B^{2/3}$. The condition $D_{\bb_1,\bb_2}=d$ forces $\bb_1\propto \bb_2\Mod{d}$, and so $\bb_1\equiv \lambda \bb_2\Mod{d}$ for some $\lambda\in \{1,\dots,d\}$. Since $\cc_1\not\propto\cc_2\Mod{p}$, we see $p\nmid d$. Thus $\bb_1\equiv \cc_0(\lambda)\Mod{dp}$, where $\cc_0(\lambda)\equiv \lambda\bb_2\Mod{d}$ and $\cc_0(\lambda)\equiv \cc_1\Mod{p}$. By Lemma \ref{lmm:Cylinder}, the number of choices of $\bb_1$ is therefore
\[
\ll \sum_{1\le \lambda\le d}\#\Cc_{C,pd,\cc_0(\lambda),\bb_2}\ll d\Bigl(1+\frac{B}{p d}\Bigr)\Bigl(1+\frac{B}{p C d}\Bigr)^3\ll B+\frac{B^4}{p^4 C^3 d^3}. 
\]
If $D_{\bb_1,\bb_2}=d$ and $\wedge(\bb_1,\bb_2)\sim B^2/C$ then $\det(\Lambda_{\bb_1,\bb_2})\gg B^2/(Cd)$. Thus we find that the contribution from terms with $\wedge(\bb_1,\bb_2)\ge B^{4/3}$ is
\begin{align*}
\sum_{\substack{\bb_1,\,\bb_2\in \ZZ^4\cap\Cc_B \\ \wedge(\bb_1,\bb_2)\ge B\\ \bb_1\equiv \cc_1\Mod{p}\\ \bb_2\equiv \cc_2\Mod{p}\\ \mathrm{primitive}}}\frac{1}{\det(\Lambda_{\bb_1,\bb_2})}&\ll \sum_{1\le d\le B}\sum_{C=2^j\ll B^{2/3}}\frac{d C}{B^2}\sum_{\substack{\bb_2\in \ZZ^4\cap\Cc_B\\ \bb_2\equiv \cc_2\Mod{p}}}\Bigl(B+\frac{B^4}{C^3 d^3 p^4}\Bigr)\\
&\ll \frac{B^6}{p^8}+B^{17/3}.
\end{align*}
Thus we have a suitable bound in each case, giving the result.
\end{proof}
%
%
\begin{lmm}\label{lmm:BasicSum2}
Let $\cc_1,\cc_2\in\ZZ^4$ be non-zero $\Mod{p}$ with $\cc_1\propto\cc_2\Mod{p}$. Then we have
\[
\sum_{\substack{\bb_1,\,\bb_2\in \ZZ^4\cap\Cc_B \\ \wedge(\bb_1,\bb_2)\ne 0\\ \mathrm{primitive}\\ \bb_1\equiv \cc_1\Mod{p}\\ \bb_2\equiv \cc_2\Mod{p}}}\frac{1}{\det(\Lambda_{\bb_1,\bb_2})}\ll \frac{B^6}{p^7}+B^{17/3}.
\]
\end{lmm}
%
%
\begin{proof}
This is similar to the proof of Lemma \ref{lmm:BasicSum}. Since the estimates in the proof of Lemma \ref{lmm:BasicSum} when $\wedge(\bb_1,\bb_2)\ll B^{4/3}$ didn't depend on whether $p|D_{\bb_1,\bb_2}$ or not, an identical argument shows that the contribution of $\bb_1,\bb_2$ with $\wedge(\bb_1,\bb_2)\ll B^{4/3}$ contributes $O(B^{17/3})$. Therefore we just need to consider the contribution when $\wedge(\bb_1,\bb_2)\gg B^{4/3}$. 

We split the summation according to $\wedge(\bb_1,\bb_2)\sim B^2/C$ and $D_{\bb_1,\bb_2}=d$. Since $\cc_1\propto\cc_2\Mod{p}$, we have $\cc_1\equiv \lambda_0\cc_2\Mod{p}$ for some $\lambda_0$.  Since $\bb_1\equiv \cc_1\Mod{p}$ and $\bb_2\equiv \cc_2\mod{p}$ we then see that $p|d$. The condition $D_{\bb_1,\bb_2}=d$ forces $\bb_1=\lambda\bb_2\Mod{d}$ for some $\lambda$, with $\lambda\equiv \lambda_0\Mod{p}$. Thus, by Lemma \ref{lmm:Cylinder}, the number of choices of $\bb_1,\bb_2$ with $\wedge(\bb_1,\bb_2)\sim B^2/C$ and $D_{\bb_1,\bb_2}=d$ s
\begin{align*}
\ll \sum_{\substack{\bb_2\in \ZZ^4\cap\Cc_B\\ \bb_2\equiv \cc_2\Mod{p}}}\sum_{\substack{1\le \lambda\le d\\ \lambda\equiv \lambda_0\Mod{p}}}\#\Cc_{C,d,\lambda\bb_2,\bb_2}&\ll \Bigl(1+\frac{B^4}{p^4}\Bigr)\frac{d}{p}\Bigl(1+\frac{B}{d}\Bigr)\Bigl(1+\frac{B}{Cd}\Bigr)^3\\
&\ll \frac{B^8}{p^5 C^3 d^3}+B^5.
\end{align*}
When $\wedge(\bb_1,\bb_2)\sim B^2/C$ and $D_{\bb_1,\bb_2}=d$ we have $\det(\Lambda_{\bb_1,\bb_2})\gg B^2/(Cd)$. Thus the total contribution from terms with $\wedge(\bb_1,\bb_2)\gg B^{4/3}$ is
\begin{align*}
\sum_{\substack{d\le B\\ p|d}}\sum_{C=2^j\ll B^{2/3}}\frac{C d}{B^2}\Bigl(\frac{B^8}{p^5 C^3 d^3}+B^5\Bigr)\ll \frac{B^6}{p^7}+B^{17/3}.
\end{align*}
This gives the result.
\end{proof}
%
%
We are now able to make progress on our aim of bounding $T_8$.
%
%
\begin{lmm}\label{lmm:T8}
Let $T_8$ be as given by \eqref{eq:T6}. Then we have
\[
T_8\ll \eta_2A^2B^6+\eta_2^{-4}A^2\sup_{\cC_1,\cC_2}\Bigl(|T_{11}|+|T_{12}|\Bigr),
\]
where the supremum is over all hypercubers $\cC_1,\cC_2\subseteq \cC_B$ of side length $\eta_2B$ and
\begin{align}
T_{11}&:=\sum_{\substack{p\in [P_1,P_2]\\ p\equiv 1\Mod{D_f}}}\sum_{\substack{\bb_1\in \ZZ^4\cap \Cc_1,\bb_2\in \ZZ^4\cap\Cc_2\\  \bb_1\equiv \bb_2\equiv \bb_0\Mod{m'} \\ \wedge(\bb_1,\bb_2)\ne 0\\ p\nmid N(\bb_1) N(\bb_2)}}\frac{g_{\bb_1}g_{\bb_2}}{f_{\bb_1,\bb_2,p,p}\det(\Lambda_{\bb_1,\bb_2})},\label{eq:T11Def}\\
T_{12}&:=\sum_{\substack{p_1,p_2\in [P_1,P_2]\\ p_1\equiv p_2\equiv 1\Mod{D_f}}}\frac{1}{p_1 p_2}\sum_{\substack{\bb_1\in \ZZ^4\cap \Cc_1,\bb_2\in \ZZ^4\cap\Cc_2\\  \bb_1\equiv \bb_2\equiv \bb_0\Mod{m'} \\ \wedge(\bb_1,\bb_2)\ne 0\\ p_1\nmid N(\bb_1),\,p_2\nmid N(\bb_2)}}\frac{g_{\bb_1}g_{\bb_2}D_{\bb_1,\bb_2}}{\wedge(\bb_1,\bb_2)}.\label{eq:T12Def}
\end{align}
\end{lmm}
%
%
\begin{proof}
We wish to replace $\Rc''_{\bb_1,\bb_2}$ with a quantity which doesn't depend on $\bb_1,\bb_2$ by splitting $\Cc_B$ into $O(\eta_2^{-4}\eta_1^4)$ smaller hypercubes of side length $\eta_2 B$. We see that $\vol(\Rc''_{\bb_1,\bb_2})$ depends continuously on the components of $\bb_1$ and $\bb_2$, and that $\vol(\Rc''_{\bb_1,\bb_2})$ is always of size $O(A^2)$. Moreover, if we restrict $\bb_1,\bb_2$ to hypercubes of side length $\eta_2 B$ then $\vol(\Rc'_{\bb_1,\bb_2})$ varies by $O(\eta_2 A^2)$ as $\bb_1,\bb_2$ vary within these hypercubes. Thus we see that
\begin{align}
T_8&=\sum_{\substack{p_1,p_2\in [P_1,P_2]\\ p_1\equiv p_2\equiv 1\Mod{D_f} }}\sum_{\substack{\bb_1,\bb_2\in \ZZ^4\cap\Cc_B\\ \bb_1\equiv \bb_2\equiv \bb_0\Mod{m'}\\ p_1\nmid N(\bb_1),\,p_2\nmid N(\bb_2)\\\wedge(\bb_1,\bb_2)\ne 0}}\frac{g_{\bb_1}g_{\bb_2}\vol(\Rc''_{\bb_1,\bb_2})}{f_{\bb_1,\bb_2,p_1,p_2}\det(\Lambda_{\bb_1,\bb_2})}\nonumber \\
&\ll T_9+\eta_2^{-4} A^2\sup_{\Cc_1,\Cc_2}|T_{10}|,
\end{align}
where
\begin{align*}
T_9&:=\eta_2\varepsilon_0^{-4}\sum_{\substack{p_1,p_2\in [P_1,P_2]\\ p_1\equiv p_2\equiv 1\Mod{D_f} }}\sum_{\substack{\bb_1,\bb_2\in \ZZ^4\cap\Cc_B\\ \bb_1\equiv \bb_2\equiv \bb_0\Mod{m'}\\ p_1\nmid N(\bb_1),\,p_2\nmid N(\bb_2)\\\wedge(\bb_1,\bb_2)\ne 0}}\frac{\vol(\Rc''_{\bb_1,\bb_2})}{f_{\bb_1,\bb_2,p_1,p_2}\det(\Lambda_{\bb_1,\bb_2})},\\
T_{10}&=T_{10}(\Cc_1,\Cc_2):= \sum_{\substack{p_1,p_2\in [P_1,P_2]\\ p_1\equiv p_2\equiv 1\Mod{D_f}}}\sum_{\substack{\bb_1\in \ZZ^4\cap \Cc_1,\bb_2\in \ZZ^4\cap\Cc_2\\  \bb_1\equiv \bb_2\equiv \bb_0\Mod{m'} \\ \wedge(\bb_1,\bb_2)\ne 0\\ p_1\nmid N(\bb_1),\,p_2\nmid N(\bb_2)}}\frac{g_{\bb_1}g_{\bb_2}}{f_{\bb_1,\bb_2,p_1,p_2}\det(\Lambda_{\bb_1,\bb_2})}.
\end{align*}
By the above lemmas, we have that
\[
T_9\ll \varepsilon^{-4}\eta_2 A^2 B^6,
\]
which is acceptable if $\eta_2\ll \eta_1^9$. Thus we are left to bound $T_{10}$. We separate the terms when the two primes in the outer sum are the same. Thus
\begin{equation}
T_{10}=T_{11}+T_{12},
\end{equation}
where $T_{11}$ denotes the terms with $p_1=p_2$ and $T_{12}$ those terms with $p_1\ne p_2$.

$T_{11}$ clearly is equal to the expression given in the lemma, but (recalling that $\det(\Lambda_{\bb_1,\bb_2})=\wedge(\bb_1,\bb_2)/D_{\bb_1,\bb_2}$) we need to show that $f_{\bb_1,\bb_2,p_1,p_2}=p_1p_2$ in $T_{12}$ to obtain the desired expression. We first note that since $p_1\nmid N(\bb_1)$ the multiplication-by-$\bb_1$  matrix $M_{\bb_1}$ is invertible $\Mod{p_1}$. This means that for every $\xx\Mod{p_1}$ there is a unique $\aa\Mod{p_1}$ such that $\xx=\aa\diamond \bb_1\Mod{p_1}$ , and so $\vv_{p_1}\cdot (\aa\diamond\bb_1)=0\Mod{p_1}$ is therefore a non-trivial constraint on the components of $\aa\Mod{p_1}$. Similarly since $p_2\nmid N(\bb_2)$, we see $p_2|\vv_{p_2}\cdot(\aa\diamond\bb_2)$ is a non-trivial constraints on the components of $\aa\Mod{p_2}$. From this it follows that we have that $f_{\bb_1,\bb_2,p_1,p_2}=p_1p_2$, and so $T_{12}$ is given by the expression in the lemma.
\end{proof}
%
%
  First we concentrate on $T_{11}$.
  %
%
\subsection{The case $p_1=p_2$}
%
%
In this section we wish to bound the sum $T_{11}$ from \eqref{eq:T11Def}. We first see by Lemma \ref{lmm:BasicSum2} the contribution of terms with $\bb_1\propto \bb_2\Mod{p}$ to $T_{11}$ is
\begin{align*}
\ll \sum_{p\in[P_1,P_2]}\sum_{\cc_1\propto\cc_2\Mod{p}}\Bigl(\frac{B^6}{p^7}+B^{17/3}\Bigr)\ll \frac{B^6}{P_1}+P_2^6 B^{17/3}.
\end{align*}
Thus we have
\[
T_{11}=T_{11}'+O\Bigl(\frac{B^6}{P_1}+P_2^6 B^{17/3}\Bigr),
\]
where $T_{11}'$ counts those terms in $T_{11}$ with $\bb_1\not\propto \bb_2\Mod{p}$, or equivalently with $p\nmid D_{\bb_1,\bb_2}$.

When $\bb_1\not \propto\bb_2\Mod{p}$, we see that the constraints $(\aa\diamond\bb_1)_4=0\Mod{p}$ and $(\aa\diamond\bb_2)_4=0\Mod{p}$ are two linearly independent linear constraints on $\aa\Mod{p}$. In  particular, the index $f_{\bb_1,\bb_2,p,p}=[\Lambda_{\bb_1,\bb_2}:\Lambda_{\bb_1,\bb_2,p,p}]$ simplifies to give
\[
\frac{1}{f_{\bb_1,\bb_2,p,p}}= \frac{\#\{\aa\Mod{p}:\,(\aa\diamond\bb_1)_4=(\aa\diamond\bb_2)_4=\vv\cdot(\aa\diamond \bb_1)=\vv\cdot(\aa\diamond\bb_2)=0\Mod{p}\}}{p^2}.
\]
We separate the above count according to the rank  of the multiplication-by-$\aa$ matrix $M_a\Mod{p}$. Thus
\begin{equation}
\frac{1}{f_{\bb_1,\bb_2,p,p}}=\sum_{i=0}^4\frac{1}{p^2}\widetilde{S}_i(\bb_1,\bb_2),
\label{eq:ftoS}
\end{equation}
where $\widetilde{S}_i(\bb_1,\bb_2)$ counts those $\aa\Mod{p}$ such that $M_\aa$ has rank $i$ and satisfies $(\aa\diamond\bb_1)_4=(\aa\diamond\bb_2)_4=\vv\cdot(\aa\diamond \bb_1)=\vv\cdot(\aa\diamond\bb_2)=0\Mod{p}$.

First we consider $\widetilde{S}_4$. 
%
%
\begin{lmm}\label{lmm:S4t}
\[
\sum_{\cc_1,\cc_2\Mod{p}}\frac{1}{p^2}\widetilde{S}_4(\cc_1,\cc_2)\ll p^6.
\]
\end{lmm}
%
%
\begin{proof}
In this case $M_{\aa}$ has rank 4, and so is invertible $\Mod{p}$. Given any choice of $\cc_1\Mod{p}$ with $p\nmid N(\cc_1)$, we see that $\aa\diamond\cc_1=M_{\cc_1}\aa$ where the multiplication-by-$\cc_1$ matrix $M_{\cc_1}$ has determinant $N(\cc_1)$, and so is invertible $\Mod{p}$. Therefore, given any choice of $\xx\Mod{p}$, there is a unique choice of $\aa\Mod{p}$ with $p\nmid N(\aa)$ such that $\aa\diamond\cc_1\equiv \xx\Mod{p}$. Similarly, since we only consider $\aa$ with $M_\aa$ is invertible, given any choice of $\yy\Mod{p}$ there is then a unique choice of $\cc_2\Mod{p}$ such that $\aa\diamond \cc_2\equiv \yy\Mod{p}$. Since there are $O(p^4)$ choices of $\xx,\yy\Mod{p}$ with $\xx_4=\yy_4=0$ and $\vv\cdot\xx=\vv\cdot\yy=0\Mod{p}$, there are therefore $O(p^4)$ choices of $\aa,\cc_2\Mod{p}$ such that $p\nmid N(\aa)$ and $(\aa\diamond\cc_1)_4=(\aa\diamond\cc_2)_4=\vv\cdot(\aa\diamond\cc_1)=\vv\cdot(\aa\diamond\cc_2)=0\Mod{p}$. Thus we have that
\[
\sum_{\cc_1,\cc_2\Mod{p}}\frac{1}{p^2}\widetilde{S}_4(\cc_1,\cc_2)\ll p^6,
\]
as required.
\end{proof}
Now we consider $\widetilde{S}_2$ and $\widetilde{S}_3$.
\begin{lmm}\label{lmm:S23}
\[
\sum_{\cc_1,\cc_2\Mod{p}}\frac{1}{p^2}\Bigl(\widetilde{S}_2(\cc_1,\cc_2)+\widetilde{S}_3(\cc_1,\cc_2)\Bigr)\ll p^6.
\]
\end{lmm}
%
%
\begin{proof}
 Since $M_\aa$ is not invertible $\Mod{p}$ and has determinant $N(\aa)$, we see that $p|N(\aa)$ and so $p|N(\aa\diamond\cc_1)=N(\aa)N(\cc_1)$. Since $f(x_1,x_2,x_3)$ is an irreducible polynomial which splits into two linear factors over a quadratic extension, and $N(x_1\nu_1+x_2\nu_2+x_3\nu_3)$ is a quartic irreducible polynomial which has no linear factors over any quadratic extension, these polynomials have no common polynomial factors over a mutual splitting field, and so define an algebraic variety of codimension 2. Thus (by Hilbert's Theorem 90 and the Lang-Weil bound) there are $O(p)$ choices of $(x_1,x_2,x_3)\Mod{p}$ such that $f(x_1,x_2,x_3)=N(x_1\nu_1+x_2\nu_2+x_3\nu_3)=0\Mod{p}$. Thus there are $O(p^2)$ choices of $\xx,\yy$ with $p|N(\xx),N(\yy)$ and $x_4=y_4=\vv\cdot\xx=\vv\cdot\yy=0\Mod{p}$. Given $\cc_1$ with $p\nmid N(\cc_1)$ and $\xx$ and $\yy$ as above, here is a unique $\aa\Mod{p}$ such that $\aa\diamond\cc_1\equiv \xx\Mod{p}$, and there are $O(p^2)$ choices of $\cc_2$ such that $\aa\diamond \cc_2\equiv \yy\Mod{p}$ provided $M_\aa$ has rank $2$ or $3$. Putting this together gives the result.
 \end{proof}
\begin{lmm}\label{lmm:S0}
\[
S_0(\cc_1,\cc_2)\ll 1.
\]
\end{lmm}
\begin{proof}
 The only $\aa$ such that $M_\aa$ has rank 0 is the vector $\mathbf{0}\Mod{p}$.
\end{proof}
%
%
Finally, we need to consider the situation where $M_{\aa}$ has rank 1, which is slightly more complicated.
%
%
\begin{lmm}\label{lmm:S1}
\[ \sum_{\cc_1,\cc_2 \Mod{p}}\frac{1}{p^2}\widetilde{S}_1(\cc_1,\cc_2)\ll p^6.
\]
\end{lmm}
%
%
\begin{proof}
If $M_\aa$ has rank 1, then there are $p^3$ choices of $\bb\Mod{p}$ such that $M_\aa \bb=\mathbf{0}\Mod{p}$. On the other hand, let $\af=(a_1\nu_1+a_2\nu_2+a_3\nu_3+a_4\nu_4)$ and $\bfr=(b_1\nu_1+b_2\nu_2+b_3\nu_3+b_4\nu_4)$. If $M_{\aa}\bb=\mathbf{0}\Mod{p}$, then the ideal $\af\bfr$ is a multiple of $(p)$, and so $\bfr$ is a multiple of $(p)/\gcd( \af,(p))$. Therefore for there to be $p^3$ choices of $\bfr\Mod{p}$, $\af$ must be a multiple of $(p)/\pf$ for some degree one prime ideal $\pf$ above $p$. Since there are $O(1)$ degree one prime ideals $\pf$ above $p$ and there are $O(p)$ different multiples of $(p)/\pf$ we see that there are $O(p)$ possible vectors $\aa$ such that $M_\aa$ has rank 1. 

Since the rank is unchanged by replacing $\aa$ with $\lambda\aa$ for any non-zero scalar $\lambda$, we see all such $\aa$ are scalar multiples of one of $O(1)$ choices of vector $\aa^{(0)}$.


Call such a vector $\aa^{(0)}$ `normal' if the constraints $(\aa^{(0)}\diamond \cc_2)_4\equiv\vv\cdot(\aa^{(0)}\diamond \cc_2)\equiv 0\Mod{p}$ are non-trivial on $\cc_2\Mod{p}$, and call $\aa^{(0)}$ `exceptional' if the constraints are trivial on $\cc_2\Mod{p}$. We see that if $\aa^{(0)}$ is normal, then there are $O(p^3)$ choices of $\cc_2\Mod{p}$ and so $O(p^4)$ choices of $(\cc_2,\aa)\Mod{p}$ with $\aa$ a multiple of $\aa^{(0)}$.

We now prove that when $p$ is large enough,  there are no exceptional $\aa^{(0)}$.

If $(\aa^{(0)}\diamond \cc)_4\equiv 0\Mod{p}$ $\forall \cc$, then this equation holds in particular for all $\cc$ in  $\{(1,0,0,0), (0,1,0,0), (0,0,1,0), (0,0,0,1)\} $.
Writing $\aa^{(0)}=(a_1^{(0)}, a_2^{(0)},a_3^{(0)},a_4^{(0)})$ and $\nu_i\nu_j=\sum_{k=1}^4 \lambda_{ijk}\nu_k$, we get 	
$$\sum_{i=1}^4\lambda_{ij4}a_i^{(0)}\equiv 0\Mod p\qquad j=1,2,3,4.$$
This implies that $p|\det (\lambda_{ij4})_{1\le i,j\le 4}$ which is not possible for $p$ large enough if this determinant is non zero.

But this determinant can't be zero, otherwise, there would be $\mu_1,\mu_2,\mu_3,\mu_4$ such that 
\[
\mu_1 \begin{pmatrix} \lambda_{114}\\ \lambda_{214} \\ \lambda_{314}\\ \lambda_{414}\end{pmatrix}
+\mu_2 \begin{pmatrix} \lambda_{124}\\ \lambda_{224} \\ \lambda_{324}\\ \lambda_{424}\end{pmatrix}
+\mu_3 \begin{pmatrix} \lambda_{134}\\ \lambda_{234} \\ \lambda_{334}\\ \lambda_{434}\end{pmatrix}
+\mu_4 \begin{pmatrix} \lambda_{144}\\ \lambda_{244} \\ \lambda_{344}\\ \lambda_{444}\end{pmatrix}=0,
\]
and then the matrix of the multiplication by $\mu_1\nu_1+\mu_2\nu_2+\mu_3\nu_3+\mu_4\nu_4$ wouldn't be invertible. Thus $c_p=0$ for all $p\in [P_1,P_2]$.

Thus, we have that
\[
\frac{1}{p^2}\widetilde{S}_1(\cc_1,\cc_2)=\frac{1}{p}\sum_{\aa^{(0)}\text{ normal}}\1_{\substack{(\aa^{(0)}\diamond \cc_1)_4\equiv\vv\cdot(\aa^{(0)}\diamond \cc_1)\equiv 0\Mod{p}\\ (\aa^{(0)}\diamond \cc_2)_4\equiv\vv\cdot(\aa^{(0)}\diamond \cc_2)\equiv 0\Mod{p}}}+  O\Bigl(\frac{1}{p^2}\Bigr).
\]
However, we have
\[
\sum_{\cc_1,\cc_2\Mod{p}}\frac{1}{p}\sum_{\aa^{(0)}\text{ normal}}\1_{\substack{(\aa^{(0)}\diamond \cc_1)_4\equiv\vv\cdot(\aa^{(0)}\diamond \cc_1)\equiv 0\Mod{p}\\ (\aa^{(0)}\diamond \cc_2)_4\equiv\vv\cdot(\aa^{(0)}\diamond \cc_2)\equiv 0\Mod{p}}}\ll p^5.
\]
This gives the result.
\end{proof}
%
%
We're now in a position to simply our sum.
%
%
\begin{lmm}\label{lmm:T11}
Let
\[
T_{11}':=\sum_{\substack{p\in [P_1,P_2]\\ p\equiv 1\Mod{D_f}}}\sum_{\substack{\bb_1\in \ZZ^4\cap \Cc_1,\bb_2\in \ZZ^4\cap\Cc_2\\  \bb_1\equiv \bb_2\equiv \bb_0\Mod{m} \\ \wedge(\bb_1,\bb_2)\ne 0\\ p\nmid N(\bb_1) N(\bb_2) D_{\bb_1,\bb_2}}}\frac{g_{\bb_1}g_{\bb_2}}{f_{\bb_1,\bb_2,p,p}\det(\Lambda_{\bb_1,\bb_2})}.
\]
Then we have
\[
T_{11}'\ll \frac{B^6}{P_1}+P_2^7 B^{17/3}.
\]
\end{lmm}
%
%
\begin{proof}
 Firstly, by splitting $\bb_1,\bb_2$ into residue classes $\Mod{p}$, we have that
\begin{align*}
T_{11}'
&=\sum_{\substack{p\in [P_1,P_2]\\ p\equiv 1\Mod{D_f}}}\sum_{\substack{\cc_1,\cc_2\Mod{p}\\ \cc_1\not\propto\cc_2\\ N(\cc_1)N(\cc_2)\ne 0\Mod{p} }}\sum_{\substack{\bb_1\in \ZZ^4\cap \Cc_1,\bb_2\in \ZZ^4\cap\Cc_2\\  \bb_1\equiv \bb_2\equiv \bb_0\Mod{m} \\ \wedge(\bb_1,\bb_2)\ne 0\\ \bb_1\equiv \cc_1\Mod{p}\\ \bb_2\equiv \cc_2\Mod{p} }}\frac{g_{\bb_1}g_{\bb_2}}{f_{\bb_1,\bb_2,p,p}\det(\Lambda_{\bb_1,\bb_2})}.
\end{align*}
Using our expression \eqref{eq:ftoS}, we see that this is given by
\begin{align*}
&\sum_{\substack{p\in [P_1,P_2]\\ p\equiv 1\Mod{D_f}}}\sum_{\substack{\cc_1,\cc_2\Mod{p}\\ \cc_1\not\propto\cc_2\\ N(\cc_1)N(\cc_2)\ne 0\Mod{p} }}\sum_{j=0}^4 \frac{\widetilde{S}_j(\cc_1,\cc_2)}{p^2}\sum_{\substack{\bb_1\in \ZZ^4\cap \Cc_1,\bb_2\in \ZZ^4\cap\Cc_2\\  \bb_1\equiv \bb_2\equiv \bb_0\Mod{m} \\ \wedge(\bb_1,\bb_2)\ne 0\\ \bb_1\equiv \cc_1\Mod{p}\\ \bb_2\equiv \cc_2\Mod{p} }}\frac{g_{\bb_1}g_{\bb_2}}{\det(\Lambda_{\bb_1,\bb_2})}.
\end{align*}
Using Lemma \ref{lmm:S1} we get 
\begin{align}
T_{11}'&=O\Bigl(\sum_{\substack{p\in [P_1,P_2]\\ p\equiv 1\Mod{D_f}}}\sum_{\substack{\cc_1,\cc_2\Mod{p}\\ \cc_1\not\propto\cc_2\Mod{p}\\ N(\cc_1)N(\cc_2)\ne 0\Mod{p} }}\frac{T(\cc_1,\cc_2)}{p^2}\sum_{\substack{\bb_1\in \ZZ^4\cap \Cc_1,\bb_2\in \ZZ^4\cap\Cc_2\\  \bb_1\equiv \bb_2\equiv \bb_0\Mod{m} \\ \wedge(\bb_1,\bb_2)\ne 0\\ \bb_1\equiv \cc_1\Mod{p}\\ \bb_2\equiv \cc_2\Mod{p} }}\frac{|g_{\bb_1}g_{\bb_2}|}{\det(\Lambda_{\bb_1,\bb_2})}\Bigr),\label{eq:SeparatedExpression}
\end{align}
where
\[
T(\cc_1,\cc_2):=\widetilde{S}_0(\cc_1,\cc_2)+E_1(\cc_1,\cc_2)+\widetilde{S}_2(\cc_1,\cc_2)+\widetilde{S}_3(\cc_1,\cc_2)+\widetilde{S}_4(\cc_1,\cc_2).
\]
By Lemma \ref{lmm:BasicSum}, we have that
\[
\sum_{\substack{\bb_1\in \ZZ^4\cap \Cc_1,\bb_2\in \ZZ^4\cap\Cc_2\\  \bb_1\equiv \bb_2\equiv \bb_0\Mod{m} \\ \wedge(\bb_1,\bb_2)\ne 0\\ \bb_1\equiv \cc_1\Mod{p}\\ \bb_2\equiv \cc_2\Mod{p} }}\frac{|g_{\bb_1}g_{\bb_2}|}{\det(\Lambda_{\bb_1,\bb_2})}\ll \frac{B^6}{p^8}+B^{17/3}.
\]
Lemmas \ref{lmm:S0}, \ref{lmm:S1}, \ref{lmm:S23}, \ref{lmm:S4t} show that
\[
\sum_{\cc_1,\cc_2\Mod{p}}\frac{T(\cc_1,\cc_2)}{p^2}\ll p^6.
\]
Thus we see that the  term $T_{11}'$
\eqref{eq:SeparatedExpression} is 
\begin{align*}
\ll \sum_{\substack{p\in [P_1,P_2]\\ p\equiv 1\Mod{D_f}}}p^6\Bigl(\frac{B^6}{p^8}+B^{17/3}\Bigr)\ll \frac{B^6}{P_1}+B^{17/3}P_2^7.
\end{align*}
This ends the proof of Lemma \ref{lmm:T11}. \end{proof}

%
%
Putting everything in this section together, we are left to show that
\[
T_{12}\ll \eta_2^5 B^6.
\]
%
%
\subsection{The case $p_1\ne p_2$}
%
%
In this section we bound the sum $T_{12}$ given by \eqref{eq:T12Def}.
%
%
\begin{lmm}\label{lmm:T12}
We have 
\[
T_{12}\ll |S_{sep}|+\frac{B^6}{P_1}+P_2^7 B^{17/3},
\]
where, $S_{sep}$ is given by
\[
S_{sep}:=\sum_{\bb_1\in\ZZ^4\cap\Cc_1}\sum_{\bb_2\in\ZZ^4\cap\Cc_2}\frac{g_{\bb_1}g_{\bb_2}D_{\bb_1,\bb_2}}{\wedge(\bb_1,\bb_2)}.
\]
\end{lmm}
%
%
\begin{proof}
We wish to reintroduce terms with $p_1\nmid N(\bb_1)$ and $p_2\nmid N(\bb_2)$ so that the inner sum is independent of $p_1,p_2$. There are $O(p_1^3)$ choices of $\cc_1\Mod{p_1}$ such that $p_1|N(\cc_1)$. Thus, by Lemma \ref{lmm:BasicSum}, we see that the terms with $p_1|N(\bb_1)$ contribute a total
\begin{align*}
&\sum_{\substack{p_1,p_2\in [P_1,P_2]\\ p_1\equiv p_2\equiv 1\Mod{D_f}}}\frac{1}{p_1 p_2}\sum_{\substack{\bb_1\in \ZZ^4\cap \Cc_1,\bb_2\in \ZZ^4\cap\Cc_2\\  \bb_1\equiv \bb_2\equiv \bb_0\Mod{m} \\ \wedge(\bb_1,\bb_2)\ne 0\\ p_1| N(\bb_1)}}\frac{|g_{\bb_1}g_{\bb_2}|D_{\bb_1,\bb_2}}{\wedge(\bb_1,\bb_2)}\\
&\ll \sum_{\substack{p_1,p_2\in [P_1,P_2]\\ p_1\equiv p_2\equiv 1\Mod{D_f}}}\frac{1}{p_1 p_2}\sum_{\substack{\cc_1,\cc_2\Mod{p_1}\\ p_1|N(\cc_1)}}\sum_{\substack{\bb_1\in \ZZ^4\cap \Cc_1,\bb_2\in \ZZ^4\cap\Cc_2\\  \wedge(\bb_1,\bb_2)\ne 0\\ \bb_1\equiv \cc_1\Mod{p_1} \\ \bb_2\equiv \cc_2\Mod{p_2} )}}\frac{|g_{\bb_1}g_{\bb_2}|D_{\bb_1,\bb_2}}{\wedge(\bb_1,\bb_2)}\\
&\ll \sum_{\substack{p_1,p_2\in [P_1,P_2]\\ p_1\equiv p_2\equiv 1\Mod{D_f}}}\frac{1}{p_1 p_2}p_1^7\Bigl(\frac{B^6}{p_1^8}+B^{17/3}\Bigr)\\
&\ll \frac{B^6}{P_1}+P_2^7 B^{17/3}.
\end{align*}
Similarly, we see that terms $p_2|N(\bb_2)$ contribute a total $O(B^6/P_1+P_2^7 B^{17/3})$. Thus we find that
\begin{align*}
T_{12}&=\Bigl(\sum_{\substack{p_1,p_2\in [P_1,P_2]\\ p_1\equiv p_2\equiv 1\Mod{D_f}}}\frac{1}{p_1 p_2}\Bigr)\Bigl(\sum_{\substack{\bb_1\in \ZZ^4\cap \Cc_1,\bb_2\in \ZZ^4\cap\Cc_2\\  \bb_1\equiv \bb_2\equiv \bb_0\Mod{m} \\ \wedge(\bb_1,\bb_2)\ne 0 }}\frac{g_{\bb_1}g_{\bb_2}D_{\bb_1,\bb_2}}{\wedge(\bb_1,\bb_2)}\Bigr)\\
&\qquad\qquad+O\Bigl(\frac{B^6}{P_1}+P_2^7 B^{17/3}\Bigr).
\end{align*}
Noting that the sum over $p_1,p_2$ is $O(1)$, this gives the result.
\end{proof}
%
%
Thus it remains to bound $S_{sep}$.
%
%
\subsection{Reduction to small residue classes and small boxes}
%
%
We first show that the contribution to $S_{sep}$ from terms with $D_{\bb_1,\bb_2}>(\log{B})^K$ or $\wedge(\bb_1,\bb_2)\le B^2/(\log{B})^K$ is negligible.
%
%
\begin{lmm}
\begin{align*}
\sum_{\substack{\bb_1\in \ZZ^4\cap \Cc_1,\,\bb_2\in\ZZ^4\cap\Cc_2\\ \wedge(\bb_1,\bb_2)>0 \\ \max(B^2/\wedge(\bb_1,\bb_2),D_{\bb_1,\bb_2})>(\log{B})^K}}\frac{|g_{\bb_1}g_{\bb_2}|}{\det(\Lambda_{\bb_1,\bb_2})}\ll_K\frac{B^6}{(\log{B})^K}.
\end{align*}
\end{lmm}
%
%
\begin{proof}
This is similar to the proof of Lemma \ref{lmm:BasicSum}. Indeed, the argument in the proof of Lemma \ref{lmm:BasicSum} shows that the contribution from terms with $\wedge(\bb_1,\bb_2)\ll B^{4/3}$ is $O(B^{17/3})$, and the contribution from terms with $\wedge(\bb_1,\bb_2)\sim B^2/C$ (for $C=2^j\ll B^{2/3}$) and $D_{\bb_1,\bb_2}=d$ is
\[
\ll \frac{dC}{B^2}\Bigl(B+\frac{B^4}{C^3d^3}\Bigr).
\]
Thus we see that the total contribution is
\[
\ll B^{17/3}+\sum_{C=2^j\ll B}\sum_{\substack{d\le B\\ \max(d,C)>(\log{B})^K}}\frac{dC}{B^2}\Bigl(B+\frac{B^4}{C^3d^3}\Bigr)\ll_K \frac{B^6}{(\log{B})^K}.\qedhere
\]
\end{proof}
%
%
Thus we just need to consider $D_{\bb_1,\bb_2}\le (\log{x})^K$ and $\wedge(\bb_1,\bb_2)\ge B^2/(\log{x})^K$. 
%
%
\begin{lmm}\label{lmm:Ssep}
Imagine that for every cube $\Cc\subseteq[1,B]^4$, every and any $\cc\Mod{d}$ and every $K>0$ we have
\[
\sum_{\substack{\bb\in \ZZ^4\cap \Cc \\ \bb\equiv \cc\Mod{d}}}g_\bb \ll_K\frac{B^4}{(\log{B})^{10K}}.
\]
Then for every choice of $K>0$ we have that
\[
S_{sep}\ll_K \frac{B^6}{(\log{B})^K}
\]
\end{lmm}
%
%
\begin{proof}
Since $\wedge(\bb_1,\bb_2)$ is continuous in $\bb_1,\bb_2$ we see that if a pair of cubes $\Cc_1',\Cc_2'$ of side length $B/(\log{x})^{2K}$ contains a point with $\wedge(\bb_1,\bb_2)\ge B^2/(\log{x})^K$, then in fact for all $\bb_1'\in\Cc'_1$ and $\bb_2\in\Cc'_2$ we have $\wedge(\bb_1',\bb_2')=\wedge(\bb_1,\bb_2)(1+O(\log{x})^{-K})$. Thus we may replace $\wedge(\bb_1,\bb_2)$ with 
\[
\wedge(\Cc_1',\Cc_2'):=\sup_{\xx\in\Cc_1',\yy\in\Cc_2'}\wedge(\xx,\yy)
\]
at the cost of an error term of size $B^6/(\log{x})^K$. Thus we have
\[
S_{sep}\ll \frac{B^6}{(\log{x})^K}+\frac{(\log{x})^{9K}}{B^2} \sum_{d\le (\log{x})^K}d\sup_{\Cc_1',\Cc_2'}\sum_{\substack{\bb_1\in\ZZ^4\cap\Cc_1',\,\bb_2\in\ZZ^4\cap\Cc'_2\\ D_{\bb_1,\bb_2}=d}}g_{\bb_1}g_{\bb_2}.
\]
Now we wish to simplify the condition $D_{\bb_1,\bb_2}=d$ to a congruence condition, which will finally allow us to separate the variables $\bb_1,\bb_2$. By Moebius inversion we have
\begin{align*}
\mathbf{1}_{D_{\bb_1,\bb_2}=d}&=\sum_{e|D_{\bb_1,\bb_2}/d}\mu(e)\\
&=\sum_{e\le (\log{x})^{30K} }\mu(e)\mathbf{1}_{\bb_1\propto\bb_2\Mod{de}}+O((\log{x})^{30K}\mathbf{1}_{D_{\bb_1,\bb_2}\ge (\log{x})^{30K}}).
\end{align*}
By Lemma \ref{lmm:Cylinder}, the contribution of the second term to $S_{sep}$ is $O(B^4/(\log{x})^{K})$. Thus we see that
\begin{align}
S_{sep}\ll \frac{B^6}{(\log{x})^A}+(\log{x})^{50K}\sup_{\substack{\Cc_1',\Cc_2'\\ de\ll (\log{B})^{31K}}}| S_{sep}'|
\end{align}
where
\begin{align*}
S_{sep}'&:=\sum_{\bb,\lambda_1,\lambda_2\Mod{de}}\Biggl(\sum_{\substack{\bb_1\in\ZZ^4\cap\Cc_1',\\ \bb_1\equiv \lambda_1\bb\Mod{de}\\ \bb_1\equiv \bb_0\Mod{m}}}g_{\bb_1}\Biggr)\Biggl(\sum_{\substack{\bb_2\in\ZZ^4\cap\Cc_2',\\ \bb_2\equiv \lambda_2\bb\Mod{de}\\ \bb_2\equiv \bb_0\Mod{m}}}g_{\bb_2}\Biggr).
\end{align*}
By assumption of the lemma, we have that
\[
\sum_{\substack{\bb_1\in\ZZ^4\cap\Cc_1',\\ \bb_1\equiv \lambda_1\bb\Mod{de}\\ \bb_1\equiv \bb_0\Mod{m}}}g_{\bb_1}\ll_K \frac{B^4}{(\log{x})^{100K}}.
\]
Substituting this in then gives the final bound.
\end{proof}
%
%
Thus we see that it is sufficient to obtain a suitable bound for $g_{\bb}$ on average over hypercubes in residue classes.
%
%
\subsection{Localised bound and Proof of Proposition \ref{prpstn:T1}}
%
%

To finish our proof we need to show that we have a suitable estimate for $g_{\bb}=\1_{\Rc}(\bb)-\widetilde{\1}_{\Rc}(\bb)$ over $\bb$ restricted to small boxes and arithmetic progressions.  We don't require estimates arithmetic progressions to moduli larger than $(\log{X})^{O(1)}$, and there are no issues caused by a possible Siegel zero.

\begin{prpstn}\label{prpstn:Local}
For every $K>0$ and every polytope $\mathcal{R}$ under consideration, we have
\[
\sum_{\substack{\bb\in \ZZ^4\cap \Cc \\ \bb\equiv \cc\Mod{d}}}\Bigl(\mathbf{1}_{\mathcal{R}}(\bb)-\widetilde{\mathbf{1}}_{\mathcal{R}}(\bb)\Bigr) \ll_K\frac{B^4}{(\log{B})^{10K}}.
\]
\end{prpstn}
\begin{proof}
This is the equivalent of \cite[Proposition 9.7]{May15b}, and the proof works in exactly the same manner for our situation. Therefore we only highlight a couple of main details.

First we estimate the contribution from $\1_{\Rc}(\bb)$. Since $\bb$ is in a small cube, no two elements can generate the same ideal, and so we can write the sum as a sum of principal ideals. We can use Hecke Grossencharacters to detect the congruence conditions and the restriction of $\bb$ to the cube $\mathcal{C}$. The Prime Number Theorem for Grossencharacters then allows one to suitably estimate the resulting sums over $\1_{\Rc}(\bb)$, giving an explicit main term and an error term which is $O_K(B^4/(\log{B})^{10K})$. This is essentially the same argument as \cite[Lemmas 9.1-9.4]{May15b}. 

The contribution from $\widetilde{\1}_{\Rc}(\bb)$ can be estimated by swapping the order of summation in the sieve sum and using the fact that $\bb\in\mathbb{Z}\cap\mathcal{C}$ are equidistributed in suitable aritmetic progressions as in \cite[Lemmas 9.5 and 9.6]{May15b}. This gives a main term and a error term $O_K(B^4/(\log{B})^{10K})$.

The main term contributions from $\1_{\Rc}(\bb)$ and $\widetilde{\1}_{\Rc}(\bb)$ are the same apart from opposite signs and so cancel, giving the result.
\end{proof}
Finally, we are able to complete our proof of Proposition \ref{prpstn:T1}.
%
%
\begin{proof}[Proof of Proposition \ref{prpstn:T1}]
Putting together the equations \eqref{eq:T1}, \eqref{eq:Tc} and the argument of Section \ref{sec:Cosmetic},  we find that provided $B<X^{3/4-\epsilon}/P_2^{3/2}$ (from \eqref{eq:BSize}) we have
\begin{align*}
T_1(\cR)&=\sum_{\cC\in Cl_K}\sum_{\substack{{\a_0,\b_0 \Mod{m'}}\\ {N(\gc )(\a_0\diamond \b_0)_i\equiv (\v_0)_i\Mod {m'}}}} T_3(\cR)+O(X^{3+\epsilon}/P_1),
\end{align*}
where $T_3$ is given by \eqref{eq:T3Def}.

Putting together \eqref{eq:T3}, \eqref{eq:T4}, \eqref{eq:T6} and Lemmas \ref{lmm:T5} \ref{lmm:T7}, \ref{lmm:T8}, \ref{lmm:T11}, \ref{lmm:T12}, \ref{lmm:Ssep}, and Proposition \ref{prpstn:Local} then gives
\begin{align*}
T_3(\cR)^2&\ll_K \eta_2^{-5}A^4\Bigl(A^3B^3+AB^7P_2^2+A^2B^{17/3}P_2^7+\frac{A^2B^6}{(\log{X})^K}\Bigr).
\end{align*}
Since $|\cAI(X_0)|\asymp A^3B^3$, this gives the result provided
\begin{align*}
A<B^{3-\epsilon},\qquad
BP_2^2<A^{1-\epsilon},\qquad
P_2^{21}<B^{1-\epsilon},.
\end{align*}
(Here we used that the second inequality implies \eqref{eq:BSize}.) After taking $\epsilon$ suitably small, we see that the first condition is implied by the first inequality of \eqref{eq:th5}, whereas the final two inequalities are implied by the assumption $\tau'\le \min(4-2\theta_1'-\cdots-2\theta'_{\ell'},\theta_1+\dots+\theta_{\ell'}-1)/100$. This gives Proposition \ref{prpstn:T1}.
\end{proof}
%
%

This completes the proof of Proposition \ref{prpstn:T1}, and hence Theorem \ref{DistriNorm} and Theorem \ref{cP}.

\bibliographystyle{plain} 
\bibliography{bibC4D4}

\begin{thebibliography}{10}

\bibitem{B15}
R.~de~la Bret\`eche.
\newblock Plus grand facteur premier de valeurs de polyn\^omes aux entiers.
\newblock {\em Acta Arith.}, 169:221--250, 2015.

\bibitem{BD20}
R.~de~la Bret\`eche and S.~Drappeau.
\newblock Niveau de répartition des polynômes quadratiques et crible majorant
  pour les entiers friables.
\newblock {\em J. Eur. Math. Soc.}, 22(5):1577--1624, 2020.

\bibitem{C}
K.~Conrad.
\newblock Galois groups of cubics and quartics (not in characteristic $2$).
\newblock {\em
  https://kconrad.math.uconn.edu/blurbs/galoistheory/cubicquartic.pdf}.

\bibitem{D15}
C.~Dartyge.
\newblock Le probl\`eme de {T}ch\'ebychev pour le douzi\`eme polyn\^ome
  cyclotomique.
\newblock {\em Proc. London Math. Soc.}, 111(1):1--62, 2015.

\bibitem{DI82}
J.-M. Deshouillers and H.~Iwaniec.
\newblock On the greatest prime factor of $n^2+1$.
\newblock {\em Ann. Inst. Fourier (Grenoble)}, 32(4):1--11, 1982.

\bibitem{Er52}
P.~Erd\H{o}s.
\newblock On the greatest prime factor of $\prod_{k=1}^x f(k)$.
\newblock {\em J. London Math. Soc.}, 27:379--384, 1952.

\bibitem{HB84}
D.~R. Heath-Brown.
\newblock Diophantine approximation with square-free numbers.
\newblock {\em Math. Z.}, 187(3):335--344, 1984.

\bibitem{HB01}
D.~R. Heath-Brown.
\newblock The largest prime factor of $x^3+2$.
\newblock {\em Proc. London Math. Soc.}, 82(3):554--596, 2001.

\bibitem{Hooley}
C.~Hooley.
\newblock On the greatest prime factor of a quadratic polynomial.
\newblock {\em Acta Math.}, 281-299:21--50, 1967.

\bibitem{Ir15}
A.~J. Irving.
\newblock The largest prime factor of $x^3+2$.
\newblock {\em Acta Arith.}, 171(1):67--80, 2015.

\bibitem{I80b}
H.~Iwaniec.
\newblock A new form of the error term in the linear sieve.
\newblock {\em Acta Arith.}, 37:307--320, 1980.

\bibitem{I80a}
H.~Iwaniec.
\newblock Rosser's sieve.
\newblock {\em Acta Arith.}, 36:171--202, 1980.

\bibitem{JLY02}
C.~U. Jensen, A.~Ledet, and N.~Yui.
\newblock {\em Generic Polynomials: Constructive Aspects of the Inverse Galois
  Problem}, volume~45.
\newblock Math. Sci. Res. Inst. Publ., Cambridge Univ. Press, Cambridge, 2002.

\bibitem{May15b}
J.~Maynard.
\newblock Primes represented by incompletes norm forms.
\newblock {\em Forum of Mathematics, Pi}, 8(3):1--128, 2020.

\bibitem{Merikoski}
J.~Merikoski.
\newblock On the largest prime factor of $n^2 +1$.
\newblock {\em https://arxiv.org/abs/1908.08816v3.pdf}, 2021.

\bibitem{Nark}
W.~Narkiewicz.
\newblock {\em Elementary and analytic theory of algebraic numbers}, volume
  Third Edition.
\newblock Springer, 712 pp., 2004.

\bibitem{SchinzelSierpinski}
A.~Schinzel and W.~Sierpi\'nski.
\newblock Sur certaines hypoth\`eses concernant les nombres premiers.
\newblock {\em Acta Arithmetica}, 4(3):185--208, 1958.

\bibitem{Tenenbaum}
G.~Tenenbaum.
\newblock Sur une question d'{E}rd{\H{o}}s et {S}chinzel. {II}.
\newblock {\em Invent. Math.}, 99:215--224, 1990.

\bibitem{Web99}
H.~Weber.
\newblock {\em Lehrbuch der Algebra}, volume~2.
\newblock Vieweg and Sohn, 592 pp., 1899.

\end{thebibliography}

\vskip 2cm

\noindent Cécile Dartyge, Institut \'Elie Cartan, Université de Lorraine, BP 70239, 54506 Vand\oe uvre-lès-Nancy Cedex, France

\noindent cecile.dartyge@univ-lorraine.fr

\hfill\\

\noindent James Maynard, Mathematical Institute, Woodstock Road, Oxford OX2 6GG, UK

\noindent james.alexander.maynard@gmail.com
\end{document}